\documentclass[12pt]{article}

\usepackage{amsthm,amssymb, amsthm, amscd, graphicx, bm, comment}
\usepackage[leqno]{amsmath}
\usepackage{tikz, dsfont}
\usepackage{tikz-cd} 

\usepackage{authblk}

\usetikzlibrary{arrows}
\tikzset{>=latex}
\tikzset{inner sep=3pt}

\usepackage{hyperref}  
\hypersetup{%
	bookmarksnumbered=true,%
	bookmarks=true,%
	colorlinks=true,%
	linkcolor=blue,%
	citecolor=blue,%
	filecolor=blue,%
	menucolor=blue,%
	pagecolor=blue,%
	urlcolor=blue,%
	pdfnewwindow=true,%
	pdfstartview=FitBH}

\usepackage[left=2.90cm, right=2.90cm, top=3.40cm, bottom=3.50cm]{geometry}

\theoremstyle{plain}      

\newtheorem{thm}{Theorem}[section]
\newtheorem{lem}[thm]{Lemma}
\newtheorem{prop}[thm]{Proposition}
\newtheorem{cor}[thm]{Corollary}

\newtheorem{conj}[thm]{Conjecture}

\newtheorem{defn}[thm]{Definition}
\newtheorem{ex}[thm]{Example}

\newtheorem{rmk}[thm]{Remark}

\newtheorem{example}[thm]{Example}

\DeclareMathOperator{\im}{\mathsf{Im}}

\DeclareMathOperator{\Fuk}{\mathsf{Fuk}}
\DeclareMathOperator{\MF}{\mathsf{MF}}
\DeclareMathOperator{\tw}{\mathsf{tw}}

\newcommand{\Z}{\mathbb{Z}}

\def\bbK{\mathbb{K}}

\newcommand{\CC}{\mathcal{C}}
\newcommand{\cA}{\mathcal{A}}
\newcommand{\DD}{\mathcal{D}}
\newcommand{\PP}{\mathcal{P}}
\newcommand{\QQ}{\mathcal{Q}}

\newcommand{\OC}{{\mathcal OC}}

\newcommand{\m}{\mathfrak{m}}
\newcommand{\fl}{\mathfrak{l}}
\newcommand{\n}{\mathfrak{n}}
\def\Ai{$A_\infty$}
\newcommand{\id}{{\sf id}}
\newcommand{\sh}{{\sf sh}}
\newcommand{\Cl}{{\sf Cl}}
\newcommand{\GG}{{\mathbb{G}}}
\newcommand{\wt}{{\sf wt}}
\newcommand{\Aut}{{\sf Aut}}
\newcommand{\fr}{{\sf fr}}
\newcommand{\sym}{{\sf sym}}
\newcommand{\Sym}{{\sf Sym}}
\newcommand{\sgn}{{\sf sgn}}
\newcommand{\ra}{{\rightarrow}}
\newcommand{\one}{{\mathds{1}}}

\begin{document}
	
\title{On the Morita invariance of Categorical Enumerative Invariants}

\author{Lino Amorim and Junwu Tu}

\date{}

	\maketitle
	
	\begin{abstract}
		Categorical Enumerative Invariants (CEI) are invariants associated with a unital, cyclic, smooth $A_\infty$-category and a splitting of its non-commutative Hodge filtration. In this paper, we extend the definition of CEI to Calabi-Yau $A_\infty$-categories with a splitting. Moreover, we formulate and prove the Morita invariance of CEI. As part of our proof, we develop tools to construct unital and cyclic models for Calabi-Yau categories. In particular, we prove a unital version of Kontsevich-Soibelman's Darboux theorem. 
        
      As an application, we compute CEI in some new examples. Also, when applied to derived categories of coherent sheaves, our results yield new invariants of smooth, proper Calabi-Yau 3-folds.
	\end{abstract}

\section{Introduction}

Let $\mathbb{K}$ be a field of characteristic zero and let $A$ be a $\Z/2\Z$-graded, cyclic $A_\infty$-algebra over $\mathbb{K}$. Assume that $A$ is 
\begin{quote}
{($\dagger$)} smooth, finite dimensional, unital, and satisfies the Hodge-to-de-Rham degeneration property. (The Hodge-de Rham degeneration property is automatic if $A$ is
$\Z$-graded~\cite{Kal}.)
\end{quote}
Assume that the cyclic pairing is of parity $d\in \Z/2\Z$.  Let $s$ be a choice of splitting of the non-commutative Hodge filtration, (see Definition~\ref{defi:splitting}). Associated to the pair $(A,s)$, we obtain from~\cite{Cos2,CT} a set of $\mathbb{K}$-valued enumerative invariants, with insertions in the Hochschild homology of $A$. Given integers $g\geq0$, $n\geq1$, satisfying $2g-2+n>0$ one constructs:
\begin{align*}
 \langle \alpha_1 u^{k_1},\ldots, \alpha_n u^{k_n}\rangle_{g,n}^{A,s} & \in \mathbb{K},\\
 \textrm{for any collection} \ \ \alpha_1,\ldots,\alpha_n &\in HH_*(A)[d], \; k_1,\ldots,k_n\in \mathbb{N}.
\end{align*}
These invariants are called {\sl Categorical Enumerative Invariants (CEI for short)} in~\cite{CT}. 

The main motivation to study these invariants goes back to Kontsevich's proposal in \cite{Kon}, that the enumerative predictions of Mirror Symmetry \cite{COGP, BCOV} should follow from Homological Mirror Symmetry. More precisely, on the A-side of Mirror Symmetry, he proposes one can recover the Gromov-Witten invariants of a symplectic manifold $X$ from its Fukaya category ${\sf Fuk}(X)$. Similarly on the B-side, by the same process, one can recover the period integral of a Calabi-Yau variety $Y$ from the category of coherent sheaves ${\sf Coh}(Y)$. Assuming this process of recovering invariants from the categories is Morita invariant, we could then automatically deduce Enumerative Mirror Symmetry from Homological Mirror Symmetry. The CEI from \cite{CT}, conjecturally, provide the mechanism to extract enumerative invariants from categories, predicted by Kontsevich (including the higher genus invariants). Moreover, when applied to appropriate categories/algebras the CEI are expected to recover many other enumerative invariants in the literature, such as FJRW invariants and BCOV invariants. We refer to \cite{CT} for a more detailed list of invariants and the corresponding categories.

\vspace{.2cm}

The goal of this paper is to extend the definition of CEI to \emph{Calabi-Yau} \Ai-categories and then formulate and prove their Morita invariance. 

The main difficulty lies in the very formulation of Morita invariance, since cyclic structures don't even pull-back via $A_\infty$ quasi-isomorphisms. For this reason, we need to replace cyclic structures with something more flexible, namely Calabi-Yau structures. We would like to note that there are two {\sl a priori} different definitions in the literature: proper Calabi-Yau structures~\cite{KS} and smooth Calabi-Yau structures~\cite{KTV} (sometimes also referred to as right and left). However, as proved in \cite{GPS} these two definitions are equivalent under the assumption ($\dagger$), so we will not distinguish them.

\subsection{Main results} 
Let $\CC$ be a {\sl small} $\Z/2\Z$-graded $A_\infty$-category  over $\mathbb{K}$. We assume that $\CC$ is
\begin{quote}
{($\dagger\dagger$)} smooth, proper, unital, and satisfies the Hodge-to-de-Rham degeneration property. 
\end{quote}
Notice that we no longer assume cyclicity, and we have replaced the finite-dimensional condition in ($\dagger$) by the properness condition, that is, the $\hom$ spaces in the category have finite dimensional cohomology. 

Recall from \cite{GPS}, an element $\omega\in HH_\bullet(\CC)$ is called a weak Calabi-Yau structure if its induced pairing on $\CC$ defined by $a\otimes b \mapsto \langle \m_2(a,b), \omega\rangle_{\sf Muk}$ is homologically non-degenerate, see Definition~\ref{def:cy}. Fix $d\in \Z/2\Z$, termed parity, and consider pairs $(\omega, s) $ where 
\begin{itemize}
\item[--] $\omega$ is a weak Calabi-Yau structure of $\CC$ of degree $d$, 
\item[--] $s$ is a splitting of the non-commutative Hodge filtration of $\CC$. 
\end{itemize}
We shall refer to a pair $(\omega, s)$ as an extended Calabi-Yau structure. Next, we formulate a unitality condition for the pair $(\omega,s)$. We denote by  ${\tw}^\pi\CC$ the triangulated split-closure of $\CC$. It is well known that the embedding $\iota: \CC \hookrightarrow {\tw}^\pi\CC$ induces isomorphisms in Hochschild and cyclic homologies, and hence we can identify extended Calabi-Yau structures in $\CC$ with those in ${\tw}^\pi\CC$. In the following, we shall make use of this identification freely. An extended Calabi-Yau structure $(\omega,s)$ is called unital if there exists a split-generating subcategory $\mathcal{A}\subset \tw^\pi \CC$ such that
\[  \langle \one_X , s(\omega)\rangle_{{\sf hres}} \in \mathbb{K}, \;\;\;\forall X\in \mathcal{A},\]
where $\one_X$ in the identity morphism of an object $X$ in the category $\mathcal{A}$, and $\langle -, - \rangle_{{\sf hres}}$ denotes the higher residue pairing (see Subsection~\ref{sec:trivCY}). Denote by $\mathcal{M}^d_\CC$ the set of unital extended Calabi-Yau structures of parity $d$.

Our first result is that a unital pair $(\omega,s)\in \mathcal{M}^d_\CC$, determines a cyclic model for the split-generating subcategory $\mathcal{A}\subset \tw^\pi \CC$. More precisely, there is a minimal, cyclic, unital \Ai-category $\mathcal{A}'$ and a \Ai \ quasi-isomorphism
\[\mathcal{A}' \to \mathcal{A}.\]
Without requiring unitality, a version of this result was first proved by Kontsevich-Soibelman \cite{KS} for \Ai-algebras. The unitality requirement creates new phenomena. To deal with this we introduce a new ``unital version" of cyclic homology (see Section 3.1) which might be of independent interest.

Via the quasi-isomorphism $\mathcal{A}'\cong \mathcal{A}$, $s$ determines a splitting of $\mathcal{A}'$ (still denoted by $s$). We can then apply the CEI construction from \cite{CT} to $(\mathcal{A}',s)$. This defines a function
\begin{equation*}
 \langle \ldots \rangle_{g,n}^{\mathcal{A}',s}:  HH_\bullet(\CC)[d][[u]]^{\otimes n} \ra  \mathbb{K}
 \end{equation*}
for each pair of integers $(g,n)$ such that $2-2g-n<0$. Our second main result is that this function is independent of the choice of $\mathcal{A}$ and its cyclic unital model $\mathcal{A}'$ and therefore we have a well-defined function
\begin{equation}\label{eq:intro-formulation}
 F_{g,n}^{\CC}: \mathcal{M}^d_\CC \times  HH_\bullet(\CC)[d][[u]]^{\otimes n} \ra  \mathbb{K}.
\end{equation}
See Section \ref{sec:definition} for details.

In order to formulate the Morita invariance of this construction first recall \cite{She} that two \Ai-categories $\CC$ and $\DD$ are Morita equivalent if and only if ${\tw}^\pi\CC$ and ${\tw}^\pi\DD$ are quasi-equivalent.

\begin{thm}~\label{thm:intro}
Let $\mathcal{C}$ and $\mathcal{D}$ be two $\Z/2\Z$-graded $A_\infty$-category over a field $\mathbb{K}$ of characteristic zero. Assume that they are both small categories and satisfy Condition ($\dagger\dagger$).  Assume $\CC$ and $\DD$ are Morita equivalent through a quasi-equivalence $f: {\tw}^\pi\CC\to{\tw}^\pi\DD$. Then the naturally defined push-forward map
\[ (f)_*: \mathcal{M}^d_\CC \times  HH_\bullet(\CC)[d][[u]]^{\otimes n} \ra \mathcal{M}_\DD^d \times  HH_\bullet(\DD)[d][[u]]^{\otimes n}, \;\; \forall d\in \Z/2\Z,\]
intertwines CEI, i.e. we have
\[ F^\CC_{g,n} = F^\DD_{g,n} \circ (f)_*, \;\; \forall (g,n), \;\; 2-2g-n<0.\]
\end{thm}

This theorem is an important step towards showing that CEI provide a solution to Kontsevich's original proposal~\cite{Kon}, thus helping unveil some of the mysteries of mirror symmetry. From our point of view, the most important step is then comparing the CEI with the \emph{geometric invariants.} For Fukaya categories, we propose a precise form of this comparison between CEI and Gromov-Witten invariants, see Conjecture~\ref{conj:a-model}. 

In Section~\ref{sec:ex-app} we give a few applications of the above theorem. Namely, we compute the CEI of the category of matrix factorizations of a polynomial with a non-degenerate critical point.

In the case of derived categories of coherent sheaves, using a canonical splitting (the ``Blanc-To\"en" splitting) constructed by Blanc~\cite{Bla}, we argue the resulting CEI are in fact birational invariants when the complex dimension is less or equal to three. This yields new birational invariants of smooth and proper Calabi-Yau $3$-folds. Remarkably, these birational invariants seem to also appear as partition functions of topological string theory~\cite{BCOV}, as pointed out by Costello~\cite{Cos2}. 

\subsection{Strategy of the proof}

A unital pair $(\omega, s)\in \mathcal{M}^d_\CC$, determines a unital Calabi-Yau structure $s(\omega)$ (see Definition~\ref{def:cy}) of parity $d$. To define the map $F^\CC_{g,n}$, we first choose a split-generating subcategory $\mathcal{A}\subset \tw^\pi \CC$ from the definition of the unitality condition. Then we prove the existence of a unital, cyclic model of the Calabi-Yau category $\big(\mathcal{A},s(\omega)\big)$ given by a unital $A_\infty$ quasi-isomorphism
\[ f: \mathcal{A}' \to \mathcal{A}.\]
This is proved in Proposition~\ref{prop:cyclic-even} and Proposition~\ref{prop:cyclic-odd}. The proof has two steps, first we show a Calabi-Yau structure is equivalent to a \emph{strong homotopy inner product} \cite{Cho1} and then show this can be made into a cyclic structure by a Darboux-like theorem. The unital version, follows the same strategy. However $s(\omega)$ might not determine a \emph{unital Calabi-Yau structure}, when $d$ is even, which is why we need to restrict to unital splittings to ensure that $s(\omega)$ lifts to a unital Calabi-Yau structure.

 In the even case, the unital cyclic model $\mathcal{A}'$ is then unique up to unital, cyclic $A_\infty$-functor, see Proposition~\ref{prop:cyclic-even}. In contrast, in the odd case, we always have existence of a unital, cyclic model, since $s(\omega)$ always lifts to a unital Calabi-Yau structure. However the uniqueness of $\mathcal{A}'$ fails to hold - see Example~\ref{ex:non-existence}. 

We then apply the CEI construction to $\mathcal{A}'$ in order to define the value of $F^\CC_{g,n}$ on the pair $(\omega,s)\in \mathcal{M}_\CC^d$. The main difficulty then is to prove that the definition of $F^\CC_{g,n}$ is independent of the choice of the cyclic model $\mathcal{A}'$. We first show that CEI are invariant under cyclic \Ai-isomorphism. In order to do this, we prove that a cyclic \Ai-isomorphism is equivalent to a \emph{cyclic pseudo-isotopy} - a notion introduced by Fukaya \cite{Fuk}. And then show that CEI can be defined in families which allows us to prove that they are invariant under pseudo-isotopies. This is accomplished in Theorem~\ref{thm:cyclic-inv}. In the even case, this is enough to show that $F^\CC_{g,n}$ are independent of the model, since $\mathcal{A}'$ is unique up to unital, cyclic isomorphism.  In the odd case the cyclic model is not unique. To get around this problem, in Appendix~\ref{app:b-f} we prove that CEI are invariant under tensoring with the Clifford algebra $\Cl=\mathbb{K}[\epsilon]$. This property of invariants is known as the Boson-Fermion correspondence in the physics literature.  By applying this tensor trick, we can reduce the odd parity case to the even case, thus bypassing the non-uniqueness of unital cyclic models in this case.
Once we prove that $F^\CC_{g,n}$ is well-defined, its Morita invariance follows relatively easily, see Subsection~\ref{sec:definition}.

\subsection{Organization of the paper}

Section~\ref{sec:cy1} is mainly devoted to prove that a Calabi-Yau category (i.e. an $A_\infty$-category endowed with a strong Calabi-Yau structure) admits a cyclic model which is unique up to $A_\infty$ quasi-isomorphisms. Section~\ref{sec:cy2} is the unital version of Section~\ref{sec:cy1}. Adding the unitality condition requires the introduction of a new version of cyclic homology.

After briefly reviewing the definition of CEI in Section~\ref{sec:defi}, we proceed to formulate and prove their Morita invariance in Section~\ref{sec:morita}, using the results from Section~\ref{sec:cy2}. In Section~\ref{sec:ex-app}, we present some basic examples and applications, as well as some conjectures in the case of Fukaya categories and derived categories of coherent sheaves.

Finally, Appendix~\ref{sec:sign} deals with sign diagrams that appear in different places of the paper. In  Appendix~\ref{app:b-f}, we prove that CEI are invariant under tensoring with the Clifford algebra $\Cl=\mathbb{K}[\epsilon]$.

\subsection{Acknowledgments}
We would like to thank Cheol-Hyun Cho, Sasha Polishchuk, Yan Soibelman and Arkady Vaintrob for useful conversations about cyclic and Calabi-Yau \Ai-categories. LA would like to thank Ilia Zharkov for some helpful combinatorics lessons. 

JT was partially supported by the National Key Research and Development
Program of China No. 2023YFA1009803 and the NSFC grant 12071290.

\subsection{Conventions and Notations.} We shall always work over a field $\mathbb{K}$ of characteristic zero. Given a $\Z/2$-graded vector space $V$ and an element $v\in V$, we denote by $|v|$ the degree of $v$ and by $|v|'$ its shifted degree. That is $|v|'=|v|-1$. 

We denote by $\Sym^k (V)$ the $k$-th (graded) symmetric power of $V$, and by $\Sym (V)=\oplus_k \Sym^k (V)$.


For an $A_\infty$-category $\CC$ over $\mathbb{K}$, we use the notation $\hom_\CC(-,-)$ for the morphism space of $\CC$, while the notation $Hom(-,-)$ for the space of $\mathbb{K}$-linear maps between vector spaces over $\mathbb{K}$.

We use bold letters to denote elements 
$\textbf{x}= x_1 \otimes \ldots \otimes x_n $ and write $|\textbf{x}|'=\sum_i|x_i|'$ for its shifted degree.
Furthermore, we use Sweedler notation for the coproduct on a tensor coalgebra, namely:
$$ \sum {\bf x}^{(1)} \otimes {\bf x}^{(2)} = \sum_{i=0}^{n} (x_1\otimes \ldots \otimes x_i)\otimes (x_{i+1}\otimes \ldots \otimes x_n). $$

When working with a homologically $\mathbb{Z}$-graded complex $V=\oplus_n V_n$, its shift $V[k]$ is the graded vector space whose $n$-th graded piece is $V_{n-k}$.

\section{Calabi--Yau \Ai-categories}\label{sec:cy1}

\subsection{\Ai-bimodules and Hochschild invariants}

In this subsection, we recall some basic definitions of \Ai -bimodules, Hochschild (co) homology and related constructions.

Let $\CC$ be an \Ai -category. We use the following notation
\[\CC(X_0, X_1, \ldots, X_n):= \hom_\CC(X_0, X_1)[1] \otimes \cdots \otimes \hom_\CC(X_{n-1}, X_n)[1],\]
with the convention that $\CC(X_0, X_1, \ldots, X_n)=\mathbb{K}\cdot 1_{X_0}$, when $n=0$.
\begin{defn}
	Hochschild cochains of length $n$ are defined as:
	\[ C^\bullet(\CC)^n=\prod_{X_0, \ldots, X_n} Hom^\bullet(\CC(X_0, X_1, \ldots, X_n), \CC(X_0, X_n))[-1].
	\] 
	The Hochschild cochain complex is defined as $C^\bullet(\CC)=\prod_{n\geq 0} C^\bullet(\CC)^n$. A Hochschild cochain $\varphi=\prod_n \varphi_n$ is said to be of order $n$ if $\varphi_0=\ldots=\varphi_{n-1}=0$.
	
	The normalized (or reduced) Hochschild complex $C^\bullet_{red}(\CC)$ is defined as the subspace of Hochschild cochains $\varphi$ satisfying $\varphi_n(\ldots, \one_{X_i}, \ldots )=0$ for all $X_i$ and $n\geq 1$. 
\end{defn}

On the Hochschild cochain complex, one defines the Gerstenhaber product:
\[\varphi \bullet \psi (\textbf{x})= \sum (-1)^{|\textbf{x}^{(1)}|'|\psi|'}\varphi(\textbf{x}^{(1)}, \psi(\textbf{x}^{(2)}), \textbf{x}^{(3)}),\]
 and the corresponding bracket $[\varphi, \psi]:=\varphi \bullet \psi - (-1)^{|\varphi|'|\psi|'}\psi \bullet \varphi$. The differential on $C^\bullet(\CC)$ is defined as $\delta(-):=[\m,-]$, where $\m=\prod_n \m_n$ are the \Ai-operations. The corresponding cohomology is called the Hochschild cohomology of $\CC$ and is denoted by $HH^\bullet(\CC)$.

It is elementary to check that $C^\bullet_{red}(\CC)$ is closed under the Gerstenhaber product and a subcomplex of $C^\bullet(\CC)$, assuming that $\CC$ is strictly unital. On the cohomology level, this brings nothing new as the inclusion map $C^\bullet_{red}(\CC) \to C^\bullet(\CC)$ is a quasi-isomorphism, see \cite{Lod}, \cite{Cho}. 

One can reinterpret Hochschild cochains as \emph{vector fields} in the following way: first define
\[B\CC :=\bigoplus_{n\geq 0} \bigoplus_{X_0, \ldots, X_n} \CC(X_0, \ldots, X_n).
\]
As explained in \cite[Example 2.1.6]{KS}, $B\CC$ is naturally a counital, coalgebra. A Hochschild cochain $\varphi$ can be extended uniquely to a \emph{coderivation} of this coalgebra by the formula
\[\widehat{\varphi}(\textbf{x}):= \sum_n\sum (-1)^{\star|\phi|'} \textbf{x}^{(1)}\otimes \varphi_n(\textbf{x}^{(2)})\otimes \textbf{x}^{(3)}.
\]

\begin{rmk}\label{rmk:codev}
	In fact, one can check that any coderivation of $B\CC$ with the following support condition
	\begin{equation}
		\widehat{\varphi}(\CC(X_0, \ldots, X_n)) \subset \bigoplus_{Z_0, \ldots, Z_k} \CC(X_0, Z_0)\otimes \CC(Z_{0}, Z_{1}) \otimes \cdots \otimes \CC(Z_k, X_{n})
	\end{equation}
is uniquely determined by a Hochschild cochain by the formula above.
\end{rmk}
With this notation, we have $\varphi\bullet\psi=\varphi\circ\widehat{\psi}$. Also note that, due to our degree convention, we have $|\widehat{\varphi}|=|\varphi|+1$. With this in mind, an easy computation shows $\widehat{[ \varphi, \psi ]}= [\widehat{\varphi}, \widehat{\psi}]$.

\begin{defn}\label{defn:cc}
	The Hochschild complex of $\CC$ is defined as the vector space $$ C_\bullet(\CC)= \bigoplus_{X_0, \ldots, X_n} \CC(X_n, X_0) \otimes \CC(X_0, X_1, \ldots, X_n)[-1],$$
	with differential
	\begin{align}
		b(x_0 \otimes \textbf{x})= &\sum (-1)^{\star} x_0 \otimes \textbf{x}^{(1)}\otimes \m(\textbf{x}^{(2)})\otimes \textbf{x}^{(3)} + \sum (-1)^@ \m(\textbf{x}^{(3)}, x_0, \textbf{x}^{(1)})\otimes \textbf{x}^{(2)}
	\end{align}
	where $\star=|x_0|'+|\textbf{x}^{(1)}|' $ and $@= (|x_0|'+|\textbf{x}^{(1)}|'+ |\textbf{x}^{(2)}|') |\textbf{x}^{(3)}|'$.
\end{defn}

\begin{rmk}
	In the rest of the paper we will use the symbol @ for the sign obtained, following the Koszul convention for the shifted degrees, from rotating the inputs of the expression from the initial order.
\end{rmk}	

As for cochains, there is a \emph{reduced} Hochschild chain complex $C_\bullet^{red}(\CC)$. It is defined as the quotient of $C_\bullet(\CC)$ by the subcomplex spanned by chains of the form $x_0 \otimes x_1 \cdots \otimes x_i \otimes \one_{X_i} \otimes \cdots \otimes x_n$, for $i\geq 0 $. Once again the natural map $C_\bullet(\CC)\to C_\bullet^{red}(\CC)$ is a quasi-isomorphism \cite{Lod, Cho}.

We also recall the cyclic complex of $\CC$. We introduce two additional operators $b' , t : C_\bullet(\CC) \to C_\bullet(\CC)$. Let $t$ be the map (of degree 0) defined as
\[t(x_0 \otimes \cdots \otimes x_n)= (-1)^@ x_n \otimes x_0 \otimes \cdots \otimes x_{n-1}.
\]
The map $b'$, of degree 1, is defined as
\[b'(x_0 \otimes \textbf{x})=  \sum \m( x_0 \otimes \textbf{x}^{(1)})\otimes \textbf{x}^{(2)}+\sum (-1)^{\star} x_0 \otimes \textbf{x}^{(1)} \m(\textbf{x}^{(2)})\otimes \textbf{x}^{(3)}. \]
An important, but easy to verify, fact about this map is that it is a differential, that is $(b')^2=0$. 

\begin{defn}
	Let $\CC^{bar}$ be the complex whose underlying vector space is $C_\bullet(\CC)$ equipped with the differential $b'$. This is called the bar complex of $\CC$.
\end{defn}

The two maps above satisfy the following relation
\begin{equation}\label{eq:b'b}
	(\id-t)b'=b(\id-t).
\end{equation}	
This is a straightforward check. For the case of algebras, see \cite{Lod}.

\begin{defn}\label{defn:cyclic}
	Let $ C^\lambda_\bullet(\CC) = \left. C_\bullet(\CC) \middle/ \im(\id - t) \right.$. The relation in Equation (\ref{eq:b'b}) implies that the differential $b$ induces a differential on $ C^\lambda_\bullet(\CC) $. We will refer to the resulting complex as the (positive) cyclic complex of $\CC$, it is also sometimes called the Connes' complex.
	Its homology is called the cyclic homology of $\CC$, denoted by $HC^{\lambda}_\bullet(\CC)$.
\end{defn}

\begin{rmk}
It is often convenient to use different chain models for the cyclic homology of $\CC$, sometimes also referred to as \emph{positive cyclic homology}. We will make use of the $u$-model, $\big( C^{red}_\bullet(\CC)[u^{-1}], b+uB \big)$ where $B$ is the Connes' differential (see \cite{AT} for example). It is proved in \cite{Lod} (for algebras) and \cite{Cho} (in the \Ai \ case) that the homology $H_\bullet\big( C^{red}_\bullet(\CC)[u^{-1}], b+uB \big)$, usually denoted by $HC_\bullet^+(\CC)$,  is isomorphic to $HC^{\lambda}_\bullet(\CC)$.
\end{rmk}

Let $\CC$ and $\DD$ be \Ai-categories. An \Ai-pre-functor $F:\CC \to \DD$ consists of a map on objects and a sequence of maps
\[ F_n: \CC(X_0, \ldots, X_n) \to \DD(F X_0, F X_n),
\]
for $n\geq 1$. We say $F$ is unital if $F_1(\one_{X_i})=\one_{F X_i}$ and $F_n(\ldots, \one_{X_i}, \ldots)=0$ for all $n\geq 2$.

These can be uniquely extended to a counital, \emph{coalgebra} homomorphism $\widehat{F}: B\CC \to B\DD$, by the formula
\[\widehat{F}(\textbf{x})= \sum F(\textbf{x}^{(1)})\otimes F(\textbf{x}^{(2)}) \otimes \cdots \otimes F(\textbf{x}^{(k)}),
\]
and $\widehat{F}(1_X)=1_{F X}$. As in the case of coderivations, all the counital, coalgebra homomorphisms with the same support condition as in Remark \ref{rmk:codev} are obtained via this formula.


$F$ is called an \Ai-functor if $\widehat{F}\circ \widehat{\m_\CC} = \widehat{\m_\DD}\circ \widehat{F}$, or equivalently $F\circ \widehat{\m_\CC} = \m_\DD\circ \widehat{F}$. There is an (associative) composition operation for \Ai-functors (or pre-functors), denoted by $\circ$, and there is an identity functor (\cite[Section 1]{Sei}). Moreover $\widehat{F\circ G}= \widehat{F}\circ\widehat{G}$. 

In the case where a functor (or pre-functor) $F$ is a bijection on objects and all the $F_1$ maps are linear isomorphisms, it admits a strict inverse, which we denote by $F^{-1}$.\\

 An \Ai-functor $F$ induces a chain map $F_*: C_\bullet(\CC) \to C_\bullet(\DD)$ defined as follows
\begin{equation}\label{eq:funHH}
	F_*(x_0 \otimes \textbf{x})= \sum (-1)^@ F( \textbf{x}^{(k+1)} , x_0, \textbf{x}^{(1)} ) \otimes F(\textbf{x}^{(2)} )\otimes \cdots \otimes F(\textbf{x}^{(k)} )
\end{equation}
We remark that when $F$ is unital, this formula determines a chain map between the reduced complexes as well. This map is also functorial meaning, $(F\circ G)_*=F_*\circ G_*$.

There is  a similar induced chain map between the bar complexes, $F'_*:\CC^{bar}\to \DD^{bar}$ defined as 
\begin{equation}\label{eq:funHH'}
	F'_*(x_0 \otimes \textbf{x})= F(  x_0, \textbf{x}^{(1)} ) \otimes F(\textbf{x}^{(2)} )\otimes \cdots \otimes F(\textbf{x}^{(k)} ).
\end{equation}
Again, $(F\circ G)'_*=F'_*\circ G'_*$.
The two chain maps are related by the equation
\begin{equation}\label{eq:F'F}
	(\id-t)F'_*=F_*(\id-t).
\end{equation}	
In particular, this relation implies that $F_*$ induces a chain map between the cyclic complexes $ C^\lambda_\bullet(\CC) $ and $ C^\lambda_\bullet(\DD) $.

We now recall some facts about \Ai-bimodules following \cite{Sei, Sei2}.

\begin{defn}
	An \Ai-bimodule over $\CC$ (or $\CC-\CC$ bimodule) consists of a graded vector space $\PP(X,Y)$ for each pair of objects $X,Y$ in $\CC$ and a family of maps 
	\[\n_{r,s}: \CC(X_0, \ldots, X_r) \otimes \PP(X_r,Y_0) \otimes \CC(Y_0, \ldots, Y_s) \to \PP(X_0,Y_s)[1],\]
	for objects $X_0,\ldots X_r, Y_0, \ldots Y_s$, satisfying
	\begin{align}
		&\sum (-1)^{\star} \n_{r-i+1,s}(\textbf{x}^{(1)},\m_i(\textbf{x}^{(2)}), \textbf{x}^{(3)}, \underline{p}, \textbf{y}) \nonumber 	\\
		& + \sum (-1)^{\star}\n_{r,s-i+1}(\textbf{x}, \underline{p}, \textbf{y}^{(1)}, \m_i(\textbf{y}^{(2)}), \textbf{y}^{(3)}) \\
		& + \sum (-1)^{\star}\n_{r-i,s-j}(\textbf{x}^{(1)}, \underline{\n_{i,j}(\textbf{x}^{(2)}, \underline{p}, \textbf{y}^{(1)})}, \textbf{y}^{(2)} )=0, \nonumber
	\end{align}
	where $\star$ is the sum of the degrees to the left of the second multi-linear map. For example, in the second line above, $\star= |x_1|'+\ldots + |x_r|'+|y_1|'+\ldots + |y_{j-1}|' + |p|$. 
\end{defn}

\Ai -bimodules form a dg-category $[\CC, \CC]$ (see \cite{Sei}). Given \Ai-bimodules $\PP$ and $\QQ$ an element $\rho \in \hom^k_{[\CC, \CC]}(\PP, \QQ)$ of degree $k$, called a pre-homomorphism, consists of maps
\[\rho_{r, s}: \CC(X_0, \ldots, X_r) \otimes \PP(X_r,Y_0) \otimes \CC(Y_0, \ldots, Y_s) \to \QQ(X_0,Y_s)[k].\]
The differential in $\hom^\bullet_{[\CC, \CC]}(\PP, \QQ)$ is defined as
\begin{align}
	(\partial \rho)_{r,s}(\textbf{x}, \underline{p}, \textbf{y})= &\sum (-1)^{|\rho|\star} \n_{r_1,s_1}(\textbf{x}^{(1)}, \underline{\rho_{r_2,s_2}(\textbf{x}^{(2)}, \underline{p}, \textbf{y}^{(1)})}, \textbf{y}^{(2)} )\nonumber\\
	& + \sum (-1)^{1+|\rho|+\star} \rho_{r_1,s_1}(\textbf{x}^{(1)}, \underline{ \n_{r_2,s_2}( \textbf{x}^{(2)}, \underline{p}, \textbf{y}^{(1)})}, \textbf{y}^{(2)} )\\
	& + \sum (-1)^{1+|\rho|+\star} \rho_{r_1,s}(\textbf{x}^{(1)} \m_{i}( \textbf{x}^{(2)}) \textbf{x}^{(3)}, \underline{p}, \textbf{y} )\nonumber\\
	& + \sum (-1)^{1+|\rho|+\star} \rho_{r,s_1}(\textbf{x}, \underline{p}, \textbf{y}^{(1)} \m_i(\textbf{y}^{(2)}), \textbf{y}^{(3)})\nonumber
\end{align}	

\begin{defn}
	Given \Ai-bimodules $\PP$ and $\QQ$, a pre-homomorphism $\rho \in \hom^\bullet_{[\CC, \CC]}(\PP, \QQ)$ with $\partial \rho=0$ is called a \emph{homomorphism}. A homomorphism  is called a \emph{quasi-isomorphism} if $\rho_{0,0}:\PP(X, Y)\to \QQ(X, Y)$ induces an isomorphism on cohomology for all $X,Y$.
\end{defn}

A standard fact of the \Ai \ world is that any quasi-isomorphism has a homotopy inverse; see, for example, \cite{Sei}. Now we define the two most relevant bimodules for our purposes.

\begin{defn}
	Let $\CC_{sd}$ be the bimodule defined as $\CC_{sd}(X, Y):=\hom_{\CC}(X, Y)[1]$, with operations 
	\[\n_{r,s}(\textbf{x}, \underline{p}, \textbf{y}):=\m_{r+s+1}(\textbf{x}, p, \textbf{y}).\]
	We will refer to $\CC_{sd}$ as the (shifted) diagonal bimodule.
\end{defn}

\begin{defn}
	Let $\CC_{sd}^\vee$ be the bimodule defined as $\CC_{sd}^\vee(X, Y):=\hom_{\CC}(Y, X)^\vee[1]$, where $\vee$ stands for the $\bbK$-linear dual. We define the operations 
	\[\n_{r,s}^\vee(\textbf{x}, \underline{\pi}, \textbf{y})(p):=(-1)^{|\textbf{x}|'+|\textbf{y}|'+|p|'}\pi(\m_{r+s+1}(\textbf{y}, p, \textbf{x})),\]
    for $p\in \hom_{\CC}(Y_0, X_s).$
	We will refer to $\CC_{sd}^\vee$ as the (shifted) dual diagonal bimodule.
\end{defn}

In order to define a Calabi--Yau structure, we need the following lemma. Denote by $C_\bullet(\CC)^\vee$ the linear dual to $C_\bullet(\CC)$ equipped with the dual differential $b^\vee (\varphi)(\textbf{x})=(-1)^{|\varphi|}\varphi(b(\textbf{x}))$.

\begin{lem}
	There is a quasi-isomorphism of chain complexes
	\[\Psi: C_\bullet(\CC)^\vee\to \hom^\bullet_{[\CC, \CC]}(\CC_{sd}, \CC_{sd}^\vee),
	\]
	given explicitly as 
	\begin{equation}
		\Psi(\varphi)(\textbf{x}, \underline{v}, \textbf{y})(w)= \sum (-1)^@ \varphi(\m(\textbf{y}^{(3)}, w, \textbf{x}, v, \textbf{y}^{(1)}), \textbf{y}^{(2)})
	\end{equation}
\end{lem}
\begin{proof}
	It follows from \cite[Formula (2.27)] {Sei2} that $\hom^\bullet_{[\CC, \CC]}(\CC_{sd}, \CC_{sd}^\vee)$ is canonically isomorphic to the dual of $C_\bullet(\CC, \CC_d \otimes_{\CC} \CC_d)$, the Hochschild complex of the bimodule $\CC_d \otimes_{\CC} \CC_d$, where $\CC_d$ is the unshifted diagonal bimodule and $\otimes_{\CC}$ and is the tensor product of bimodules as defined in \cite{Sei}. The bimodule $\CC_d \otimes_{\CC} \CC_d$ is naturally quasi-isomorphic to $\CC_d$, see \cite[Formula (2.21)]{Sei}. This induces a quasi-isomorphism on Hochschild complexes $C_\bullet(\CC, \CC_d \otimes_{\CC} \CC_d) \simeq C_\bullet(\CC, \CC_d)$, where $C_\bullet(\CC, \CC_d)$ is by definition the complex $C_\bullet(\CC)$ from Definition \ref{defn:cc}. Tracing through these quasi-isomorphisms one obtains the map $\Psi$ in the statement.
\end{proof}

\begin{defn}~\label{def:cy}
	An element $\phi \in HH_\bullet(\CC)^\vee$ is called non-degenerate if $\Psi(\phi)$ is a quasi-isomorphism.
	
	A \emph{weak Calabi-Yau} structure on $\CC$ is a non-degenerate element  $\phi \in HH_\bullet(\CC)^\vee$.
	
	A \emph{strong Calabi-Yau} structure on $\CC$ is an element $\widetilde{\phi}\in HC^{\lambda}_\bullet(\CC)^\vee$, such that $\widetilde{\phi}\circ \pi_* \in HH_\bullet(\CC)^\vee$ is non-degenerate, where $\pi_*$ is induced by the natural projection $\pi: C_\bullet(\CC)\to C^{\lambda}_\bullet(\CC)$. We call the pair $(\CC, \widetilde{\phi})$ a Calabi-Yau \Ai-category.
\end{defn}

\subsection{From Calabi--Yau to homotopy cyclic}

In this subsection, we give an alternative description of strong Calabi-Yau structures. This is essentially a reinterpretation of a result from \cite{KS} for \Ai-algebras. The main difference is that instead of using the language of non-commutative geometry, we work with the more algebraic notion of a strong homotopy inner product \cite{Cho1}. This approach will also make it easier to tackle the unital version of these results (see Section 3) that we will later need.

\begin{lem}\label{lem:Schain}
	There is a degree one chain map $\displaystyle S: C_\bullet(\CC)^\vee \to \hom^\bullet_{[\CC, \CC]}(\CC_{sd}, \CC_{sd}^\vee)$ defined by the formula
	\[(S\varphi)(\textbf{x},\underline{v}, \textbf{y})(w)=(-1)^{@}\varphi(v,\textbf{y},w, \textbf{x} ) - (-1)^{@} \varphi(w, \textbf{x}, v, \textbf{y}),\]
	where the symbol $@$ stands for the sign obtained, following the Koszul convention for the shifted degrees, from rotating the inputs of the expression from their original order. 
\end{lem}
\begin{proof}
	This a long but straightforward calculation that we omit.
\end{proof}

\begin{rmk}
	The proof of Lemma \ref{lem:Schain} does not use the \Ai \ relations. Therefore, one can prove that the map $S$ commutes with Lie derivatives $\mathcal{L}_Z^\vee$ and $\mathcal{L}_Z$, introduced in Definitions \ref{defn:lie_cc} and \ref{defn:lie}, for any Hochschild cochain $Z$. When $Z=\m$, this is exactly the content of Lemma \ref{lem:Schain}.
\end{rmk}

\begin{defn}
	Given $\rho \in \hom^\bullet_{[\CC, \CC]}(\CC_{sd}, \CC_{sd}^\vee)$, we say 
	\begin{itemize}
		\item $\rho$ is \emph{anti-symmetric} if \[\rho(\textbf{x},\underline{v}, \textbf{y})(w)=-(-1)^@\rho(\textbf{y},\underline{w}, \textbf{x})(v).\]
		\item $\rho$ is \emph{closed} if for any $x_0\otimes \cdots \otimes x_n \in \CC(X_n, X_0, X_1 \ldots, X_n)$ and $0\leq i < j  < k \leq n$ we have 
		\begin{align*}
			(-1)^{\epsilon_j}\rho(\ldots,\underline{x_i},\ldots)(x_j)+ (-1)^{\epsilon_k}\rho(\ldots,\underline{x_j}, \ldots)(x_k)+(-1)^{\epsilon_i}\rho(\ldots, \underline{x_k}, \ldots)(x_i)=0,
		\end{align*}
		where $\epsilon_l=(|x_0|'+\ldots |x_l|')(|x_{l+1}|'+\ldots+|x_n|')$.
	\end{itemize}
	
	Denote by $\Omega^{2,cl}(\CC)$ the subspace of $\hom^\bullet_{[\CC, \CC]}(\CC_{sd}, \CC_{sd}^\vee )$ consisting of closed, anti-symmetric pre-homomorphisms.
\end{defn}
\begin{lem}\label{lem:Ssurjective}
	$\Omega^{2,cl}(\CC)$ is a subcomplex of $\hom^\bullet_{[\CC, \CC]}(\CC_{sd}, \CC_{sd}^\vee )$. Moreover $\Omega^{2,cl}(\CC)$ coincides with the image of the map $S$.
\end{lem}
\begin{proof}
	We first show that the image of $S$ equals $\Omega^{2,cl}(\CC)$. Since $S$ is a chain map, this implies $\Omega^{2,cl}(\CC)$ is a subcomplex.
		A direct calculation, that we skip, shows that the image of $S$ is contained in  $\Omega^{2,cl}(\CC)$.
		Next given $\rho \in \hom^\bullet_{[\CC, \CC]}(\CC_{sd}, \CC_{sd}^\vee)$ define $h(\rho)\in C_\bullet(\CC)^\vee$ as follows
	\begin{equation}\label{eq:Ssurj}
		h(\rho)(x_0 \otimes x_1\cdots \otimes x_n) = \sum_{i=1}^n \frac{(-1)^@}{n+1} \rho(\ldots, \underline{x_0}, \ldots)(x_i),
	\end{equation}
	for $n\geq 1$ and $h(\rho)(x_0)=0$.
	We claim that if $\rho \in \Omega^{2,cl}(\CC)$ we have $S(h(\rho))=\rho$, which gives the desired result. For this we compute

	\begin{align*}
		S(h(\rho))_{r,s} & (\textbf{x},\underline{v}, \textbf{y})(w)=  \frac{1}{r+s+2} \big(\rho(\textbf{x},\underline{v}, \textbf{y})(w)+ \sum_i (-1)^@\rho(\textbf{x}^{(3)},\underline{v}, \textbf{y}, w, \textbf{x}^{(1)})(x_i)+\\ 
		&+\sum_j(-1)^@ \rho(\textbf{y}^{(3)},w,\textbf{x},\underline{v},\textbf{y}^{(1)})(y_j) - (-1)^@ \rho(\textbf{y},\underline{w}, \textbf{x})(v) \\
		&- \sum_j (-1)^@ \rho(\textbf{y}^{(3)},\underline{w},\textbf{x},v,\textbf{y}^{(1)})(y_j) - \sum_i (-1)^@\rho(\textbf{x}^{(3)},v, \textbf{y},\underline{ w}, \textbf{x}^{(1)})(x_i)\big)\\
        = & \frac{1}{r+s+2} \big( \rho(\textbf{x},\underline{v}, \textbf{y})(w) - \sum_i (-1)^@ \rho(\textbf{y}, w, \textbf{x}^{(1)},\underline{x_i}, \textbf{x}^{(3)})(v)\\
		& + \sum_j (-1)^@ \rho(\textbf{y}^{(3)},w,\textbf{x},\underline{v},\textbf{y}^{(1)})(y_j) + \rho(\textbf{x},\underline{v}, \textbf{y})(w) \\
		& + \sum_j (-1)^@ \rho(\textbf{x},v, \textbf{y}^{(1)},\underline{y_j}, \textbf{y}^{(3)})(w)-\sum_i (-1)^@ \rho(\textbf{x}^{(3)},v,\textbf{y},\underline{w},\textbf{x}^{(1)})(x_i)\big)
        \end{align*}
        Here, on the second equality we used anti-symmetry of $\rho$. Then, we use closedness of $\rho$ twice: once to combine the third and fifth terms, and then the second and sixth. Thus the equation above simplifies to
        \begin{align*}
		& \frac{1}{r+s+2} \big( 2 \rho(\textbf{x},\underline{v}, \textbf{y})(w) - \sum_j (-1)^@ \rho(\textbf{y}^{(1)},y_j,\textbf{y}^{(3)},\underline{w}, \textbf{x})(v)\\
		& + \sum_i \rho(\textbf{x}^{(1)},x_i,\textbf{x}^{(3)},\underline{v}, \textbf{y})(w)\big)\\
		= & \frac{1}{r+s+2} \big( 2 \rho(\textbf{x},\underline{v}, \textbf{y})(w) + s \rho(\textbf{x},\underline{v}, \textbf{y})(w) + r \rho(\textbf{x},\underline{v}, \textbf{y})(w))\\
		= & \rho(\textbf{x},\underline{v}, \textbf{y})(w).
	\end{align*}
	Here, on the second equality we used again anti-symmetry of $\rho$. 
\end{proof}

We now identify the kernel of the map $S$.  Recall from Definition \ref{defn:cyclic}, there is a natural map of complexes 
\[\pi: C_\bullet(\CC)\to C^\lambda_\bullet(\CC).\]
We denote by $\mathfrak{K}_\CC$ the kernel of $\pi$ equipped with the induced differential. 

\begin{prop}\label{prop:exact1}
	We have the following short exact sequence of complexes
	\[0 \to C^\lambda_\bullet(\CC)^\vee  \stackrel{\pi^\vee }{\longrightarrow} C_\bullet(\CC)^\vee  \stackrel{S}{\longrightarrow} \Omega^{2,cl}(\CC)[1] \to 0. \]
\end{prop}
\begin{proof}
	Lemmas \ref{lem:Schain} and \ref{lem:Ssurjective} imply the map $S$ is a surjective chain map. From its definition, it is clear that the kernel of $S$ consists of $\varphi \in C_\bullet(\CC)^\vee $ satisfying 
	\[\varphi(x_0 \otimes \cdots \otimes x_n)= (-1)^@ \varphi(x_n \otimes x_0 \otimes \cdots \otimes x_{n-1}).\]
	But this is exactly the image of $\pi^\vee $.
\end{proof}

This short exact sequence is functorial in the following sense. 

\begin{prop}\label{prop:funct1}
	Let $F:\CC \to \DD$ be an \Ai-functor. We have the following commutative diagram
	\[ \begin{CD}
		0 @>>> C^\lambda_\bullet(\CC)^\vee  @>\pi^\vee >>  C_\bullet(\CC)^\vee  @>S>>  \Omega^{2,cl}(\CC)[1] @>>> 0 \\
		@. @A F^* AA    @A F^* AA      @A F^* AA     @. \\
		0 @>>> C^\lambda_\bullet(\DD)^\vee  @>\pi^\vee >>  C_\bullet(\DD)^\vee  @>S>>  \Omega^{2,cl}(\DD)[1] @>>> 0 
	\end{CD} \]
	Here $F^*$ in the first two vertical arrows is simply the dual to $F_*$ in (\ref{eq:funHH}). The third map $F^*$ is a chain map defined as follows:
	\[F^*\rho(\textbf{x},\underline{v}, \textbf{y})(w)= \sum (-1)^{@} \rho(\widehat{F}(\textbf{x}^{(2)}), \underline{F(\textbf{x}^{(3)}, v, \textbf{y}^{(1)})}, \widehat{F}(\textbf{y}^{(2)}))F(\textbf{y}^{(3)},w,\textbf{x}^{(1)}).\]
\end{prop}
\begin{proof}
	Commutativity of the left square is obvious, as the left most vertical arrow is induced by the middle one (Equation \ref{eq:F'F}).
	
	It is straightforward to check that the square on the right commutes once we observe that (by definition) $|F_k(x_1,\ldots, x_k)|'= |x_1|'+\ldots +|x_k|'$. Finally, one can check that $F^*$ on the right is a chain map, directly using the \Ai \ functor equation for $F$. Alternatively, one can use the fact that $S$ is surjective and that all the other sides of the commutative square are chain maps.
\end{proof}

\begin{rmk}
    A straightforward computation shows the map $F^*: \Omega^{2,cl}(\DD) \to \Omega^{2,cl}(\CC)$ respects composition: $(F\circ G)^*=G^*\circ F^*$.  
\end{rmk}

Next we describe $\mathfrak{K}_\CC$.

\begin{lem}
	Let $p: \CC^{bar} \to \mathfrak{K}_\CC$ be the map $p=\id - t$ and  let $N: C^\lambda_\bullet(\CC) \to \CC^{bar}$  be the map $N(x_0\otimes \cdots \otimes x_n)= \sum_{k=0}^n t^k(x_0\otimes \cdots \otimes x_n)$.
	
	The maps $p$ and $N$ are well defined chain maps. 
\end{lem}
\begin{proof}
	By definition, $\mathfrak{K}_\CC=\ker(\pi)$ is the image of $\id-t$, so the map $p$ is well defined. The fact that $p$ is a chain map is exactly Equation (\ref{eq:b'b}).
	
	Well-definedness of $N$ follows from the simple observation $N\circ t =N$. The condition $N\circ b = b' \circ N$ can be checked directly, see \cite{Lod} for the case of associative algebras.
\end{proof}

\begin{prop}\label{prop:exac2}
	We have the following short exact sequence of complexes
	\[0 \to C^\lambda_\bullet(\CC) \stackrel{N}{\longrightarrow} \CC^{bar} \stackrel{p}{\longrightarrow} \mathfrak{K}_\CC \to 0 \]
\end{prop}
\begin{proof}
	As observed before $p\circ N=0$. Also, by definition of $\mathfrak{K}_\CC$, $p$ is surjective. Given $\textbf{x}\in \ker(p)$, we have $t(\textbf{x})=\textbf{x}$. Therefore $\textbf{x}=N(\frac{1}{n+1}[\textbf{x}]).$
	
	All that is left to show is that $N$ is injective. If $N([\textbf{x}])=0$, then $\sum_{k=1}^n t^k(\textbf{x})=-\textbf{x}$. We define $\gamma :=-\frac{1}{n+1}\sum_{k=1}^n kt^k(\textbf{x})$, then 
	\[ (\id -t)(\gamma)= -\frac{1}{n+1}\Big(\sum_{k=1}^n t^k(\textbf{x}) - n \textbf{x}\Big)=-\frac{1}{n+1}\Big(-(n+1)\textbf{x}\Big)=\textbf{x}.
	\]
	Hence $[\textbf{x}]=0\in C^\lambda_\bullet(\CC)$.
\end{proof}

This short exact sequence is also functorial.  

\begin{prop}\label{prop:funct2}
	Let $F:\CC \to \DD$ be an \Ai-functor. We have the following commutative diagram
	\[ \begin{CD}
		0 @>>> C^\lambda_\bullet(\CC) @>N>>  \CC^{bar} @>p>>  \mathfrak{K}_{\CC} @>>> 0 \\
		@. @A F_* AA    @A F'_* AA      @A F_* AA     @. \\
		0 @>>> C^\lambda_\bullet(\DD) @>N>>  \DD^{bar} @>p>>  \mathfrak{K}_{\DD} @>>> 0 
	\end{CD} \]
\end{prop}
\begin{proof}
	Commutativity of the square on left is easy to check once we note $F_k$ has degree $1-k$.	
	The square on the right is exactly Equation (\ref{eq:F'F}). 	
\end{proof}

\begin{lem}\label{lem:baracyc}
	Assume $\CC$ is unital. Then $\CC^{bar}$ is an acyclic complex.
\end{lem}
\begin{proof}
	There is an explicit contracting homotopy $h : \CC^{bar} \to \CC^{bar}$. Given $x_0 \otimes \cdots x_n \in \CC(X_n, X_0, \ldots, X_n)$ one defines $h(x_0\otimes \cdots \otimes x_n)= \one_{X_n} \otimes x_0\otimes \cdots x_n$. One easily checks that $hb'+b'h=\id_{\CC^{bar}}$.
\end{proof}

Before we state the main result of this subsection, we introduce some terminology.

\begin{defn}
	A strong homotopy inner product (SHIP) on $\CC$, is an element $\rho \in \Omega^{2,cl}(\CC) \subseteq \hom^\bullet_{[\CC, \CC]}(\CC_{sd}, \CC_{sd}^\vee )$, which is a quasi-isomorphism of bimodules. 
\end{defn}

When $\CC$ is minimal and $\rho$ is  a \emph{constant} SHIP, meaning $\rho_{r,s}=0$ for $r,s>0$, then
\[\rho_{0,0}:=\langle -, - \rangle : \hom_\CC (X, Y)[1] \otimes \hom_\CC (Y, X)[1] \to \mathbb{K}
\]
defines an anti-symmetric, non-degenerate pairing. Moreover, $\partial \rho=0$ is equivalent to the following equation
\begin{equation}\label{eq:cyclic}
	\langle \m_n(x_1,\ldots, x_n), x_{n+1} \rangle = (-1)^@ \langle \m_n(x_2,\ldots, x_{n+1}), x_1 \rangle.
\end{equation}
In this case, $\CC$ together with the pairing $\langle -, - \rangle$ is called a \emph{cyclic \Ai-category}.

Given two cyclic \Ai-categories $\CC$ and $\DD$, a functor $F:\CC \to \DD$ is called cyclic if $F^*\langle -, - \rangle_\DD=\langle -, - \rangle_\CC$. This condition is equivalent to the following equations
\begin{equation}\label{eq:F_cyclic}
	\langle F_1(x), F_1(y) \rangle_\DD = \langle x, y \rangle_\CC ; \  \  \sum_{i=1}^n \langle F_i(x_1,\ldots, x_i), F_{n-i}(x_{i+1},\ldots,x_n) \rangle_\DD=0, n\geq 3.
\end{equation}

\begin{cor}\label{cor:CY_SHIP}
	The complexes $\Omega^{2,cl}(\CC)$ and $C^\lambda_\bullet(\CC)^\vee$ are quasi-isomorphic. This quasi-isomorphism is functorial with respect to \Ai-functors.
	Moreover, it identifies strong Calabi--Yau structures with SHIPs.
\end{cor}
\begin{proof}
	Proposition \ref{prop:exact1} gives the isomorphism
	\[\Omega^{2,cl}(\CC)[1] \cong  \left. C_\bullet(\CC)^\vee \middle/ C^\lambda_\bullet(\CC)^\vee  \right. \cong \mathfrak{K}_{\CC}^\vee .\]
	Proposition \ref{prop:exac2}, together with $\CC^{bar}$ being acyclic, implies that $\mathfrak{K}_{\CC}$ and   $C^\lambda_\bullet(\CC)[1]$ are quasi-isomorphic. Combining these we obtain the first statement, and therefore, on cohomology level we obtain an isomorphism
	\[HC^\lambda_\bullet(\CC)^\vee  \cong H^\bullet(\Omega^{2,cl}(\CC), \partial).\]
	
	Let $\phi \in HC^\lambda_\bullet(\CC)^\vee $ and $\rho \in  H^\bullet(\Omega^{2,cl}(\CC), \partial)$ be related by the above isomorphism. We claim that $\phi$ is a Calabi-Yau structure if and only if $\rho$ is a quasi-isomorphism of bimodules. We assume, without loss of generality that $\CC$ is a minimal \Ai-category.
	Recall $\phi$ is Calabi-Yau when $\Psi(\phi\circ \pi)$ is a quasi-isomorphism, which is equivalent to the pairing
	\[\Psi(\phi\circ \pi)_{0,0}(v)(w)= (-1)^@\phi_0(\m_2(w,v)),\]
	being non-degenerate as a map $\hom_\CC(X, Y)\otimes \hom_\CC(Y, X)\to \mathbb{K}$, for all objects $X, Y$.
		Similarly, $\rho$ is a quasi-isomorphism if the pairing $\rho_{0,0}(v)(w)$ is non-degenerate. Take $\theta \in C_\bullet(\CC)^\vee $ with $S(\theta)=\rho$. Then 
	$\rho_{0,0}(v)(w)= \theta_1(v\otimes w - (-1)^@ w \otimes v)$.  Diagram chasing one sees that $v\otimes w - (-1)^@ w \otimes v \in \mathfrak{K}_\CC$ is sent to $(-1)^@ \m_2(w,v)\in C^\lambda_\bullet(\CC)[1]$. Therefore under the above isomorphism 
	$$\rho_{0,0}(v)(w)=\theta_1(v\otimes w - (-1)^@ w \otimes v)= (-1)^@\phi_0(\m_2(w,v)).$$
	Hence $\rho_{0,0}(v)(w)= \Psi(\phi\circ \pi)_{0,0}(v)(w)$, which proves the claim.
	
	Finally, the statement on functoriality follows directly from Propositions \ref{prop:funct1} and \ref{prop:funct2}.
\end{proof}


\begin{ex}\label{ex:cyclicCY}
Let $\CC$ be a cyclic \Ai-category. We claim that the the cyclic structure $\rho:=\langle -, -\rangle \in \Omega^{2,cl}(\CC)$ corresponds, by Corollary \ref{cor:CY_SHIP}, to the strong Calabi--Yau structure $\phi \in HC^\lambda_\bullet(\CC)^\vee$ defined as
\begin{align}\label{eq:CY}
	\phi([ x_0\otimes \cdots \otimes x_n])&=\left\{
	\begin{array}{ll}
		0 & , n>0\\
		\langle \one_{X_0}, x_0 \rangle & ,  n=0.
	\end{array}
	\right.
\end{align}
In order to see this first note that if we define $\theta\in C_\bullet(\CC)^\vee$ as $\theta_1(x_0 \otimes x_1):=\frac{1}{2}\langle x_0, x_1 \rangle$, and all other entries zero, then $\theta$ determines an element in $\mathfrak{K}_\CC^\vee$ with $S(\theta)=\rho$. Diagram chasing one sees that the corresponding Calabi--Yau structure is given by
\[\phi([ x_0\otimes \cdots \otimes x_n])= \theta\left(p(\one_{X_0} \otimes N(x_0\otimes \cdots \otimes x_n))\right),\]
which by definition of $\theta$ is non-zero only when $n=0$ in which case it equals $\theta(\one_{X_0} \otimes x_0- (-1)^@ x_0 \otimes \one_{X_0})=\langle \one_{X_0}, x_0 \rangle $, by anti-symmetry of the inner product.
\end{ex}

\subsection{From homotopy cyclic to cyclic}

In this subsection we prove an analogue of the Darboux theorem in symplectic geometry for SHIPs. Namely we show that up to an \Ai-isomorphism any SHIP can be made constant. Therefore any \Ai-category with a SHIP has a cyclic model. We also study the uniqueness of this cyclic model.

From now on we assume that $\CC$ is minimal. Recall, the homological perturbation lemma guarantees one can always achieve this. Together with our starting assumption that $\CC$ is proper, this implies that all $\hom$ spaces in $\CC$ are finite dimensional.

\begin{defn}\label{defn:contraction}
	Let $\rho \in \hom^\bullet_{[\CC, \CC]}(\CC_{sd}, \CC_{sd}^\vee)$ be a bimodule pre-homomorphism and $Z \in C^\bullet(\CC)$ a Hochschild cochain. We define the contraction $\displaystyle \iota_Z \rho \in C_\bullet(\CC)^\vee $ as follows
	\[\iota_Z \rho (x_0 \otimes x_1 \cdots \otimes x_n)= \sum (-1)^{@+\dagger} \rho(x_1, \ldots \underline{Z(x_i, \ldots, x_j)}\ldots x_n)(x_0).
	\]
	where $\dagger:=|Z|'( |\rho| + |x_1|'+\ldots+|x_{i-1}|')$.
\end{defn}

\begin{defn}\label{defn:lie}
	Let $Z \in C^\bullet(\CC)$ be a Hochschild cochain. We define the Lie derivative
	\[\mathcal{L}_Z: \hom^\bullet_{[\CC, \CC]}(\CC_{sd}, \CC_{sd}^\vee ) \to \hom^\bullet_{[\CC, \CC]}(\CC_{sd}, \CC_{sd}^\vee )\]
by the formula
\begin{align}
	(\mathcal{L}_Z\rho)_{r,s}(\textbf{x}, \underline{v}, \textbf{y})(w)= &\sum (-1)^{@+1+|Z|'(|\rho|+|\textbf{y}^{(1)}|'+|v|'+|\textbf{x}^{(2)}|')} \rho(\textbf{x}^{(2)}, \underline{v}, \textbf{y}^{(1)})(Z(\textbf{y}^{(2)}, w,  \textbf{x}^{(1)}))\nonumber\\
	& + \sum (-1)^{1+(|\rho|+\star)|Z|'} \rho_{r_1,s_1}(\textbf{x}^{(1)}, \underline{ Z( \textbf{x}^{(2)}, v, \textbf{y}^{(1)})}, \textbf{y}^{(2)} )(w)\\
	& + \sum (-1)^{1+(|\rho|+\star)|Z|'} \rho_{r_1,s}(\textbf{x}^{(1)}, Z( \textbf{x}^{(2)}), \textbf{x}^{(3)}, \underline{v}, \textbf{y} )(w)\nonumber\\
	& + \sum (-1)^{1+(|\rho|+\star)|Z|'} \rho_{r,s_1}(\textbf{x}, \underline{v}, \textbf{y}^{(1)}, Z(\textbf{y}^{(2)}), \textbf{y}^{(3)})(w)\nonumber
\end{align}	
\end{defn}

Note that when $Z=\m$ we have $\mathcal{L}_Z\rho=\partial \rho$.

\begin{defn}\label{defn:lie_cc}
	Let $Z \in C^\bullet(\CC)$ be a Hochschild cochain. We define the Lie derivative map
	\[\mathcal{L}_Z: C_\bullet(\CC) \to C_\bullet(\CC)
	\]
	by the formula
	\begin{align*}
	\mathcal{L}_Z(x_0 \otimes \textbf{x})= &\sum (-1)^{\star|Z|'} x_0 \otimes \textbf{x}^{(1)} \otimes Z(\textbf{x}^{(2)})\otimes \textbf{x}^{(3)} + \sum (-1)^@ Z(\textbf{x}^{(3)}, x_0, \textbf{x}^{(1)})\otimes \textbf{x}^{(2)}
\end{align*}
\end{defn}

Please note that when $Z=\m$, we have $\mathcal{L}_\m=b$, the Hochschild differential. The following lemma is well-known.

\begin{lem}\label{lem:lieder}
    Let $Z, Y \in C^\bullet(\CC)$ be Hochschild cochains.
    \begin{enumerate}
        \item $[\mathcal{L}_Z,\mathcal{L}_Y]= \mathcal{L}_{[Z,Y]}$ ;
        \item If Z is reduced, $\mathcal{L}_Z \circ B=(-1)^{|Z|'}B\circ \mathcal{L}_Z $.
    \end{enumerate}
\end{lem}

\begin{lem}\label{lem:contraction}
	Let $\rho \in \Omega^{2,cl}(\CC)$ be a bimodule pre-homomorphism.
	\begin{enumerate}
		\item If $\rho$ is a quasi-isomorphism then the map $\iota_{-}\rho: C^\bullet(\CC) \to C_\bullet(\CC)^\vee $ is an isomorphism. 
	    \item For Hochschild cochains $Y, Z$, we have the identity $$\iota_{[Z,Y]}\rho=-\iota_{Z}(\mathcal{L}_Y \rho) + (-1)^{|Z||Y|'} \mathcal{L}_Y^\vee (\iota_{Z}\rho),$$
	    where $\mathcal{L}_Y^\vee (\psi)(\textbf{x})=(-1)^{|\psi||Y|'}\psi(\mathcal{L}_Y(\textbf{x}))$. In particular, when $Y=\m$ we have $\iota_{\delta(Z)}\rho=(-1)^{|Z|'}\iota_{Z}(\partial \rho) + b^\vee (\iota_{Z}\rho)$.
		\item For a Hochschild cochain $Z $ we have $\mathcal{L}_Z \rho = S(\iota_Z \rho)$. In particular $\mathcal{L}_Z \rho \in \Omega^{2,cl}(\CC)$.
		\item Let $F:\CC \to \CC$ be an \Ai-pre-functor with the inverse $F^{-1}$. Given a Hochschild cochain $Z$ we define a new one by the formula $F^*Z:= F^{-1}\circ \widehat{Z}\circ \widehat{F}$. 
		
		We have $\displaystyle F^*(\mathcal{L}_z \rho)= \mathcal{L}_{F^*Z}(F^* \rho)$.
	\end{enumerate}
\end{lem}
\begin{proof}
	We have assumed that $\CC$ is minimal, therefore $\rho$ is a quasi-isomorphism if and only if 
	\[\rho_{0,0} : \hom_\CC (X, Y) \otimes \hom_\CC (Y, X) \to \mathbb{K}
	\]
	is a non-degenerate pairing for all $X, Y$. Take $\varphi \in C_\bullet(\CC)^\vee $, we will show by induction on length that there exists a unique $Z \in C^\bullet(\CC)$ satisfying $\iota_{Z}\rho=\varphi.$ Assume we have found $Z_k$ for $k<n$, in order to define $Z_n$ one must solve
	
	\begin{align*}\varphi(x_0 \otimes x_1 \otimes \cdots \otimes x_n)= & (-1)^{@+\dagger}\rho_{0,0}(Z_n(x_1, \ldots, x_n))(x_0) \\
		 & + \sum_{0\leq k <n} (-1)^{@+\dagger} \rho(x_1, \ldots, \underline{Z_k(x_{i+1}, \ldots, x_{i+k})},\ldots x_n)(x_0).
\end{align*}
Since  $\rho_{0,0}$ is non-degenerate there is a unique $Z_n(x_1, \ldots, x_n)$ solving this equation, for all $x_0$.

The second item in this lemma follows from a direct computation that we leave to the interested reader. For the third item we compute
\begin{align*}
	S(\iota_Z\rho)&(\textbf{x},\underline{v}, \textbf{y})(w)= \\
	& \sum (-1)^{@+\dagger}\rho(\textbf{y}^{1},\underline{Z(\textbf{y}^{2}, w, \textbf{x}^{1})}, \textbf{x}^{2})(v) - \sum (-1)^{@+\dagger}\rho(\textbf{x}^{(1)},\underline{Z(\textbf{x}^{2}, v, \textbf{y}^{1})}, \textbf{y}^{(2)})(w)\\ 
	&+\sum(-1)^{@+\dagger} \rho(\textbf{y}^{1},\underline{Z(\textbf{y}^{2})},\textbf{y}^{3}, w, \textbf{x})(v) - \sum (-1)^{@+\dagger} \rho(\textbf{x}^{1},\underline{Z(\textbf{x}^{2})},\textbf{x}^{3}, v, \textbf{y})(w) \\
	&+\sum(-1)^{@+\dagger} \rho(\textbf{y}, w, \textbf{x}^{1}, \underline{Z(\textbf{x}^{2})},\textbf{x}^{3})(v)  - \sum(-1)^{@+\dagger}  \rho(\textbf{x}, v, \textbf{y}^{1}, \underline{Z(\textbf{y}^{2})},\textbf{y}^{3})(w) \\
	= & -\sum (-1)^{@+\dagger} \rho(\textbf{x}^{2},\underline{v}, \textbf{y}^{1})(Z(\textbf{y}^{2}, w, \textbf{x}^{1})) - \sum (-1)^{@+\dagger}  \rho(\textbf{x}^{(1)},\underline{Z(\textbf{x}^{2}, v, \textbf{y}^{1})}, \textbf{y}^{(2)})(w)\\ 
	&-\sum(-1)^{@+\dagger} \rho(\textbf{y}^{3}, w, \textbf{x},\underline{v},\textbf{y}^{1})(Z(\textbf{y}^{2})) + \sum (-1)^{@+\dagger}  \rho(\textbf{x}^{3}, v, \textbf{y},\underline{w},\textbf{x}^{1})(Z(\textbf{x}^{2})) \\
	&+\sum(-1)^{@+\dagger} \rho(\textbf{y}, w, \textbf{x}^{1}, \underline{Z(\textbf{x}^{2})},\textbf{x}^{3})(v)  - \sum(-1)^{@+\dagger} \rho(\textbf{x}, v, \textbf{y}^{1}, \underline{Z(\textbf{y}^{2})},\textbf{y}^{3})(w) \\
	=  &  -\sum (-1)^{@+\dagger}  \rho(\textbf{x}^{2},\underline{v}, \textbf{y}^{1})(Z(\textbf{y}^{2}, w, \textbf{x}^{1})) - \sum (-1)^{@+\dagger} \rho(\textbf{x}^{(1)},\underline{Z(\textbf{x}^{2}, v, \textbf{y}^{1})}, \textbf{y}^{(2)})(w)\\ 
	& + \sum(-1)^{@+\dagger} \rho(\textbf{y}^{1}, Z(\textbf{y}^{2}), \textbf{y}^3,\underline{w},\textbf{x})(v) - \sum (-1)^{\dagger}  \rho(\textbf{x}^{1}, Z(\textbf{x}^{2}), \textbf{x}^3,\underline{v},\textbf{y})(w)\\
	=   &  \sum (-1)^{@+1+\dagger}  \rho(\textbf{x}^{2},\underline{v}, \textbf{y}^{1})(Z(\textbf{y}^{2}, w, \textbf{x}^{1})) + \sum (-1)^{1+\dagger} \rho(\textbf{x}^{(1)},\underline{Z(\textbf{x}^{2}, v, \textbf{y}^{1})}, \textbf{y}^{(2)})(w)\\ 
	& + \sum(-1)^{1+\dagger}  \rho(\textbf{x},\underline{v},\textbf{y}^{1}, Z(\textbf{y}^{2}), \textbf{y}^3)(w) + \sum (-1)^{1+\dagger}  \rho(\textbf{x}^{1}, Z(\textbf{x}^{2}), \textbf{x}^3,\underline{v},\textbf{y})(w)\\
	= & \ \mathcal{L}_Z \rho.
\end{align*}
	Here the first and last equality follow from the definitions, the second and fourth from anti-symmetry of $\rho$ and the third equality follows from closedness of $\rho$. 
	
	The final item follows from a direct calculation, simply using the definitions and the fact that $\widehat{F} \circ \widehat{F^*Z} = \widehat{Z}\circ \widehat{F}$.
\end{proof}

Next we define (one-parameter) families of Hochschild cochains as follows 
	\begin{equation}\label{eq:param}
	    C^\bullet(\CC)\{t\}=\prod_{n\geq 0}\prod_{X_0, \ldots, X_n} Hom^\bullet(\CC(X_0, X_1 \ldots, X_n), \CC(X_0, X_n))\otimes \mathbb{K}[t][-1].
	\end{equation} 
Similarly we can define families of \Ai-pre-functors, which are constant on objects.

The next lemma shows the existence and uniqueness of \emph{flows} in our setting.

\begin{lem}\label{lem:flow}
	Let $Z^t \in C^1(\CC)\{t\}$ be a family of degree one Hochschild cochains of order two, that is $Z^t_0=Z^t_1=0$. Then there exists an unique family of \Ai \ pre-functors $F^t: \CC \to \CC$, which is the identity on objects and satisfies
	\begin{equation}\label{eq:flow}
F^0=\id, \ \ \frac{d}{dt} F^t = Z^t \circ \widehat{F^t}.
	\end{equation}
		We refer to $F^t$ as the flow of $Z^t$. In addition, if $Z^t$ is a reduced cochain, then $F^t$ is unital. Moreover, if $\delta(Z^t)=[\m, Z^t]=0$ then $F^t$ is an \Ai-functor.
\end{lem}
\begin{proof}
	We will construct the  \Ai \ pre-functor maps $F^t_n$ by induction on $n$. For $n=1$, Equation (\ref{eq:flow}) immediately gives $F^t_1|_{\CC(X_0,X_1)}= \id_{\CC(X_0,X_1)}$, since $Z^t$ has order two. For the induction step we have to solve the equation
	\[\frac{d}{dt} F^t_n(x_1, \ldots, x_n) = \sum_{k \geq 2} \sum_{i_1+\ldots +i_k=n} Z^t_k(F^t_{i_1}(x_1,\ldots ),\ldots, F^t_{i_k}( \ldots, x_n)).\]  
	Note that every term on the right-hand side of the equation was constructed in the previous induction steps, therefore we can find an unique $F^t_n(x_1, \ldots, x_n)$, by integrating the right-hand side, with the initial condition $F^0_n=0$. Note that, since $Z^t$ has degree $1$, $F^t_n$ has (shifted) degree $0$ as required.
	
	By construction $F^t_1(\one_{X_i})=\one_{X_i}$. For $n\geq 2$, if $Z^t$ is reduced it is clear that the right-hand side of the above equation vanishes whenever one of the $x_j=\one_{X_j}$. Therefore $F^t_n( \ldots, \one_{X_j}, \ldots)=0$, hence $F^t$ is unital.
	
	We now assume that $[\m, Z^t]=0$ and compute
	\begin{align*}
	 \frac{d}{dt}(\widehat{F^t}\circ \widehat{\m} - \widehat{\m} \circ \widehat{F^t})& = \widehat{Z^t} \circ \widehat{F^t} \circ \widehat{\m} - \widehat{\m} \circ \widehat{Z^t} \circ \widehat{F^t}\\
 & = \widehat{Z^t} \circ (\widehat{F^t} \circ \widehat{\m} - \widehat{\m}\circ\widehat{F^t}  ).
 	\end{align*}
 Here we used the fact that $\frac{d}{dt} \widehat{F^t} = \widehat{Z^t} \circ \widehat{F^t}$, which is equivalent to Equation (\ref{eq:flow}).
	Therefore, as before we conclude, by uniqueness of the solution to this equation, that $\widehat{F^t} \circ \widehat{\m} - \widehat{\m}\circ\widehat{F^t}=0$ which is equivalent to the \Ai-functor equation for $F^t$. 
\end{proof}

\begin{lem}\label{lem:lie_pull}
Let $Z^t$ and $F^t$	be as in the previous lemma. 
	\begin{enumerate}
		\item Given $\rho \in \hom^\bullet_{[\CC, \CC]}(\CC_{sd}, \CC_{sd}^\vee )$, we have $\displaystyle \frac{d}{dt} \big( (F^t)^*\rho\big) = - (F^t)^*(\mathcal{L}_{Z^t} \rho )$.
		\item Let $\textbf{x}\in C_\bullet (\CC)$, we have $\displaystyle \frac{d}{dt} \big(  (F^t)_*(\textbf{x})  \big)=  \mathcal{L}_{Z^t} \big(  F^t_*(\textbf{x})  \big)$.
		\item Given $\theta\in C_\bullet (\CC)^\vee $, we have $\displaystyle \frac{d}{dt} \big( (F^t)^*\theta\big) =  (F^t)^*(\mathcal{L}_{Z^t}^\vee \theta )$.
	\end{enumerate}
\end{lem}
\begin{proof}
We observe that the third statement is the dual to the second one. The first two claims have identical proofs, we prove the second one:
\begin{align*}
	\frac{d}{dt} \big(  (F^t)_*(x_0\otimes\textbf{x})  \big) & =  \sum (-1)^@ \frac{d F^t}{dt}(\textbf{x}^{(3)}, x_0, \textbf{x}^{(1)})\otimes \widehat{F^t}(\textbf{x}^{(2)}) +\\
	& \ \ + \sum (-1)^@ F^t(\textbf{x}^{(3)}, x_0, \textbf{x}^{(1)})\otimes \frac{d \widehat{F^t}}{dt}(\textbf{x}^{(2)})=\\
	& = \sum (-1)^@ Z^t\circ\widehat{F^t}(\textbf{x}^{(3)}, x_0, \textbf{x}^{(1)})\otimes \widehat{F^t}(\textbf{x}^{(2)}) +\\
	& \ \ + \sum (-1)^@ F^t(\textbf{x}^{(3)}, x_0, \textbf{x}^{(1)})\otimes \widehat{Z^t}\circ\widehat{F^t}(\textbf{x}^{(2)}),
\end{align*}
	where we used Equation (\ref{eq:flow}) in the second equality. The expression above then equals
	\[\mathcal{L}_{Z^t}\big(\sum (-1)^@ F^t(\textbf{x}^{(3)}, x_0, \textbf{x}^{(1)})\otimes \widehat{F^t}(\textbf{x}^{(2)})\big)=\mathcal{L}_{Z^t}\big(  (F^t)_*(x_0\otimes\textbf{x})  \big).\]
\end{proof}

We are now ready to prove the analogue of the Darboux theorem in our setting.

\begin{prop}\label{prop:darboux}
	Let $\rho \in \Omega^{2,cl}(\CC)$ be a SHIP. There exists a cyclic \Ai-category $\CC'$, with the same objects and morphism spaces as $\CC$, and an \Ai-isomorphism $F: \CC' \to \CC$. Moreover, if we denote by $\langle -, - \rangle$ the cyclic pairing (or constant SHIP on $\CC'$), we have $F^*\rho=\langle -, - \rangle$.
\end{prop}
\begin{proof}
	Define $\rho^t=(1-t)\rho_{0,0} +t \rho \in \hom^\bullet_{[\CC, \CC]}(\CC_{sd}, \CC_{sd}^\vee )\{t\}$. It is easy to see that $\rho^t$ is anti-symmetric, closed and non-degenerate since $\rho^t_{0,0}= \rho_{0,0}$. By Proposition \ref{prop:exact1}, there exists $\theta \in C_\bullet(\CC)^\vee$ with $S(\theta)= \rho - \rho^0$. In fact, since $(\rho-\rho^0)_{0,0}=0$, we can take $\theta$ satisfying $\theta_0=\theta_1=0$. Lemma \ref{lem:contraction} then implies the existence of a Hochschild cochain $Z^t$ satisfying $\iota_{Z^t} \rho^t= \theta$. Moreover from the construction we see that we can choose $Z^t$ of order two, that is $Z^t_0=Z^t_1=0$. Also note that, since $S$ has degree 1, $|\theta|=|\rho|+1$ which implies $Z$ has degree one.
	
	Then applying Lemma \ref{lem:flow} to $Z^t$ we obtain a family of \Ai-pre-functors $F^t$ satisfying (\ref{eq:flow}). We compute
	\begin{align*}
		\frac{d}{dt} (F^t)^* \rho^t = &  (F^t)^*(\frac{d \rho^t}{dt} - \mathcal{L}_{Z^t}\rho^t)\\
		 = &  (F^t)^*( \rho -\rho^0 - S(\iota_{Z^t} \rho^t))\\
		 =&  (F^t)^*( S(\theta) - S(\theta)) = 0.
	\end{align*}
	Here we used Lemma \ref{lem:lie_pull}(1) in the first equality and Lemma \ref{lem:contraction}(3) in the second.
Therefore $(F^1)^* \rho=(F^1)^* \rho^1= (F^0)^*\rho^0= \rho_{0,0}$ and we can take $F:=F^1$. We then define a new \Ai-category $\CC'$, with the same objects and $\hom$ spaces as $\CC$ and operations given by the formula
\[ \m':= F^{-1} \circ \widehat{\m} \circ \widehat{F}.
\]
A simple computation shows that $\m'$ is an \Ai\   structure, and by construction $F: \CC' \to \CC$ defines an \Ai-isomorphism with $F^*\rho=\rho_{0,0}$. Finally,  we have
\[\partial \rho_{0,0} = \mathcal{L}_{\m'}\rho_{0,0}= \mathcal{L}_{F^*\m}(F^*\rho) = F^*(\mathcal{L}_{\m}\rho)= F^*(\partial \rho) = 0,  \]
where the third equality follows from Lemma \ref{lem:contraction}(4). This implies $\rho_{0,0}$ is a constant strong homotopy inner product for $\CC'$, hence $\CC'$ is cyclic.
\end{proof}

The previous proposition guarantees that any (minimal) \Ai-category equipped with a strong homotopy inner product is isomorphic to a cyclic one. We want to examine to what extent this is unique. For this purpose we need the following lemma, which is the analogue of the Moser stability theorem in Symplectic Geometry.

\begin{lem}\label{lem:homologous}
	Let $\rho^0$ and $\rho^1$ be two SHIPs with $[\rho^0]=[\rho^1]\in H^\bullet(\Omega^{2,cl}(\CC), \partial)$. Then there exists an \Ai-isomorphism $F: \CC \to \CC$ with $F^*\rho^1=\rho^0$.
\end{lem}
\begin{proof}
	By assumption $\rho^1-\rho^0= \partial \beta$ for some $\beta \in \Omega^{2,cl}(\CC)$.  This, in particular, implies $(\rho^1)_{0,0}=(\rho^0)_{0,0}$, since $\CC$ is minimal. Therefore $\rho^t:=(1-t)\rho^0+t \rho^1$ is a quasi-isomorphism for all $t$. Moreover, by Lemma \ref{lem:Ssurjective}, $\beta=S(\theta)$ for some $\theta \in C_\bullet(\CC)^\vee$, with $\theta_0=0$. 
	
	Lemma \ref{lem:contraction}(1) guarantees the existence of $Z^t$ (of degree 1) satisfying $\iota_{Z^t} \rho^t = b^\vee \theta$. Since $\theta_0=0$ and $\CC$ is minimal, $(b^\vee\theta)_0=(b^\vee\theta)_1=0$ and therefore $Z^t$ has order 2. Hence we can apply Lemma \ref{lem:flow} to $Z^t$ and obtain a family of \Ai-pre-functors $F^t$. As in the previous proof,  we have
		\begin{align*}
		\frac{d}{dt} (F^t)^* \rho^t = &  (F^t)^*(\frac{d \rho^t}{dt} - \mathcal{L}_{Z^t}\rho^t)\\
		= &  (F^t)^*( \rho^1 -\rho^0 - S(\iota_{Z^t} \rho^t))\\
		=&  (F^t)^*( \partial S(\theta) - S(b^\vee(\theta)) = 0,
	\end{align*}
	since $S$ is a chain map. Hence $(F^t)^* \rho^t = \rho^0$. 
	Additionally, by Lemma \ref{lem:contraction}(2), one has
	\[\iota_{\delta(Z^t)} \rho^t =\iota_{Z^t}(\partial \rho^t)+b^\vee(\iota_{Z^t}\rho^t)= b^\vee(\iota_{Z^t} \rho^t)= b^\vee(b^\vee \theta)=0.\]
	
	Therefore, since $\rho^t$ is non-degenerate we conclude $\delta(Z^t)=0$ and therefore $F^t$ are \Ai-functors by Lemma \ref{lem:flow}. Hence $F^1$ is the required \Ai-isomorphism.
\end{proof}

The following theorem assembles all the results from this section.

\begin{thm}~\label{thm:cyclic-model}
	Let $(\CC, \phi_\CC)$ and $(\DD, \phi_\DD)$ be minimal Calabi--Yau \Ai-categories and let $F:\CC \to \DD$ be an \Ai-functor with $[F^*\phi_\DD]=[\phi_\CC]$. Then there exist cyclic \Ai-categories $\CC'$ and $\DD'$, \Ai-isomorphisms $G_\CC: \CC'\to \CC$ and $G_\DD: \DD'\to \DD$, and a cyclic \Ai-functor $F': \CC' \to \DD'$ making the following diagram commute.
	\[ \begin{CD}
		\CC @>F>>  \DD \\
		 @A G_\CC AA    @A G_\DD AA \\
		 \CC'  @>F'>>  \DD' 
	\end{CD} \]
\end{thm}
\begin{proof}
	The Calabi-Yau structures determine SHIPs $\rho_\CC$ and $\rho_\DD$ by Corollary \ref{cor:CY_SHIP}. Moreover $[F^*\rho_\DD]=[\rho_\CC] \in H^\bullet(\Omega^{2,cl}(\CC),\partial)$. Applying Proposition \ref{prop:darboux} to $\rho_\CC$ and $\rho_\DD$ one obtains cyclic categories $\CC'$ and $\DD'$ and \Ai-isomorphisms $G_\CC: \CC'\to \CC$ and $G_\DD: \DD'\to \DD$. 
	
	Let $\widetilde{F}:= G_\DD^{-1}\circ F \circ G_\CC$, by assumption $[\widetilde{F}^*\langle -, - \rangle_{\DD'}]=[\langle -, - \rangle_{\CC'}]$. By the previous lemma there exists an \Ai-isomorphism $E: \CC' \to \CC'$ satisfying $E^*(\widetilde{F}^*\langle -, - \rangle_{\DD'})=\langle -, - \rangle_{\CC'}$. Now replace $G_\CC$ by $G_\CC \circ E $ and take $F':= \widetilde{F} \circ E$. With these choices we have the required commutative diagram and the condition $(F')^*\langle -, - \rangle_{\DD'}= \langle -, - \rangle_{\CC'}$, since $(\widetilde{F}\circ E)^*\rho=E^*\widetilde{F}^*\rho$. 
\end{proof}

\section{Unital Calabi--Yau structures}\label{sec:cy2}

We would like to have unital versions (meaning all the \Ai-categories and functors are unital) of Proposition \ref{prop:darboux} and Theorem \ref{thm:cyclic-model}. As we will see, this is not true in the later case. For this purpose we need to introduce an unital version of cyclic homology. Unlike the case of Hochschild (co)homology this is not isomorphic to the normalized or reduced versions~\cite[Section 2.2.12]{Lod}.

Throughout this chapter we assume all the \Ai-categories are minimal.

\subsection{Unital cyclic homology}

\begin{defn}
	Denote by $Q$ the subspace of $C^\lambda_\bullet(\CC)$ spanned by chains of the form $[\one_{X_0} \otimes x_1 \cdots \otimes x_n]$, for $n\geq 1$. This is a subcomplex of $C^\lambda_\bullet(\CC)$. We define $C^{\lambda, un}_\bullet(\CC)$ to be the quotient complex $\left. C^\lambda_\bullet(\CC)\middle/ Q \right.$. We will refer to the homology of this complex as \emph{unital cyclic homology} and denote it by $HC^{un}_\bullet(\CC)$.
\end{defn}

There are natural maps
\[ HH_\bullet(\CC) \stackrel{\pi}{\longrightarrow} HC_\bullet^\lambda(\CC) \to HC^{un}_\bullet(\CC).
\]
As before, an element $\phi \in HC^{un}_\bullet(\CC)^\vee$ will be called non-degenerate, if the induced element in $HH_\bullet(\CC)^\vee$ is non-degenerate. In this case, $\phi$ will be called a \emph{unital Calabi--Yau structure}.

Let $F:\CC \to \DD$ be an unital $A_\infty$ functor. The induced map $F_*$ on the cyclic complex preserves $Q$ and therefore it induces a chain map $F_*: C^{\lambda, un}_\bullet(\CC)\to C^{\lambda, un}_\bullet(\DD)$.

\begin{rmk}
	Loday \cite[Section 2.2.13]{Lod} introduces the notion of \emph{reduced cyclic homology} for unital, associative algebras. That notion is different from the one just introduced. For example, let $\CC:=\mathbb{K}$ be the ground field. Using the proposition below, it is easy to see that $HC^{un}_\bullet(\mathbb{K})\cong \mathbb{K}$, while the reduced cyclic homology of \cite{Lod} vanishes in this case.
\end{rmk}

In order to compare $HC^{un}_\bullet(\CC)$ and $HC^{\lambda}_\bullet(\CC)$ we will use the $u$-model for cyclic homology: $C_\bullet^+(\CC):=\left(C_\bullet^{red}(\CC)[u^{-1}], b+uB\right)$. 
We define $C_\bullet^{+,un}(\CC)$ as the cokernel of the inclusion
\begin{align}
	\bigoplus_X u^{-1}\mathbb{K}[u^{-1}] & \to C_\bullet^+(\CC) \\
	 (u^{-k})_X & \to u^{-k} \one_{X}.\nonumber
\end{align}
Note that this is a chain map (from a complex with trivial differential) since $\one_X$ is $(b+uB)$-closed on $C_\bullet^{red}(\CC)[u^{-1}]$.

\begin{prop}
The map $\Pi: C_\bullet^{+,un}(\CC) \to C^{\lambda, un}_\bullet(\CC)$ defined as
\begin{align}
	\Pi(u^{-k} x_0\otimes \cdots \otimes x_n)&=\left\{
	\begin{array}{ll}
		0 & k>0\\
		\left[x_0\otimes \cdots \otimes x_n\right] & k=0
	\end{array}\nonumber
	\right.
	\end{align}
		is a  quasi-isomorphism.
\end{prop}
\begin{proof}
	We first observe that, by definition of $Q$, $\Pi$ is a well-defined map on $C_\bullet^{red}(\CC)$. Moreover $\Pi$ is a chain map since $\Pi(\im(B))\subseteq Q$.
	Now consider the (non-negative) increasing filtrations 
	\[\mathcal{F}^pC^{\lambda, un}_\bullet(\CC):= {\sf span}\left\{ [x_0\otimes \cdots \otimes x_n] \  | \ n\leq p \right\}
	\]
		\[\widetilde{\mathcal{F}}^pC^{+, un}_\bullet(\CC):= {\sf span}\left\{ u^{-k} x_0\otimes \cdots \otimes x_n \  | \ n\leq p \right\} \oplus {\sf span}\left\{ u^{-k} \one_{X_n}\otimes x_0\otimes \cdots \otimes x_n \  | \ n=p \right\}
	\]
	Again by definition of $Q$ we see that $\Pi$ preserves the above filtrations. We claim $\Pi$ induces quasi-isomorphisms on the associated graded pieces, which implies the statement. For $p=0$ we have $\mathcal{F}^0= \oplus_X \CC(X,X)$ and $\widetilde{\mathcal{F}}^0$ is given by the complex 
	\begin{align*}
		\oplus_X \ \CC(X, X) \stackrel{b}{\longleftarrow} \one_X \otimes \overline{\CC(X, X)} \stackrel{uB}{\longleftarrow} u^{-1} \overline{\CC(X, X)} \stackrel{b}{\longleftarrow} u^{-1}\one_X \otimes \overline{\CC(X, X)} \stackrel{uB}{\longleftarrow} u^{-2}\overline{\CC(X, X)} \cdots
	\end{align*}
where $\overline{\CC(X, X)}= \left. \CC(X, X)\middle/ {\sf span}\{\one_X\}\right.$. The above complex is isomorphic to 
	\begin{align*}
	\oplus_X \ \CC(X, X) \stackrel{0}{\longleftarrow} \overline{\CC(X, X)} \stackrel{\id}{\longleftarrow} \overline{\CC(X, X)} \stackrel{0}{\longleftarrow} \overline{\CC(X, X)} \stackrel{\id}{\longleftarrow} \overline{\CC(X, X)} \cdots.
	\end{align*}
Hence it is quasi-isomorphic to $\mathcal{F}^0$ via $\Pi$. For $p>0$ we have $$\left.\mathcal{F}^p \middle/ \mathcal{F}^{p-1} \right. \cong \bigoplus_{X_0, \ldots, X_p} \left.\overline{\CC(X_p, X_0, \ldots, X_p)}\middle/ \im(\id -t) \right.,$$
where, analogously, $\overline{\CC(X_p, X_0, \ldots, X_p)}$ is the quotient of $\CC(X_p, X_0, \ldots, X_p)$ by the span of tensors with a unit. Similarly to the $p=0$ case we have
	\begin{align*}
		\left.\widetilde{\mathcal{F}}^p \middle/ \widetilde{\mathcal{F}}^{p-1} \right. \cong & \bigoplus_{X_0, \ldots, X_p}  \overline{\CC(X_p, X_0, \ldots, X_p)} \stackrel{b}{\longleftarrow} \one_{X_p} \otimes \overline{\CC(X_p, X_0, \ldots, X_p)} \stackrel{uB}{\longleftarrow} \\
		&  \hspace{1.1cm} u^{-1} \overline{\CC(X_p, X_0, \ldots, X_p)} \stackrel{b}{\longleftarrow} u^{-1} \one_{X_p} \otimes \overline{\CC(X_p, X_0, \ldots, X_p)} \stackrel{uB}{\longleftarrow} \cdots \\
		\cong & \bigoplus_{X_0, \ldots, X_p}  \overline{\CC(X_p, X_0, \ldots, X_p)} \stackrel{\id -t}{\longleftarrow} \overline{\CC(X_p, X_0, \ldots, X_p)} \stackrel{N}{\longleftarrow}\\
		& \hspace{1.1cm} \overline{\CC(X_p, X_0, \ldots, X_p)} \stackrel{\id -t}{\longleftarrow} \overline{\CC(X_p, X_0, \ldots, X_p)} \stackrel{N}{\longleftarrow} \cdots
\end{align*}
since $b(\one_{X_n}\otimes x_0\otimes \cdots \otimes x_p)= (\id-t)(x_0\otimes \cdots \otimes x_p)+ \widetilde{\mathcal{F}}^{p-1}$. Therefore, as follows from Proposition \ref{prop:exac2}, this is a resolution of $\bigoplus_{X_0, \ldots, X_p} \left.\overline{\CC(X_p, X_0, \ldots, X_p)}\middle/ \im(\id -t) \right.$ which proves the claim.
\end{proof}

It now follows we have two different chain models for computing the unital cyclic homology. Using $C_\bullet^{+,un}(\CC)$ we immediately obtain:
		
\begin{cor}\label{cor:exa_uni}
	We have the following exact sequence
	\begin{align*}
		0 \to HC^{+}_{odd}(\CC) \to HC^{un}_{odd}(\CC) \to \bigoplus_X u^{-1}\mathbb{K}[u^{-1}] \to HC^{+}_{even}(\CC) \to HC^{un}_{even}(\CC) \to 0.
	\end{align*}
\end{cor}	

Using $C_\bullet^{\lambda,un}(\CC)$ we can relate $HC^{un}_\bullet(\CC)$ to an unital version of $\Omega^{2,cl}(\CC)$.

\begin{defn}
	An element $\rho \in \Omega^{2,cl}(\CC)$ is called unital if 
	\[\rho(\textbf{x}^{(1)}, \one_{X_i}, \textbf{x}^{(2)},\underline{v}, \textbf{y})(w)=0,\]
	for all $\textbf{x}^{(1)}, \textbf{x}^{(2)}, v, \textbf{y}, w$. By anti-symmetry, this also implies $\rho(\textbf{x},\underline{v}, \textbf{y}^{(1)}, \one_{Y_j}, \textbf{y}^{(2)})(w)=0$. We denote by $\Omega^{2,cl}_{un}(\CC)$ the subset of unital pre-homomorphisms. It is easy to check that $\Omega^{2,cl}_{un}(\CC)$ is a subcomplex of $\Omega^{2,cl}(\CC)$. 
	
	A quasi-isomorphism $\rho \in \Omega^{2,cl}_{un}(\CC)$  will be called an unital strong homotopy inner product (unital SHIP).
	
\end{defn}

We will now prove unital versions of Propositions \ref{prop:exact1} and \ref{prop:exac2}. First we need to modify the domain of the map $S$ (Lemma \ref{lem:Ssurjective}) for  $\Omega^{2,cl}_{un}(\CC)$. We introduce the sub-complexes
\[I:= {\sf span}\left\{ x_1 \otimes \cdots \otimes x_{i} \otimes \one_{X_{i+1}} \otimes x_{i+1} \otimes \cdots \otimes x_n \  | \ n\geq 2, 1\leq i \leq n-1 \right\} \subseteq C_\bullet(\CC)
\]
and $U:= (\id -t)(I)$. One can check that $U$ is a subcomplex of $C_\bullet(\CC)$.

\begin{lem}\label{lem:SESun}
	There are short exact sequences of complexes
	$$0 \to C^\lambda_\bullet(\CC)^\vee  \stackrel{\pi^\vee_U }{\longrightarrow} \big(\left.C_\bullet(\CC)\middle/U\right.\big)^\vee  \stackrel{S}{\longrightarrow} \Omega^{2,cl}_{un}(\CC)[1] \to 0,
	$$
	
	$$ 0 \to \left.C^\lambda_\bullet(\CC)\middle/N^{-1}(I)\right. \stackrel{N}{\longrightarrow} \left.\CC^{bar}\middle/I\right. \stackrel{p}{\longrightarrow} \left.\mathfrak{K}_\CC\middle/U\right. \to 0. $$
	
	Moreover, these are functorial under unital $A_\infty$ functors, as in Propositions \ref{prop:funct1} and \ref{prop:funct2}.
\end{lem}
\begin{proof}
	The proof follows closely the proofs of Propositions \ref{prop:exact1} and \ref{prop:exac2}, with only minor modifications. For the first sequence, we observe that $\pi^\vee_U$ is injective, since $U\subset Im(\id -t)$, by definition. Additionally, a straightforward check shows that, if $\rho \in \Omega^{2,cl}_{un}(\CC)$, then $h(\rho)$ (see proof of Proposition \ref{prop:exact1}) vanishes on $U$, which  proves $S$ is surjective.
	
	The second sequence follows directly from the non-unital case, since $p(I)=U$ by definition.	 
	
	It is clear that $F^*$ preserves unitality on $\Omega^{2,cl}$. We already saw that $F_*$ induces a map on the cyclic unital complex. Similarly one checks that $F_*'$ preserves $I$. This together with Equation (\ref{eq:F'F}) implies that $F_*$ preserves $U$. Functoriality now follows as in the non-unital setting.
\end{proof}

We now arrive at the unital version of Corollary \ref{cor:CY_SHIP}.

\begin{prop}\label{prop_un_iso}
	There is an isomorphism
	\[HC^{un}_\bullet(\CC)^\vee  \cong H^\bullet(\Omega^{2,cl}_{un}(\CC), \partial),\]
	which identifies unital Calabi--Yau structures with unital SHIPs.
\end{prop}
\begin{proof}
	The homotopy $h$ used in Lemma \ref{lem:baracyc} to prove $\CC^{bar}$ is acyclic preserves $I$ therefore we can use it to show acyclicity of $\left.\CC^{bar}\middle/I\right.$. The kernel of $\pi_U$ is $\left.\mathfrak{K}_\CC\middle/U\right.$. Combining these with the two short exact sequences in Lemma \ref{lem:SESun} we conclude that 
	\[H^\bullet(\Omega^{2,cl}_{un}(\CC), \partial) \cong H_\bullet(\left.C^\lambda_\bullet(\CC)\middle/N^{-1}(I)\right.)^\vee.
	\]
	
	We compute the later cohomology by considering the projection 
		$$q: \left.C^\lambda_\bullet(\CC)\middle/N^{-1}(I)\right. \to \left.C^\lambda_\bullet(\CC)\middle/Q\right.=C^{\lambda, un}_\bullet(\CC).$$ 
		
	By definition of $I$ we see that $N^{-1}(I)$ is spanned by tensors with at least three units, or two non-consecutive units. In particular this shows that $N^{-1}(I)\subseteq Q$, which in turn implies $q$ is well-defined, surjective and has kernel $\ker(q)\cong\left. Q \middle/ N^{-1}(I)\right. \subseteq C^\lambda_\bullet(\CC)$. 
	
	We claim that $\ker(q)$ is acyclic, which shows that $q$ is a quasi-isomorphism and in turn proves the statement.
	On $\ker(q)$ we consider the filtration
		\begin{align}
		\mathcal{F}^p\ker(q) & :=  {\sf span}\left\{ [\one_{X_1}\otimes x_1 \otimes \cdots \otimes x_n]  \in Q \  | \ n\leq p-1 \right\} \nonumber \\
		& \ \ \  \bigoplus {\sf span}\left\{ [\one_{X_1}\otimes \one_{X_1} \otimes x_1\otimes \cdots \otimes x_{p-1}] \in Q \right\}\nonumber
	\end{align}
	From the description of $N^{-1}(I)$ above, we have
	\begin{align}
		\left.\mathcal{F}^p \middle/ \mathcal{F}^{p-1} \right. \cong \one_{X_{p-1}} \otimes \overline{\CC(X_{p-1}, X_1, \ldots, X_{p-1})} 
	 \bigoplus \one_{X_{p-1}}\otimes \one_{X_{p-1}} \otimes \overline{\CC(X_{p-1}, X_1, \ldots, X_{p-1})} \nonumber
	\end{align}
	The induced differential is non-zero only on the second summand, where it equals
	\[b([\one_{X_{p-1}}\otimes \one_{X_{p-1}} \otimes x_1\otimes \cdots \otimes x_{p-1}])= - [\one_{X_{p-1}}\otimes x_1\otimes \cdots \otimes x_{p-1}] + [\one_{X_{p-1}} \otimes \one_{X_{p-1}}\otimes b'(x_1\otimes \cdots \otimes x_{p-1})].
	\]
	By definition the second term is in $\mathcal{F}^{p-1} $, thus we conclude $\left.\mathcal{F}^p \middle/ \mathcal{F}^{p-1} \right. $ is acyclic. Hence $\ker(q)$ is acyclic.
	
	The last statement has the same proof as the non-unital case.
\end{proof}
	
\subsection{Unital Darboux theorem}
In this subsection we prove the unital version of the Darboux theorem. We will also see that the analog of the Moser stability theorem is only partially true in the unital setting.

\begin{lem}\label{lem:unit}
  Let $\rho \in \Omega^{2,cl}_{un}(\CC)$ be an unital pre-homomorphism.
  \begin{enumerate}
  	\item If $\rho$ is a quasi-isomorphism then the map $\iota_{-}\rho: C^\bullet_{red}(\CC) \to (C^{red}_\bullet(\CC))^\vee$ is an isomorphism. 
  	\item Assume $\rho_{0,0}(\one_X)(\one_X)=0$ for all $X$, then there is $\theta \in (C^{red}_\bullet(\CC))^\vee$ satisfying $S(\theta)=\rho$.
  \end{enumerate}
\end{lem}
\begin{proof}
	From Definition \ref{defn:contraction} it is clear that if $Z$ is reduced and $\rho$ is unital then $\iota_Z \rho$ is in $(C^{red}_\bullet(\CC))^\vee$. Moreover, close inspection shows that the proof of Lemma \ref{lem:contraction}(1) still applies in the reduced setting.
	
	For the second claim, we first observe that the condition on $\rho_{0,0}$ plus the fact that $\rho$ is closed and unital implies that any expression of the form \begin{equation}\label{eq:rhovanish}
		\rho(\textbf{x},\underline{\one_X}, \textbf{y})(\one_Y)
	\end{equation}
 vanishes. Next, one picks basis for the hom spaces, which in the case of self-homs extend the units. Then one defines $\theta(x_0, x_1, \ldots, x_n)$ on the basis elements as follows: when none of the $x_i$ are units one uses $h(\rho)$ as in Equation (\ref{eq:Ssurj}); when one or more of the $x_i$ (for $i>0$) is a unit then $\theta$ must be zero, for it to be reduced; and when $x_0=\one_X$ one takes 
 \[\theta(\one_X, x_1, \ldots, x_n):=\rho(\underline{\one_X},  x_1, \ldots, x_{n-1} )(x_n).\]
	Equation (\ref{eq:rhovanish}) guarantees these conditions are compatible - the above expression vanishes if $x_n$ is a unit. We claim that $S(\theta)=\rho$. This follows from Lemma \ref{lem:Ssurjective}, when none of $x_i$ are units. The only non-trivial case left is when $x_0=\one_X$, in which case we have:
	\begin{align*}
		S(\theta)_{r,s}(\textbf{x},\underline{\one_X}, \textbf{y})(w) & =  (-1)^@\rho(\underline{\one_X}, \textbf{y},w, \textbf{x}^{(1)})(x_r)\\
		& = (-1)^{@+1}\rho( \textbf{y}, \underline{w},\textbf{x}^{(1)}, x_r)(\one_X)+ (-1)^{@+1}\rho(\textbf{x}^{(1)}, \underline{x_r},\one_X,\textbf{y})(w)\\
		& = \rho(\textbf{x},\underline{\one_X}, \textbf{y})(w).
	\end{align*}
Here we used closedness of $\rho$ in the second equality and unitality and anti-symmetry in the third. This takes care of the case $r>0$, the remaining cases follow directly from the definition.
\end{proof}

\begin{prop}\label{prop:norm_darboux}
	Let $\rho \in \Omega^{2,cl}_{un}(\CC)$ be an unital SHIP. There exists a cyclic, unital \Ai-category $\CC'$, with the same objects, morphism spaces and units as $\CC$, and an unital \Ai-isomorphism $F: \CC' \to \CC$. Moreover, if we denote by $\langle -, - \rangle$ the cyclic pairing on $\CC'$, we have $F^*\rho=\langle -, - \rangle$.
\end{prop}
\begin{proof}
	The proof follows the same argument as Proposition \ref{prop:darboux}, just taking the unital version of every ingredient. Since $\rho$ is unital so is $\rho^t$ and $\rho-\rho^0$. Since $(\rho-\rho^0)_{0,0}=0$, we can apply the previous lemma to obtain a reduced $\theta$ with $S(\theta)=\rho-\rho^0$. Moreover, we can take $\theta$ satisfying $\theta_0=\theta_1=0$. By Lemma \ref{lem:unit}(1) there is a reduced family of cochains $Z^t$ satisfying $\iota_{Z^t} \rho^t= \theta$. 
	For reduced $Z^t$, the family $F^t$ provided by Lemma \ref{lem:flow} are also unital.
	Then it is easy to check that the new \Ai \ structure $\m'$ is unital (for the same units $\one_{X_i}$) and $F^1$ is an unital \Ai-isomorphism.
\end{proof}

Next we prove our version of the Moser stability theorem for unital SHIPs of \emph{even degree}.

\begin{lem}\label{lem:unit_coho}
	If $\rho^0, \rho^1$ are unital cohomologous SHIPs in $\CC$ of \emph{even degree}, then there is an unital \Ai-isomorphism $F: \CC \to \CC$ with $F^*\rho^1=\rho^0$.
\end{lem}
\begin{proof}
By assumption $\rho^1-\rho^0= \partial \beta$ for some $\beta \in \Omega^{2,cl}_{un}(\CC)$ of odd degree. Since $\beta$ has odd degree, it satisfies the condition in Lemma \ref{lem:unit}(2), which then gives $\theta\in (C^{red}_\bullet(\CC))^\vee$ satisfying $S(\theta)=\beta$. The remainder of the proof is identical to Lemma \ref{lem:homologous}.
\end{proof}

The following example shows how the above lemma fails in the odd case.

\begin{ex}\label{ex:non-existence}
	Let ${\sf Cl}$ be the (one-dimensional) Clifford algebra:  a $\mathbb{Z}/2$-graded \Ai-algebra with underlying vector space $\mathbb{K}[\epsilon]$, where $\epsilon$ is of odd degree, which is strictly unital and the only non-trivial \Ai \ operation is $\m_2(\epsilon, \epsilon)= \frac{1}{2}\one$. Moreover it is cyclic (of odd degree) with the pairing given by $\langle \one, \epsilon \rangle = 1$.
	
	Let ${\sf A}$ denote the cyclic, unital \Ai-algebra with the same vector space and pairing as ${\sf Cl}$, but with \Ai \ operations $\m_k(\epsilon, \ldots, \epsilon)= \frac{1}{2}C_{k-1} \one$, for $k\geq 2$, where $C_n$ is the $n$-th Catalan number. Recall, $C_0=1$ and $C_n=\sum_{i=1}^n C_{i-1}C_{n-i}$.
	
	There is an unital \Ai-isomorphism $G: {\sf A} \to {\sf Cl}$, defined by the formula $G_k(\epsilon, \ldots, \epsilon)= C_{k-1} \epsilon$, $k\geq 1$. Define the unital cohomologous strong homotopy inner product $\rho:=G^*\langle -, - \rangle$ in ${\sf A}$. One can check that $[\rho]=[\langle -, - \rangle ]$, explicitly 
	$$\rho=G^*\langle -, - \rangle = \langle -, - \rangle + \partial G^* \beta,$$
	where $\beta\in \Omega^{2,cl}_{un}({\sf Cl})$ is defined by $\beta(\one)(\one)=2$. By  Lemma \ref{lem:homologous} there is an \Ai-isomorphism $F: {\sf A} \to {\sf A}$ satisfying $F^*\rho=\langle -, - \rangle$. 
    
    We claim that there is no \emph{unital} \Ai-isomorphism $F: {\sf A} \to {\sf A}$ satisfying $F^*\rho=\langle -, - \rangle$. Given such $F$, then $G\circ F: {\sf A} \to {\sf Cl}$ is an unital, cyclic \Ai-isomorphism, that is $(G\circ F)^*\langle -, - \rangle = \langle -, - \rangle$. By unitality $G\circ F$ is determined by $(G\circ F)_k(\epsilon, \ldots, \epsilon)$, but it easy to check that cyclicity implies $(G\circ F)_k(\epsilon, \ldots, \epsilon)=0$  for $k\geq 2$, which does not satisfy the \Ai \ equations.
\end{ex}

We are now ready to state the unital versions of Theorem \ref{thm:cyclic-model}. The situation is different in the odd and even cases.

\begin{prop}~\label{prop:cyclic-odd}
	Let $\CC$ be an \Ai-category with an odd strong Calabi--Yau structure $\phi$. Then there exists a cyclic, unital \Ai-category $\CC'$ and an unital \Ai-isomorphism $G_\CC: \CC'\to \CC$.
\end{prop}
\begin{proof}
It follows from Corollary \ref{cor:exa_uni} that the odd Calabi--Yau structure $\phi\in HC^{\lambda}_1(\CC)^\vee$ lifts to an element $\phi^{un} \in HC_1^{un}(\CC)^\vee$ , meaning the natural map $HC_1^{un}(\CC)^\vee \to HC^{\lambda}_1(\CC)^\vee$ sends $\phi^{un}$ to $\phi$. Note the lift will, in general, not be unique. Since $\phi$ is non-degenerate, so will be the lift $\phi^{un}$. Therefore $\phi^{un}$ is an unital Calabi--Yau structures and hence determines an unital SHIP by Proposition \ref{prop_un_iso}. Applying Proposition \ref{prop:norm_darboux} to this unital SHIP gives the desired result.
\end{proof}

\begin{prop}~\label{prop:cyclic-even}
	Let $\CC$ and $\DD$ be \Ai-categories with even Calabi--Yau structures $\phi_\CC$ and $\phi_\DD$, respectively, which have unital lifts. Let $F:\CC \to \DD$ be an unital \Ai-functor with $[F^*\phi_\DD]=[\phi_\CC]$. Then there exist cyclic, unital \Ai-categories $\CC'$ and $\DD'$; unital \Ai-isomorphisms $G_\CC: \CC'\to \CC$ and $G_\DD: \DD'\to \DD$; and a cyclic, unital \Ai-functor $F': \CC' \to \DD'$ making the following diagram commute
		\[ \begin{CD}
		\CC @>F>>  \DD \\
		@A G_\CC AA    @A G_\DD AA \\
		\CC'  @>F'>>  \DD' 
	\end{CD} \]
\end{prop}
\begin{proof}
	 Let $\phi_\CC^{un}$ and $\phi_\DD^{un}$ be the unital lifts of $\phi_\CC$ and $\phi_\DD$, respectively. Corollary \ref{cor:exa_uni} implies these lifts are unique (since they are of even degree) which in turn implies $[F^*(\phi_\DD^{un})]= [\phi_\CC^{un}] $.  These lifts determine  unital strong homotopy inner products $\rho_\CC$ and $\rho_\DD$, by Proposition \ref{prop_un_iso}. One then applies Proposition \ref{prop:norm_darboux}, to obtain unital $\CC', \DD'$, $G_\CC$ and $G_\DD$. Next one observes that $\widetilde{F}^*\langle -, - \rangle_{\DD'}$ is an unital SHIP and $[\widetilde{F}^*\langle -, - \rangle_{\DD'}]=[\langle -, - \rangle_{\CC'}]$ (in the notation of Theorem \ref{thm:cyclic-model}). Now, using Lemma \ref{lem:unit_coho} we can proceed as in the proof of Theorem \ref{thm:cyclic-model}.
\end{proof}

\subsection{Splittings and unitality}\label{sec:trivCY}

In order to define CEI we need to choose a splitting of the nc-Hodge filtration. In this subsection we recall how a splitting together with a weak Calabi--Yau structure determine a strong Calabi--Yau structure. We also introduce a unitality condition on splittings of the nc-Hodge filtration.

We first review some preliminary notions, following \cite{AT}. The chain level Mukai pairing on $C_\bullet^{red}(\CC)$ descends to a pairing which we still denote by
\[ \langle -,-\rangle_{\sf Muk} :HH_\bullet(\CC)\otimes HH_\bullet(\CC) \to \mathbb{K}.\]
It is non-degenerate when $\CC$ is smooth and proper~\cite{Shk}. One can also extend $\langle -,-\rangle_{\sf Muk}$ sesquilinearly to the negative cyclic complex $CC_\bullet^{red}(\CC)[[u]]$ to obtain the so-called higher residue pairing
\[ \langle \alpha u^i,\beta u^j\rangle_{\sf hres}:=(-1)^j  \langle \alpha,\beta\rangle_{\sf Muk} \cdot u^{i+j}. \]
Taking the $(b+uB)$-homology yields the higher residue pairing in homology
\[ \langle -,-\rangle_{\sf hres}: HC_\bullet^-(\CC) \otimes HC_\bullet^-(\CC) \ra \mathbb{K}[[u]].\]
The following definition is taken from~\cite[Definition 3.7]{AT}, except the last unitality condition which is new.
\begin{defn}~\label{defi:splitting}
	A $\mathbb{K}$-linear map  $s: HH_\bullet(\CC) \to HC_\bullet^-(\CC)$ is called a splitting of the Hodge filtration of $\CC$ if it satisfies
	\begin{itemize}
		\item[S1.] {{\sf (Splitting)}} $s$ splits the canonical map $\pi: HC_\bullet^-(\CC) \ra HH_\bullet(\CC)$,  $\pi(\sum_{n\geq0} \alpha_n u^n)=\alpha_0$.
		\item[S2.] {{\sf (Lagrangian)}} $ \langle s(\alpha), s(\beta)\rangle_{{\sf hres}} = \langle \alpha,\beta\rangle_{{\sf Muk}}, \;\; \forall \alpha, \beta \in HH_\bullet(\CC)$. 
	\end{itemize}
	A splitting $s$ is called {\sl good} if it satisfies
	\begin{itemize}
		\item[S3.] {{\sf (Homogeneity)}} The subspace $\displaystyle\bigoplus_{l\in \mathbb{N}} u^{-l}\cdot {{\sf Im}} (s)$  is stable under the $u$-connection $\nabla_{u\frac{d}{du}}$. This is equivalent to requiring $ \nabla_{u\frac{d}{du}} s(\alpha) \in u^{-1} {{\sf Im}} (s) + {{\sf Im}}(s), \; \forall \alpha\in HH_\bullet(\CC)$.
	\end{itemize}
	Let $\omega\in HH_\bullet(\CC)$ be an element of the Hochschild homology. 	
	\begin{itemize}
		\item[S4.] {{\sf ($\omega$-Compatibility)}} A splitting $s$ is called {\sl $\omega$-compatible} if
$$\nabla_{u\frac{d}{du}} s(\omega) \in r\cdot s(\omega) +u^{-1}\cdot {{\sf Im}} (s) \;\; \mbox{ for some\;\;}  r\in \mathbb{K}.$$
	\item[S5.] {{\sf (Unitality)}} For $X \in \tw^\pi\CC$, let $[\one_X]$ be the (potentially zero) class in $HC_\bullet^-(\tw^\pi\CC)$ determined by the unit. A splitting $s$ is called {\sl unital } if there exists a split-generating subcategory $\mathcal{A}\subset \tw^\pi\CC$ such that we have
	\[ \langle [\one_X], s(\omega)\rangle_{{\sf hres}} \in \mathbb{K},\] 
	i.e. the higher residue pairing is a constant, for any object $X$ in $\mathcal{A}$.
	\end{itemize}
\end{defn}

\begin{rmk}
	The unitality condition is empty when $\omega$ is an odd class, since the higher residue pairing is even. Furthermore, the unitality condition is for a split-generating subcategory, instead of $\CC$ itself, so that this condition is stable under Morita equivalences.
\end{rmk}

As mentioned above, Shklyarov~\cite{Shk} proved that the Mukai pairing is non-degenerate, which gives a natural isomorphism
\[HH_\bullet (\CC)^\vee  \cong HH_\bullet(\CC).\]
Moreover it gives an isomorphism $\displaystyle HC_\bullet^+(\CC)^\vee \cong HC_\bullet^-(\CC)$ via the pairing
\[ \langle \alpha u^i,\beta u^j\rangle:=\delta_{i+j=0}\cdot \langle \alpha,\beta\rangle_{\sf Muk} \]
Using the above isomorphisms, we can identify elements of $HH_\bullet(\CC)$ (respectively $HC_\bullet^-(\CC)$) 
satisfying the same non-degeneracy condition with weak 
(respectively strong)
Calabi--Yau structures.

Let $\omega\in HH_\bullet(\CC)$ be a weak Calabi--Yau structure and $s$ a splitting.  By construction, the element $s(\omega) \in HC_\bullet^-(\CC)$ (via the identification above) is non-degenerate and therefore defines a strong Calabi--Yau structure. We have the following

\begin{lem}
	Let $\omega\in HH_\bullet(\CC)$ be a weak Calabi--Yau structure and $s$ an unital splitting with respect to $\mathcal{A}\subset \tw^\pi \CC$. Then the strong Calabi--Yau structure $s(\omega)$, viewed as an element of $HC_\bullet^-(\mathcal{A})$, admits a lift to an unital Calabi--Yau structure of $\mathcal{A}$.
\end{lem}
\begin{proof}
	From the definition of the unital cyclic complex $CC_\bullet^{+,un}(\CC)$, an element $\phi \in HC_\bullet^+(\CC)^\vee$ lifts to $HC_\bullet^{un}(\CC)^\vee$ if and only if $\phi(u^{-k}[\one_X])=0$ for any object $X$ and $k>0$. When $\phi= \langle - , s(\omega)\rangle$, this is exactly the Unitality condition for $s$.
\end{proof}
 
 Since the Unitality condition is new, we end this section by providing a few examples where it is implied by the more conventional condition of $\omega$-compatibility.
 
 \begin{prop}\label{prop:comp_unital}
 	Let $s$ be a $\omega$-compatible splitting for some $\omega\in HH_\bullet(\CC)$, with $r \notin \mathbb{Z}_{>0}$. Assume that $\CC$ satisfies one of the following conditions:
 	\begin{enumerate}
 		\item $\CC$ is a $\mathbb{Z}$-graded \Ai-category;
            \item  The $u$-connection on periodic cyclic homology has a simple pole at $u=0$;
 		\item $\CC$ has semi-simple Hochschild cohomology and $s$ is the \emph{semi-simple splitting}, as defined in \cite[Corollary 3.8]{AT}.
 	\end{enumerate}
 Then $s$ is unital with respect to $\CC$ itself.
 \end{prop}
 \begin{proof}
 	When $\CC$ is $\mathbb{Z}$-graded, the $u$-connection has a simple pole at $u=0$ (instead of a pole of order two) as proved in \cite[Lemma 3.2]{CLT}. Therefore, (1.) is a special case of (2.).  In this case, $\omega$-compatibility takes the form $\nabla_{u\frac{d}{du}} s(\omega)= r \cdot s(\omega)$. We compute
 	\begin{align}
 		u \frac{d}{d u}\langle \one_X, s(\omega) \rangle_{{\sf hres}} & = \langle \nabla_{u \frac{d}{d u}} \one_X, s(\omega) \rangle_{{\sf hres}} + \langle \one_X, \nabla_{u \frac{d}{d u}} s(\omega) \rangle_{{\sf hres}}\nonumber\\
 		&= r \langle \one_X, s(\omega) \rangle_{{\sf hres}},\nonumber
 	\end{align}
 since $\nabla_{u \frac{d}{d u}} \one_X =0$. It is elementary that, when $r \notin \mathbb{Z}_{>0}$, the only possible solutions of the differential equation $u \frac{d}{d u} f(u)= r f(u)$ are constants. Therefore we conclude that $\langle \one_X, s(\omega) \rangle_{{\sf hres}}$ is constant.

 For the second condition, we freely use the notation from  \cite{AT}. Recall $\CC$ has the decomposition $\CC=\Pi_i \CC_{\lambda_i}$, and consider $X\in \CC_{\lambda_l}$. A simple computation shows $\nabla_{u \frac{d}{d u}} \one_X= u^{-1} \lambda_l \one_X$. Under the decomposition above we write $\omega= \sum_i \omega_i \in \oplus_i HH_\bullet(\CC_{\lambda_i})$. Then $\omega$-compatibility gives $\nabla_{u \frac{d}{d u}} s(\omega)= r\cdot s(\omega)+u^{-1}s(\xi(\omega))$, where $\xi(\omega)=\sum_i \lambda_i \omega_i$. As above we compute
 \begin{align}
 	u \frac{d}{d u}\langle \one_X, s(\omega) \rangle_{{\sf hres}} & = u^{-1}\langle \lambda_l \one_X, s(\omega) \rangle_{{\sf hres}} + \langle \one_X, r s(\omega) + u^{-1} s(\sum \lambda_i \omega_i) \rangle_{{\sf hres}}\nonumber\\
 	& = u^{-1}\lambda_l \langle \one_X, s(\omega) \rangle_{{\sf hres}}+ r \langle \one_X, s(\omega) \rangle_{{\sf hres}} - u^{-1} \langle \one_X, \sum \lambda_i s(\omega_i) \rangle_{{\sf hres}}\nonumber\\
 	& = u^{-1}\lambda_l \langle \one_X, s(\omega_l) \rangle_{{\sf hres}}+ r \langle \one_X, s(\omega) \rangle_{{\sf hres}} - u^{-1} \lambda_l \langle \one_X, s(\omega_l) \rangle_{{\sf hres}}\nonumber\\
 	& = r \langle \one_X, s(\omega) \rangle_{{\sf hres}}.\nonumber
 \end{align}
Here in the third and fourth equality rely on the fact that $s$ is the semi-simple splitting and therefore respects the orthogonal decomposition  $\CC=\Pi_i \CC_{\lambda_i}$. As before we can now conclude that $\langle \one_X, s(\omega) \rangle_{{\sf hres}}$ is constant.	
 \end{proof}

\section{Definition of categorical enumerative invariants}\label{sec:defi}

We briefly recall the construction of categorical enumerative invariants following~\cite{Cos2,CT}. In~\cite{CT}, the construction is done for a finite dimensional cyclic $A_\infty$-algebra. As in~\cite{Cos2}, we shall work with cyclic $A_\infty$-categories. Observe that by requiring morphisms be composable in the definition of the Hochschild chain complex of $\CC$ (see Section~\ref{sec:cy1}), the constructions in~\cite{CT} generalize to the categorical setting.

Throughout the section, let $\CC$ be a $\Z/2\Z$-graded cyclic  $A_\infty$-category over a field $\mathbb{K}$ of characteristic zero. Note that in particular, the non-degeneracy condition of the cyclic pairing implies that the Hom spaces of $\CC$ are finite dimensional. We also assume that $\CC$ is strictly unital and its cyclic pairing is of parity $d\in \Z/2\Z$.

 In this chapter we will use homological gradings, for consistency with \cite{CT}. This can be made compatible with the previous chapters (where cohomological grading was used) by turning a homological grading $V=\oplus_n V_n$ into the cohomological one $V=\oplus_n V^{-n}$.


\subsection{TCFT's and DGLA's}\label{subsec:tcft-dgla} Let $M_{g,k,l}^{\sf fr}$ be the moduli space of smooth Riemann surfaces of genus $g$, with $k$ incoming framed marked points, $l$ outgoing framed marked points. Here a framing of a marked point means a choice of an embedded disk around the marked point. We also assume the framings along all $k+l$ marked points be disjoint from each other, i.e. the closures of all the embedded disks be disjoint. Denote by $C^{\sf comb}_\bullet(M_{g,k,l}^{\sf fr})$ the combinatorial model of these moduli spaces with coefficients in $\mathbb{K}$, as described in Costello~\cite{Cos1}, Kontsevich-Soibelman~\cite{KS}, and Wahl-Westerland~\cite{WahWes}.

\begin{thm}~\label{thm:tcft}
	Let $C^{red}_\bullet(\CC)[d]$ denote the $d$-shifted reduced Hochschild chain complex of the $A_\infty$-category $\CC$. Then it carries a $2$-dimensional topological conformal field theory structure, i.e. there are action maps compatible with sewing of Riemann surfaces
	\[ \rho^\CC_{g,k,l}: C^{\sf comb}_\bullet(M_{g,k,l}^{\sf fr}) \ra {\sf Hom} ( C_\bullet^{red}(\CC)[d]^{\otimes k}, C_\bullet^{red}(\CC)[d]^{\otimes l})\]
	with $g\geq 0$, $k\geq 1$, $l\geq 0$, and $2-2g-k-l<0$. 
\end{thm}

To obtain invariants from $\rho^\CC_{g,k,l}$'s one needs to overcome the difficulty that the moduli spaces $M_{g,k,l}^{\sf fr}$ are non-compact since we are only dealing with smooth surfaces. One natural approach to this problem is to extend this action to the Deligne-Mumford compactification, assuming furthermore that $\CC$ satisfies the Hodge-to-de-Rham degeneration property (which follows from smoothness by Kaledin~\cite{Kal}). This was conjectured by Kontsevich-Soibelman~\cite{KS}.

Following Costello, in~\cite{CT} we take a different (but still closely related) route which completely bypasses the Deligne-Mumford compactification. One of the key ideas is due to
Sen-Zwiebach~\cite{SenZwi} who introduced a differential graded Lie algebra (DGLA) structure on the normalized singular chains of moduli spaces of smooth curves. This idea is further implemented in the combinatorial setup in~\cite{CCT}. 

Indeed, for an integer $k\geq 1$, denote by $\underline{\sgn_k}[-k]$ the rank one local system over $M_{g,k,l}^\fr$ whose fiber over a Riemann surface $(\Sigma, p_1,\ldots,p_k,q_1,\ldots,q_l,\phi_1,\ldots,\phi_k,\psi_1,\ldots,\psi_l)$ is the sign representation of the symmetric group $S_k$ on the set $\{ p_1,\ldots, p_k\}$, shifted by $[-k]$. For simplicity, we often write this local system as $\underline{\sgn}$ when the integer $k$ is clear from the context. The combinatorial version of Sen-Zwiebach's DGLA is given by
\[ \widehat{\mathfrak{g}}:= \bigoplus_{\substack{g\geq 0, k\geq 1, l\geq 0\\ 2g-2+k+l>0}} C^{\sf comb}_\bullet(M_{g,k,l}^{\sf fr}, \underline{\sgn})_{\sf hS}[2][[\hbar,\lambda]],\]
where the subscript ${\sf hS}$ denotes the homotopy quotient by the symmetric group $\Sigma_k\times \Sigma_l$ and the $k+l$ circle actions that rotate the framings; and the formal variables $\hbar$ and $\lambda$ are both of homological degree $-2$. The additional shift by $[2]$ is explained in the remark following Theorem~\ref{thm:comb-mc} below.

In the construction of Sen-Zwiebach's DGLA structure, we need to perform the ``twisted sewing operation" where the moduli space of annuli plays an important role. Indeed, let us consider the Mukai ribbon graph
\begin{equation}~\label{eq:Mukai-graph}
M:= \begin{tikzpicture}[baseline={([yshift=-0.4ex]current bounding box.center)},scale=0.3]
\draw [thick] (0,2) circle [radius=2];
\draw [thick] (-2,2) to (-0.6,2);
\draw [thick] (-0.8,2.2) to (-0.4,1.8);
\draw [thick] (-0.8, 1.8) to (-0.4, 2.2);
\draw [thick] (2,2) to (3.4,2);
\draw [thick] (3.2,2.2) to (3.6,1.8);
\draw [thick] (3.2, 1.8) to (3.6, 2.2);
\end{tikzpicture}
\end{equation}
which is a zero chain inside $C_*^{\sf comb}(M_{0,2,0}^{\sf fr})$. 
Here, the two crosses in $M$ correspond to its $2$ inputs (cycles of the ribbon graph $M$). Using $M$ we may define a map 
\[\iota:C^{\sf comb}_\bullet(M_{g,k,l}^{\sf fr}, \underline{\sgn})_{\sf hS} \; \ra \; C^{\sf comb}_\bullet(M_{g,k+1,l-1}^{\sf fr}, \underline{\sgn})_{\sf hS},\]
obtained by sewing with one input of $M$ with an output of a chain in $C^{\sf comb}_\bullet(M_{g,k,l}^{\sf fr})_{\sf hS}$, hence this operation reduces the number of outputs by one while increases the number of inputs also by one. There is a thickened version of $M$ which we shall denote by
\begin{equation}~\label{eq:Thick-Mukai-graph} \mathbb{M}:=\begin{tikzpicture}[baseline={([yshift=-0.4ex]current bounding box.center)},scale=0.3]
\draw [thick] (0,2) circle [radius=2];
\draw [thick] (2,2) to (-0.6,2);
\draw [thick] (-0.8,2.2) to (-0.4,1.8);
\draw [thick] (-0.8, 1.8) to (-0.4, 2.2);
\draw [thick] (2,2) to (3.4,2);
\draw [thick] (3.2,2.2) to (3.6,1.8);
\draw [thick] (3.2, 1.8) to (3.6, 2.2);
\end{tikzpicture}
\end{equation}
It is obtained from $M$ by a circle action from one of its two inputs. This is a degree one chain inside $C_*^{\sf comb}(M_{0,2,0}^{\sf fr})$. Sewing with $\mathbb{M}$ at outputs yields a map of homological degree one
\[ \Delta: C^{\sf comb}_\bullet(M_{g,k,l}^{\sf fr}, \underline{\sgn})_{\sf hS} \; \ra \; C^{\sf comb}_\bullet(M_{g,k,l-2}^{\sf fr}, \underline{\sgn})_{\sf hS},\]
called the {{\em  twisted self-sewing map}}. In a similar way, twisted sewing between $r\; (r\geq 1)$ outputs and $r$ inputs yields a map
\[ \{-,-\}_r : C^{\sf comb}_\bullet(M_{g',k',l'}^{\sf fr}, \underline{\sgn})_{\sf hS} \otimes C^{\sf comb}_\bullet(M_{g'',k'',l''}^{\sf fr}, \underline{\sgn})_{\sf hS} \; \ra \; C^{\sf comb}_\bullet(M_{g,k,l}^{\sf fr}, \underline{\sgn})_{\sf hS},\]
called the $r$-th twisted sewing map, where $g=g'+g''+r-1$, $k=k'+k''-r$, and $l=l'+l''-r$. We refer to~\cite{CCT} for more details (such as signs) of the these constructions .

\begin{thm}[Caldararu--Costello--Tu\cite{CCT}]\label{thm:comb-mc}
	There exists a $\Z$-graded DGLA structure on $\widehat{\mathfrak{g}}$ whose differential and Lie bracket are of the form
	\begin{itemize}
	\item $d := \eth +\iota+ \hbar \Delta$ with $\eth$ the boundary map of $C^{\sf comb}_\bullet(M_{g,k,l}^{\sf fr})_{\sf hS}$, and $\Delta$ is the twisted self-sewing operator. 
	\item $\{-,-\}_\hbar := \sum_{r\geq 1} \frac{1}{r!} \cdot \{-,-\}_r \hbar^{r-1}$ with $\{-,-\}_r$ the $r$-th twisted sewing map.
	\end{itemize}
	Furthermore, the DGLA $ \widehat{\mathfrak{g}}$ has a unique Maurer-Cartan element $\widehat{\mathcal{V}}$ (up to gauge equivalence) of homological degree $-1$, and of the form
	\begin{align*}
	\widehat{\mathcal{V}} & = \sum_{g,k,l} \widehat{\mathcal{V}}_{g,k,l} \hbar^g \lambda^{2g-2+k+l}\\
	\widehat{\mathcal{V}}_{0,1,2} & =  \begin{tikzpicture}[baseline={([yshift=-1.2ex]current bounding
      box.center)},scale=0.3] 
\draw [thick] (0,0) to (0,2);
\draw [thick] (-0.2, 1.8) to (0.2, 2.2);
\draw [thick] (0.2, 1.8) to (-0.2, 2.2);
\draw [thick] (0,0) to (-2,0);
\draw (-2,0) node[label=left:{$\frac{1}{2}\;$}] {};
\draw [thick] (0,0) to (2,0);
\draw [thick] (-2.2,0) circle [radius=0.2];
\draw [thick] (2.2,0) circle [radius=0.2];
\end{tikzpicture}
	\end{align*}
	We shall refer to $\widehat{\mathcal{V}}_{g,k,l}$ as combinatorial string vertices.
\end{thm}

\begin{rmk}
The $\Z$-grading of $\widehat{\mathfrak{g}}$ is designed so that the boundary map $d$ has homological degree $-1$, the Lie bracket $\{-,-\}_\hbar$ has degree $0$, and that the equivariant chain $\widehat{\mathcal{V}}_{g,k,l}$ is of homological degree $6g-5+2(k+l)$ in the shifted chain complex $C^{\sf comb}_\bullet(M_{g,k,l}^{\sf fr}, \underline{\sgn})_{\sf hS}[2]$. This explains the extra shift $[2]$ in the definition of $\widehat{\mathfrak{g}}$.  
\end{rmk}

On the algebraic side, there is also a DGLA associated with the $A_\infty$-category $\CC$. Following the notations in~\cite{CT}, we have the following chain complexes
\begin{align*}
L &:= \big( C_\bullet^{red}(\CC)[d], b \big)\\
L_+ &:= \big( C_\bullet^{red}(\CC)[d][[u]], b+uB\big)\\
L^{\sf Tate} &:= \big( C_\bullet^{red}(\CC)[d]((u)), b+uB\big)\\
L_- &:= L^{\sf Tate}/u\cdot L_+ = \big( C_\bullet^{red}(\CC)[d][u^{-1}], b+uB \big)
\end{align*}
With these notations, define a DGLA 
\begin{equation}~\label{eq:hhat}
 \widehat{\mathfrak{h}}_\CC := \bigoplus_{k\geq 1,l\geq 0} {\sf Hom}^c \big( \Sym^k (L_+[1]), \Sym^l (L_-)\big)[2] [[\hbar,\lambda]],
 \end{equation}
where ${\sf Hom}^c$ stands for $u$-adic continuous $\mathbb{K}$-linear maps.

The construction of $\widehat{\mathfrak{h}}_\CC$ is completely parallel to that of the DGLA $ \widehat{\mathfrak{g}}$. For example, the Mukai graph $M$ under the TCFT action $\rho^\CC$ yields the chain-level Mukai pairing $\langle-,-\rangle_{\sf Muk}:= \rho^\CC(M) : L \otimes L \ra \mathbb{K}$, which we use to define a map 
\[\iota: {\sf Hom}^c \big( \Sym^k (L_+[1]), \Sym^l (L_-)\big) \ra {\sf Hom}^c \big( \Sym^{k+1} (L_+[1]), \Sym^{l-1} (L_-)\big).\] 
Similarly, we use $\langle B-,-\rangle_{\sf Muk}= \rho^\CC(\mathbb{M}) : L \otimes L \ra \mathbb{K}$ to define the twisted self-sewing map 
\[\Delta: {\sf Hom}^c \big( \Sym^k (L_+[1]), \Sym^l (L_-)\big) \ra {\sf Hom}^c \big( \Sym^{k} (L_+[1]), \Sym^{l-2} (L_-)\big).\]
As a conclusion, we obtain a DGLA structure on $ \widehat{\mathfrak{h}}_\CC$ of the form:

\begin{itemize}
\item Its differential is of the form $\eth + \hbar \Delta+ \iota$, where $\eth$ is the commutator with $b+uB$ and $\Delta$ is the algebraic analogue of the twisted self-sewing operator.
\item Its Lie bracket is also of the form
$\{-,-\}_\hbar := \sum_{r\geq 1} \frac{1}{r!} \cdot \{-,-\}_r \hbar^{r-1}$ with $\{-,-\}_r$ the algebraic version of the $r$-th twisted sewing operator.
\end{itemize}

By construction, the action map $\rho^\CC$ induces a morphism of {\em $\Z/2\Z$-graded} DGLA's which we still denote by
\begin{equation}~\label{eq:rho-dgla}
 \rho^\CC: \widehat{\mathfrak{g}} \;\;\ra \;\; \widehat{\mathfrak{h}}_\CC .
 \end{equation}
The push-forward of the combinatorial string vertex $\widehat{\mathcal{V}}$ yields a Maurer-Cartan element of the form
\[\widehat{\beta}^\CC  := \rho^\CC_* \widehat{\mathcal{V}} = \sum_{g,k,l} \rho^\CC\big( \widehat{\mathcal{V}}_{g,k,l}\big) \hbar^g\lambda^{2g-2+k+l}.\]
Sometimes we also denote the components of $\widehat{\beta}^\CC$ by $\widehat{\beta}^\CC_{g,k,l}:=\rho^\CC\big( \widehat{\mathcal{V}}_{g,k,l}\big)$.

In~\cite{Cos2}, instead of working with $\widehat{\mathfrak{h}}_\CC$, Costello works with the symmetric algebra $\Sym(L_-) [[\hbar,\lambda]]$. The twisted sewing map $\Delta:\Sym^2 (L_-) \rightarrow \bbK$ defines a differential graded Batalin–Vilkovisky algebra structure on $\Sym (L_-) [[\hbar,\lambda]]$.  In particular, this yields a DGLA structure on $\mathfrak{h}_\CC:= \Sym( L_-)[[\hbar,\lambda]] [1]$ with its Lie bracket defined by $$\{x,y\} := (-1)^{|x|}\big( \Delta(x\cdot y)- \Delta x \cdot y - (-1)^{|x|} x\cdot \Delta y \big).$$ Inside this DGLA lies the scalars $\mathbb{K}[[\hbar,\lambda]][1]$ which is clearly central. Let us define the quotient DGLA by
\begin{align*}
\mathfrak{h}_\CC^+ & := \Sym( L_-) [[\hbar,\lambda]][1]/ \mathbb{K}[[\hbar,\lambda]][1] = \Sym^{\geq 1}( L_-) [[\hbar,\lambda]][1].
\end{align*}
By construction, there is a short exact sequence of DGLA's:
\[ 0 \to \mathbb{K}[[\hbar,\lambda]][1] \to \mathfrak{h}_\CC \to \mathfrak{h}_\CC^+ \to 0.\]
The following lemma explains the relationship between the two DGLA's $\mathfrak{h}_\CC^+$ and $\widehat{\mathfrak{h}}_\CC$. It is proved in the same way as in~\cite[Lemma 4.4]{CCT} where $\CC$ is only an $A_\infty$-algebra.

\begin{lem}\label{lem:iota}
Assume the $A_\infty$-category $\CC$ (which is already proper by our assumption) is smooth and satisfies the Hodge-to-de-Rham degeneration property. Then there is a quasi-isomorphism of DGLA's
\begin{equation}\label{eq:iota-bar}
\bar{\iota} : \mathfrak{h}_\CC^+\; \to \; \widehat{\mathfrak{h}}_\CC.\end{equation}
\end{lem}

Using the above lemma, we obtain a unique (up to gauge equivalence) Maurer-Cartan element $\beta^\CC\in \Sym(L_-)[[\hbar,\lambda]] [1]$ such that $\bar{\iota} \beta^\CC$ is gauge equivalent to $\widehat{\beta}^\CC\in \widehat{\mathfrak{h}}_\CC$.

The discussions in this subsection may be summarized into the following diagram of DGLA's:
\[ \begin{CD}
  @. @. \widehat{\mathfrak{g}} \\
 @.  @.    @VV \rho^\CC V \\
\mathfrak{h}_\CC @>>> \mathfrak{h}_\CC^+@>\bar{\iota}>> \widehat{\mathfrak{h}}_\CC
  \end{CD}\]
In~\cite{CT}, it was proved that the Maurer-Cartan element $\beta^\CC\in \mathfrak{h}_\CC^+$ can be lifted to a Maurer-Cartan element $\widetilde{\beta^\CC}\in \mathfrak{h}_\CC$. The set of liftings is a torsor over the additive group underlying the scalars $ \mathbb{K}[[\hbar,\lambda]]$. A preferred lifting may be chosen using the Dilaton equation. Nevertheless, the particular choice of lifting will not be important in this paper. 

For a differential graded Batalin–Vilkovisky algebra, it is well-known that the Maurer-Cartan equation satisfied by $\widetilde{\beta^\CC}$ is equivalent to the equation
\[ (b+uB+\hbar\Delta) \exp ( \frac{\widetilde{\beta^\CC}}{\hbar} ) =0,\]
where the left hand side is computed in the localized and completed symmetric algebra $\widehat{\Sym}_\hbar( L_-) [[\hbar,\lambda]]$. From this, we obtain a $(b+uB+\hbar \Delta)$-homology class
\[ [ \exp  ( \frac{\widetilde{\beta^\CC}}{\hbar} ) ] \in H_\bullet\big(\widehat{\Sym}_\hbar( L_-) [[\hbar,\lambda]]\big).\]
This was called the {{\em abstract  total ancestral potential}} of $\CC$. Note that the abstract  total  ancestral potential of $\CC$ is independent of the choice of string vertices by its homotopy uniqueness.

\subsection{Trivializations and CEI}

To obtain more familiar looking invariants, we need to ``trivialize" the BV operator $\Delta$. This shall involve a choice of splitting of the non-commutative Hodge filtration, whose definition is recalled in Definition~\ref{defi:splitting}. 

Let $s: H_\bullet(L) \ra H_\bullet(L_+)$ be a splitting of the Hodge filtration of $\CC$. We may consider $H_\bullet(L)$ with the trivial circle action, and denote by $H_\bullet(L)_-:= H_\bullet(L)[u^{-1}]$. As in the previous paragraph, we take $\widehat{\Sym}_\hbar H_\bullet(L)_-[[\hbar,\lambda]]$ to be the localized and completed symmetric algebra generated by $H_\bullet(L)_-$. Following the construction in~\cite[Section 11]{Cos2}, the splitting map $s$ induces an isomorphism 
\begin{equation}\label{diag:def}
\begin{CD}
\widehat{\Sym}_\hbar H_\bullet(L)_-[[\hbar,\lambda]] @>\Psi^s>> H_\bullet(\widehat{\Sym}_\hbar L_- [[\hbar,\lambda]]).
\end{CD}
\end{equation}

\begin{defn}~\label{def:main}
The {{\sl total  ancestral potential }} of the pair $(\CC,s)$ is defined by the pre-image of the abstract  total  ancestral potential $\exp ( \frac{\widetilde{\beta^\CC}}{\hbar} ) $, i.e. we set
\[ \mathcal{D}_{\CC,s}:=(\Psi^s)^{-1} [ \exp ( \frac{\widetilde{\beta^\CC}}{\hbar} )].\]
The $n$-point function of genus $g$ denoted by $F^{\CC,s}_{g,n}\in \Sym^n H_\bullet(L)_-$ is defined by the identity
\[ \sum_{g,n}F^{\CC,s}_{g,n}\cdot \hbar^g\lambda^{2g-2+n}:= \hbar\cdot \ln \mathcal{D}_{\CC,s}.\]
\end{defn}

From the definition above, it is not immediately clear how to compute these invariants. In~\cite{CT}, an explicit formula is obtained expressing $F^{\CC,s}_{g,n}$'s using $\beta^\CC_{g,n}$'s and the splitting data $s$. The formula is obtained by constructing a $L_\infty$ quasi-isomorphism (which depends on the splitting data $s$):
\begin{equation}~\label{eq:trivialization}
 K^s: \mathfrak{h}_\CC \;\to \;  \mathfrak{h}_\CC^{\sf triv} ,
 \end{equation}
where $ \mathfrak{h}_\CC^{\sf triv}$ has the same underlying graded vector space as $ \mathfrak{h}_\CC$ but is endowed with the differential $b+uB$ and zero Lie bracket.  Namely, the morphism $K^s$ kills the extra differential $\hbar\Delta$ and the Lie bracket $\{-,-\}$ of $ \mathfrak{h}_\CC$. The explicit formula of $F^{\CC,s}_{g,n}$ is then obtained by pushing forward the Maurer-Cartan element $\widetilde{\beta^\CC}$ via the morphism $K^s$, i.e. we have
\begin{equation}~\label{eq:explicit-formula} 
\sum_{g,n} F^{\CC,s}_{g,n} \hbar^g \lambda^{2g-2+n}= K^s_* \widetilde{\beta^\CC}.
\end{equation}
The right hand side can be expanded using the definition of $K^s$ as a stable graph sum. 

Since the Maurer-Cartan element $\widetilde{\beta^\CC}$ is also not explicitly constructed, in practice it is more desirable to work with $\widehat{\beta}^\CC$ directly. In~\cite{CT}, another explicit formula of CEI was obtained which expresses $\bar{\iota}(F_{g,n}^{\CC,s})$ using the $\widehat{\beta}^\CC_{g,k,l}$'s. This formula takes the following form
\begin{equation}\label{eq:cei}
\bar{\iota} (F_{g,n}^{\CC,s}) = \sum_{m\geq 1}  \sum_{\GG\in \Gamma((g,1,n-1))_m} (-1)^{m-1} \frac{\wt(\GG)}{\Aut(\GG)} \prod_{v} {\sf Cont} (v)  \prod_{e} {\sf Cont} (e)  \prod_{l} {\sf Cont} (l).
\end{equation}
The second summation in the above equation is over isomorphism classes of the so-called {\em partially directed graphs}~\cite[Definition 8.2]{CT} of genus $g$, with $1$ input, $n-1$ outputs, and $m$ vertices. Roughly speaking, such a graph is a graph $G$ with a spanning tree $T\subset G$ in it. Furthermore, some of the edges in $G$ (hence in $T$) are directed. These data should also satisfy some axioms such as each vertex has at least one incoming half-edge, and no directed loops are allowed in $G$. We refer to {\em loc. cit.} for details. The fraction $\frac{\wt(\GG)}{\Aut(\GG)}$ is a rational weight associated to such a graph. 

Finally, the contribution $ \prod_{v} {\sf Cont} (v)  \prod_{e} {\sf Cont} (e)  \prod_{l} {\sf Cont} (l)$ is given by the composition along a given partially directed graph $\GG$ where vertices, edges and legs are explicitly given by the following assignments.
\begin{itemize}
\item At a vertex $v$, we assign the multi-linear maps given by the image of combinatorial string vertices $\widehat{\beta}^\CC_{g,k,l}$'s. 
\item Let $S: L^\CC \ra L^\CC_+$ be a chain-level lift of the splitting map $s: H_*(L^\CC) \ra H_*(L^\CC_+)$ of the form $S=\id+ O(u)$. Extending it $u$-linearly to obtain a map still denoted by $S: L^\CC_+ \ra L^\CC_+$. At an incoming leg of $\GG$, we put this map.
\item At an outgoing leg, we put the inverse operator of $S$ (which exists since $S=\id+ O(u)$). Denoted this inverse map by $R=S^{-1}$.
\item For a directed edge in the spanning tree $T$ of $\GG$, assign an homotopy operator $F: L^\CC_- \ra L^\CC_+[1] $, defined by
\[ F(\beta) := - u^{-1} \cdot S \tau_{\geq 1} R (\beta).\]
where $\tau_{\geq 1}: L^\CC((u)) \ra u\cdot L^\CC_+$ is the projection onto strictly positive powers of $u$. 
\item While at other directed edges (that are not in the spanning tree $T$) we put the operator $\Theta: L^\CC_- \ra L^\CC_+[1]$ given by 
\[ \Theta( \alpha_0 + \alpha_1 u^{-1} +\cdots ) = B\alpha_0.\]
\item At an un-directed edge that is not in the spanning tree $T$, we put the homotopy operator $H^{\sym}: L^\CC_- \otimes L^\CC_- \ra \mathbb{K}$ given by the symmetrization of the operator $H: L^\CC_- \otimes L^\CC_- \ra \mathbb{K}$ defined by
\[ H(\alpha,\beta) := - \langle S\tau_{\geq 1} R (\alpha), \beta \rangle_{\sf res}.\]
\item For an un-directed edge in the spanning tree $T$, we assign an operator
$\delta: L^\CC_- \otimes L^\CC_- \ra \mathbb{K}$ which satisfies the following equation:
\[ [b+uB, \delta] = H - H^\sym,\]
i.e. it bounds the failure of $H$ being a symmetric operator. 
\end{itemize}
Composing along $\GG$ using the above operators yields the desired contribution in Equation~\eqref{eq:cei}.

\section{Morita invariance of CEI}\label{sec:morita}

In this section, we establish the Morita invariance of CEI. In the first subsection, we show that they are invariant under inclusions of full subcategories with isomorphic Hochschild homology. The following four subsections prove invariance under cyclic isomorphisms, and the last subsection wraps up the proof for general Morita equivalences.

\subsection{Invariance for subcategories}

We shall continue to use notation from the previous section.

\begin{prop}\label{lem:inclusion}
Let $\CC$ be a minimal, unital, cyclic \Ai-category. Let $I:\cA \hookrightarrow \CC$ be the inclusion of a full $A_\infty$-subcategory. Assume $I$ induces an isomorphism $I_*: HH_\bullet(\cA) \ra HH_\bullet(\CC)$. Let $s$ be a splitting of the Hodge filtration of $\cA$. Then, for any $(g,n)$ such that $2g-2+n>0$,  we have
\[ I_* F_{g,n}^{\cA, s} = F_{g,n}^{\CC, I_* s I_*^{-1}},\]
where $I_* s I_*^{-1}$ is the splitting of $\CC$ induced by $s$ via $I_*$.
\end{prop}

\begin{proof}
$I$ is a naive $A_\infty$\ functor (no higher order terms), therefore $I_*: L_-^\cA \to  L_-^\CC$ induces a map of DGLA's $\overline{I}:\mathfrak{h}_\cA=\Sym (L_-^\cA) [[\hbar,\lambda]][1]\to \Sym (L_-^\CC) [[\hbar,\lambda]][1]=\mathfrak{h}_\CC$.
Let us then consider the following diagram of DGLA homomorphisms:
\[\begin{tikzcd}
\mathfrak{h}_\CC  \arrow[r,"\overline{\iota}^\CC"] & \widehat{\mathfrak{h}_\CC} \\
\mathfrak{h}_\cA \arrow[u,"\overline{I}"] \arrow[r,"\overline{\iota}^\cA"] & \widehat{\mathfrak{h}_\cA}. 
\end{tikzcd}\]
By construction, the two horizontal maps are quasi-isomorphisms, and by our assumption the vertical map $\overline{I}$ is also a quasi-isomorphism. However, since the DGLA's $\widehat{\mathfrak{h}_\CC}$ and $\widehat{\mathfrak{h}_\cA}$ have inputs and outputs, $I_*$ does not induce a map between them. To remedy this, we define a sub-DGLA $\widehat{\mathfrak{h}}_{\cA,\CC}\subset \widehat{\mathfrak{h}}_\CC$ which consists of elements $\Phi$ satisfying the support condition that if all inputs of $\Phi$ are in $L_+^\cA$, then the outputs are in $L_-^\cA$. It is straightforward to verify that indeed $\widehat{\mathfrak{h}}_{\cA,\CC}$ is a sub-DGLA. Denote by $i: \widehat{\mathfrak{h}}_{\cA,\CC} \rightarrow \widehat{\mathfrak{h}}_{\CC}$ the natural inclusion map. The support condition also allows us to define a restriction map denoted by $r: \widehat{\mathfrak{h}}_{\cA,\CC} \rightarrow \widehat{\mathfrak{h}}_{\cA}$. 

Next, we argue that both $i$ and $r$ are quasi-isomorphisms. Indeed, we first run the spectral sequence associated to the filtration by powers of $\hbar$. On the associated graded pieces the differential is given by $\eth +\iota$. Then we run a second spectral sequence associated to the filtration by tensor powers of the outputs. The map on the corresponding associated graded pieces are given by
\begin{align*}
    i &: \oplus_k{\sf Hom}^{c,supp}\big( \Sym^k(L_+^\CC[1]),\Sym^l (L_-^\CC)\big) \to \oplus_k{\sf Hom}^{c}\big( \Sym^k(L_+^\CC[1]),\Sym^l( L_-^\CC)\big),\\
    r &: \oplus_k{\sf Hom}^{c,supp}\big( \Sym^k(L_+^\CC[1]),\Sym^l (L_-^\CC)\big) \to \oplus_k{\sf Hom}^{c}\big( \Sym^k(L_+^\cA[1]),\Sym^l( L_-^\cA)\big)
\end{align*}
for each $l\geq 0$. Both sides are endowed with the differential $\eth$ (commutator with $b+uB$). To see $i$ is a quasi-isomorphism, we observe that it is an injective map and its cokernel is given by $\oplus_k{\sf Hom}^{c}\big( \Sym^k(L_+^\cA[1]),\Sym^l (L_-^\CC/L_-^\cA)\big)$ which is acyclic since, by our assumption that $I_*$ is an isomorphism on Hochschild homology, $L_-^\CC/L_-^\cA$ is acyclic. Similarly, in the case of $r$, the map is surjective with kernel given by $\oplus_k{\sf Hom}^{c}\big( \Sym^k(L_+^\CC/L_+^\cA[1]),\Sym^l( L_-^\CC)\big)$ which is again acyclic by the same reason.

Let us consider the following diagram of DGLA homomorphisms:
\[\begin{tikzcd}
\mathfrak{h}_\CC  \arrow[r,"\overline{\iota}^\CC"] & \widehat{\mathfrak{h}_\CC} & \\
&\widehat{\mathfrak{h}}_{\cA,\CC}\arrow[d,"r"] \arrow[u,"i"]& \widehat{\mathfrak{g}} \arrow[ld,"\rho^{\cA}"above] \arrow[lu,"\rho^{\CC}" above] \arrow[l,"\rho^{\CC}" above] \\
\mathfrak{h}_\cA \arrow[uu,"\overline{I}"] \arrow[r,"\overline{\iota}^\cA"] & \widehat{\mathfrak{h}}_\cA  & 
\end{tikzcd}\]
It easily follows from the definition of $\rho$ that both triangles commute. We now claim that $\overline{I}_* \beta^\cA$ is gauge equivalent to $\beta^\CC$. Indeed, we have
\begin{align*}
& \;\;\; \overline{I}_* \beta^\cA \cong \beta^\CC \\
\Leftrightarrow \;\;\;& \overline{\iota}^\CC_* \overline{I}_* \beta^\cA \cong \overline{\iota}^\CC_*\beta^\CC \\
\Leftrightarrow \;\;\;& \overline{\iota}^\CC_* \overline{I}_* \beta^\cA \cong \widehat{\beta^\CC}\\
\Leftrightarrow \;\;\;& \overline{\iota}^\CC_* \overline{I}_* \beta^\cA \cong \rho^\CC_*\widehat{\mathcal{V}}
\end{align*}
Now observe that both sides satisfy the support condition, hence they in fact are elements in the sub-DGLA $\widehat{\mathfrak{h}}_{\cA,\CC}$. Thus we may continue as
\begin{align*}
& \;\;\;\overline{\iota}^\CC_* \overline{I}_* \beta^\cA \cong \rho^C_*\widehat{\mathcal{V}}\\
\Leftrightarrow \;\;\;& r_*\overline{\iota}^\CC_* \overline{I}_* \beta^\cA \cong r_*\rho^\CC_*\widehat{\mathcal{V}}\\
\Leftrightarrow \;\;\;& r_*\overline{\iota}^\CC_* \overline{I}_* \beta^\cA \cong \rho^\cA_*\widehat{\mathcal{V}}
\end{align*}
Finally, we notice that $r_*\overline{\iota}^\CC_* \overline{I}_* \beta^\cA= \overline{\iota}^\cA_* \beta^\cA$. This implies the last identity  is equivalent to 
\[\overline{\iota}^\cA_* \beta^\cA \cong \widehat{\beta^\cA},\]
which holds by definition. This proves our claim that $\overline{I}_* \beta^\cA \cong \beta^\CC$.

Since $\cA\subset \CC$ is a subcategory, we have a commutative diagram of the trivialization map:
\[\begin{tikzcd}
\mathfrak{h}_\cA  \arrow[d,"\overline{I}"]  \arrow[r,"K^s"] & \mathfrak{h}^{\sf triv}_\cA \arrow[d,"\overline{I}"]  \\
\mathfrak{h}_\CC\arrow[r,"K^{I_* s I_*^{-1}}"] & \mathfrak{h}^{\sf triv}_\CC
\end{tikzcd}\]
The Lemma then follows from Equation~\eqref{eq:explicit-formula}:
\[I_*\sum_{g,n} F^{\cA,s}_{g,n} \hbar^g \lambda^{2g-2+n} =  \overline{I}_* K^s_* \beta^\cA = K^{ I_* s I_*^{-1}}_* \overline{I}_* \beta^\cA= K^{ I_* s I_*^{-1}}_* \beta^\CC= \sum_{g,n} F^{\CC,I_* s I_*^{-1}}_{g,n} \hbar^g \lambda^{2g-2+n}\]
\end{proof}

\subsection{Cyclic Pseudo-isotopies}

Let $\CC$ be an $A_\infty$-category over a field $\mathbb{K}$ of characteristic zero. Let $\CC'$ be an $A_\infty$-category which has the same underlying objects and ${\sf Hom}$-spaces as $\CC$, but is endowed with $A_\infty$-maps $\m'_k\;(k\geq 1)$. Without loss of generality, we assume that both $\CC$ and $\CC'$ are minimal, i.e., $\m_1=\m'_1=0$. We recall the notion of cyclic pseudo-isotopy, introduced by Fukaya~\cite{Fuk}.

\begin{defn}
A minimal pseudo-isotopy between $\CC$ and $\CC'$ consists of families of cochains (as in Equation (\ref{eq:param})),
\begin{align*}
\m(t)=\prod_{k\geq 2}\m_k(t)\in C^\bullet(\CC)\{t\}, \ 
\fl(t)=\prod_{k\geq 2}\fl_k(t) \in  C^\bullet(\CC)\{t\},
\end{align*}
of degree $0$ and $1$ respectively. The maps $\m(t)$  satisfy the $A_\infty$ equations $\m(t)\bullet\m(t)=0$, while the $\fl(t)$  satisfy the identity $\frac{d}{dt} \m(t) = [\m(t),\fl(t)]$, or explicitly the equation
\begin{align*}
\frac{d}{dt} \m_k(t)& (x_1,\ldots,x_k)  = \sum_{r, i+j=k+1} \m_i(t)\big( x_1,\ldots,x_r,\fl_j(t)(x_{r+1},\ldots,x_{j+r}),\ldots,x_k\big) \\
& - \sum_{r, i+j=k+1} (-1)^{|x_1|'+\cdots+|x_r|'} \fl_i(t)\big( x_1,\ldots,x_r,\m_j(t)(x_{r+1},\ldots,x_{j+r}),\ldots,x_k\big)
\end{align*}
Finally, we require the boundary conditions  $\m_k(0)=\m_k$, while $\m_k(1)=\m'_k$.
\end{defn}

This definition can be rephrased as follows. Let $\Omega^\bullet:= \mathbb{K}[t,dt]$ be the space of polynomial differential forms on $\mathbb{A}^1$. Denote by $\CC\otimes \Omega^\bullet$ the (pre-)category with the same objects as $\CC$ and the hom spaces $\hom_{\CC\otimes \Omega^\bullet}(X, Y):=\hom_\CC(X, Y)\otimes \Omega^\bullet $. Given ${\bf{x_i}}:=x_i(t)+y_i(t)dt \in \hom_{\CC\otimes \Omega^\bullet}(X_i, X_{i+1})$ We define the operations:
\begin{align*}
  \overline{\m}_1({\bf{x_1}})&:=(-1)^{|x_1|'}\frac{dx_1(t)}{dt} dt,\\
  \overline{\m}_k({\bf{x_1}}, \ldots, \bf{x_k} )&:= \m_k(t)(x_1(t),\ldots, x_k(t))+(-1)^{\sum_i |x_i|'} \fl_k(t)(x_1(t),\ldots, x_k(t))dt +\\
  &  \ + \sum_i (-1)^{\sum_{a>i} |x_a|'}\m_k(t)(x_1(t),\ldots,y_i(t),\ldots x_k(t)) dt, \ k\geq 2.
\end{align*}
It is elementary to see that being a pseudo-isotopy is equivalent to these operations defining an \Ai\ structure on  $\CC\otimes \Omega^\bullet$ with (naive) evaluation maps at $t=0$ and $t=1$ to $\CC$ and $\CC'$ respectively.

Moreover, assume that both $A_\infty$-categories $\CC$ and $\CC'$ are endowed with the same cyclic structure $\langle-,-\rangle$.
\begin{itemize}
\item A minimal pseudo-isotopy $\m(t) + \fl(t)dt$ between $\CC$ and $\CC'$ is called {\em cyclic} if the structure maps $\m_k(t)$'s and $\fl_k(t)$'s are all cyclic with respect to $\langle-,-\rangle$, meaning they satisfy Equation (\ref{eq:cyclic}).
\item It is called unital if $\m_k(t)$'s form a strictly unital $A_\infty$ structure, and the $\fl_k(t)$ are reduced cochains.
\end{itemize}

\begin{lem}\label{lem:isotopy_vectorfield}
	Let $\big(\CC,\langle-,-\rangle,\{\m_k\}_{k=2}^\infty\big)$ and $\big(\CC',\langle-,-\rangle,\{\m'_k\}_{k=2}^\infty\big)$ be two minimal $A_\infty$-categories with the same underlying set of objects and ${\sf Hom}$-spaces. Let $F: \CC \to \CC'$ be an \Ai-functor with $F_1=\id$. Then there is a degree one Hochschild cochain $Z \in CC^1(\CC)$ of order two, such that its flow, that is the family of \Ai-pre-functors determined by Lemma \ref{lem:flow}, satisfies $F^1=F$. 
	
	If $F$ is unital, then $Z$ is reduced and therefore the $F^t$ are unital. If $\CC$ and $\CC'$ are cyclic and $F$ is a cyclic functor, then $Z$ and $F^t$ are cyclic.  
\end{lem}
\begin{proof}
Let $\widehat{F}$ be the coalgebra homomorphism extension of $F$ as before, and define $G:=\widehat{F}-\id$. Since $F_1=\id$, the map $G$ decreases the length of elements in $B\CC$. Therefore the following sum is convergent: 
\[\widehat{Z}:=\log(\widehat{F})= \log(\id +G) := \sum_{n=1}^\infty (-1)^{n+1} \frac{G^{\circ n}}{n}.\]
Since $\widehat{F}$ is a coalgebra homomorphism, $G$ satisfies the identity
\[(G\otimes G + G\otimes \id + \id\otimes G)\circ \Delta = \Delta\circ G,\]
where $\Delta$ is the coproduct in $B\CC$. We then compute
\begin{align}\label{eq:coder}
    \Delta \circ \widehat{Z} & = \sum_{n=1}^\infty \frac{(-1)^{n+1}}{n}  (G\otimes G + G\otimes \id + \id\otimes G)^n \circ \Delta\nonumber\\
    &= \sum_{\substack{a,b\geq0 \\a+b>0}}\sum_{a,b\leq n\leq a+b} \frac{(-1)^{n+1}}{n}\binom{n}{n-a,n-b,a+b-n} G^{a}\otimes G^{b} \circ \Delta.
\end{align}
  From the Taylor expansion of the standard identity $\log(1+x+y+xy)=\log(1+x)+\log(1+y)$, we see that the coefficients in (\ref{eq:coder}) vanish, unless $a=0$ or $b=0$. Therefore we obtain
\[\Delta \circ \widehat{Z} = (\widehat{Z}\otimes \id+\id\otimes \widehat{Z})\circ \Delta,\]
 which means $\widehat{Z}$ is a coderivation. From the definition it is clear it satisfies the support condition in Remark \ref{rmk:codev}, hence it is determined by a degree one Hochschild cochain $Z$ of order $2$ (since $Z_0=Z_1=0$).

Since $Z$ is independent of $t$, the differential equation in Lemma \ref{lem:flow} can be easily solved:
\[\widehat{F^t}= \exp(t \widehat{Z})=\id + \sum_{n=1}^\infty \frac{t^n \widehat{Z}^{\circ n}}{n!} .\]
Then, as for the standard exponential, we have $\widehat{F^1}=\exp(\log(\widehat{F}))=\widehat{F}$. 

When $F$ is unital, it is clear from the formula above that $Z$ is reduced. 

Assume for now that $Z$ is cyclic. This is equivalent to $\mathcal{L}_Z\rho=S(\iota_Z \rho)=0$, where $\rho$ is the cyclic structure. Then using Lemmas \ref{lem:lie_pull} and \ref{lem:contraction} we calculate
\[\frac{d}{dt}(F^t)^* \rho= -(F^t)^*(\mathcal{L}_Z \rho)=-(F^t)^*(S(\iota_Z \rho))=0.\]
Hence $(F^t)^*\rho=(F^0)^*\rho=\rho$, and $F^t$ is cyclic.

Finally, we prove that if $F$ is cyclic then $Z$ is cyclic. First we observe that Lemma \ref{lem:lieder} implies that the space of cyclic Hochschild cochains is closed under the Gerstenhaber bracket. We will show that $Z_k$ is cyclic  by induction on $k$. From the definition we have $Z_2=F_2$. In length two, Equation (\ref{eq:F_cyclic}) for cyclicity of $F$ is equivalent to Equation (\ref{eq:cyclic}). Hence $Z_2$ is cyclic. Now define the pre-functor $\widehat{F^{(2)}}:= \exp(\widehat{-Z_2})\circ \widehat{F}$ and the coderivation $\widehat{Z^{(2)}}:=\log(\widehat{F^{(2)}})$. By construction, $F^{(2)}_1=\id$ and $F^{(2)}_2=0$.

We then inductively define a sequence of cyclic pre-functors $\widehat{F^{(2)}}, \ldots, \widehat{F^{(k)}}$, by setting $\widehat{F^{(k)}}:= \exp(\widehat{-Z^{(k-1)}_k})\circ \widehat{F^{(k-1)}}$ and $\widehat{Z^{(k)}}:=\log(\widehat{F^{(k)}})$. An easy computation shows that $F^{(k)}_l=0$ for $l=2,\ldots,k$, which then implies $Z^{(k)}_l=0$ for $l=0,\ldots,k$ and $Z^{(k)}_{k+1}=F^{(k)}_{k+1}$. Again, Equation (\ref{eq:F_cyclic}) for the cyclicity of $F^{(k)}_{k+1}$ is equivalent to Equation (\ref{eq:cyclic}). Hence $Z^{(k)}_{k+1}$ is a cyclic cochain and therefore $\exp(\widehat{-Z^{(k)}_{k+1}})$ is a cyclic pre-functor.

Next we apply the Baker--Campbell--Hausdorff formula (see \cite{DSV} for example) to $\widehat{Z^{(k)}}=\log \left( \exp(\widehat{-Z^{(k-1)}_k})\circ \cdots \circ \exp(\widehat{-Z_2})\circ \exp(\widehat{Z})\right) $ and obtain:
\begin{equation}
    \widehat{Z^{(k)}}= \widehat{Z} -\widehat{Z_2}-\widehat{Z^{(2)}_3}\ldots-\widehat{Z^{(k-1)}_k}+ BCH(\widehat{Z},-\widehat{Z_2}, -\widehat{Z^{(2)}_3}, \ldots -\widehat{Z^{(k-1)}_k}),
\end{equation}
where $BCH(--)$ is a linear combination of iterated brackets of the inputs (the precise formula is not important). Looking at inputs of length $k+1$, and projecting to length one, the above equation reads:
\begin{equation}
    Z^{(k)}_{k+1}=Z_{k+1}+[Z_2,\ldots, Z_k, Z^{(2)}_3, \ldots Z^{(k-1)}_k],
\end{equation}
where $[Z_2,\ldots, Z_k, Z^{(2)}_3, \ldots Z^{(k-1)}_k]$ is some linear combination of iterated brackets of those cochains, determined by the Baker--Campbell--Hausdorff formula. By induction hypothesis, all the cochains $Z_2,\ldots, Z_k, Z^{(2)}_3, \ldots Z^{(k-1)}_k$ are cyclic. Since the bracket preserves cyclicity and we have already shown that $Z^{(k)}_{k+1}$ is cyclic we conclude that $Z_{k+1}$ is cyclic, which completes the induction step.
\end{proof}

\begin{rmk}
    There is a more direct proof of the statement $F$ cyclic implies $Z$ cyclic. From the definition $\widehat{Z}=\log(\widehat{F})$, one can write down an explicit expression for $Z$ as a sum of ``trees with levels". A combinatorial argument then reduces the claim that $Z$ is cyclic to the same combinatorial identity used to simplify Equation (\ref{eq:coder}) above. We thank Ilia Zharkov for explaining us how to do this.
\end{rmk}

\begin{prop}~\label{prop:isotopy}
	Let $\big(\CC,\langle-,-\rangle,\{\m_k\}_{k=2}^\infty\big)$ and $\big(\CC',\langle-,-\rangle,\{\m'_k\}_{k=2}^\infty\big)$ be two minimal cyclic $A_\infty$-categories with the same underlying set of objects, $\hom$-spaces and cyclic structure. Then there exists an unital, cyclic $A_\infty$ functor
	\[ f=(f_1,f_2,\ldots): \CC \to \CC'\]
	with $f_1=\id$ if and only if there exists an unital, cyclic, minimal pseudo-isotopy $ \m(t)+\fl(t)dt $.
\end{prop}

\begin{proof}
	Assume that $f=(\id,f_2,f_3,\ldots)$ is such a cyclic isomorphism. Since $f_1=\id$, $f$ has a strict inverse $f^{-1}$, which is again unital, cyclic and $f^{-1}_1=\id$. By the previous lemma there exists a cyclic, reduced $\xi$, whose flow $F^t$ satisfies $f^{-1}= F^1$. 
	Pulling-back the \Ai-structure $\m$ via $ F^t$ yields a family of cochains $\m(t):= (F^t)^* \m= F^{-t}\circ\widehat{\m}\circ \widehat{F^t}$. As in the proof of Proposition \ref{prop:darboux}, $\m(t)$ form a family of \Ai-structures, such that
	\[\m(0)=\m \;\;\;\mbox{and}\;\;\; \m(1)=\m'.\] 
	The previous lemma implies that $F^t$ is cyclic, that is $(F^t)^*\rho=\rho$. Using Lemma \ref{lem:lie_pull} we compute
    \[\mathcal{L}_{\m(t)}\rho = \mathcal{L}_{(F^t)^*\m}\left((F^t)^*\rho\right) = (F^t)^*\mathcal{L}_\m \rho=0,\]
    which shows $\m(t)$ is cyclic. Next, we verify that $\m(t) + \xi dt$
	forms a cyclic pseudo-isotopy. Indeed, we have
	\begin{align*}
		 \frac{d}{dt}  \m(t)= & \frac{d}{dt} \big( F^{-t} \circ \widehat{\m} \circ \widehat{F^t}\big)\\
		= & -\xi \circ \widehat{F^{-t}} \circ \widehat{\m} \circ \widehat{F^t} + F^{-t} \circ \widehat{\m} \circ \widehat{\xi}\circ \widehat{F^t} \\
		= & -\xi \circ \widehat{F^{-t}} \circ \widehat{\m} \circ \widehat{F^t} + F^{-t} \circ \widehat{\m} \circ \widehat{F^t}\circ \widehat{\xi}\\
		= & \m(t)\circ \widehat{\xi} - \xi\circ\widehat{\m(t)} =[\m(t), \xi].
	\end{align*}
	In the second equality we used the differential equation defining $F^t$ and on the third we used that $\widehat{\xi}$ and $\widehat{F^t}$ commute, as can be seen from the explicit description of $F^t$ in the previous lemma.
	
	The other direction of the statement was proved by Fukaya \cite{Fuk}.
\end{proof}

\subsection{Homotopic TCFT's}

Let $ \m(t) + \fl(t) dt$ be a minimal, unital and cyclic pseudo-isotopy on $\CC$. Recall that this is equivalent to an \Ai\   structure in $\CC\otimes \Omega^\bullet$. 

Our next goal is to extend the construction of TCFT's in Theorem~\ref{thm:tcft} to pseudo-isotopic families. To do this, it is necessary to determine the signs when defining the map $\rho$. These are discussed in~\cite[Appendix]{WahWes}, and will be treated in more detail in~\cite{CalChe}. 

For exposition purposes, we describe the construction of the chain map $\rho^\CC$ in Theorem~\ref{thm:tcft} only in the case of Calabi-Yau dimension $d\equiv 0 \pmod{2}$ and with the number of outputs $l=0$. The discussion in the general case is analogous. We refer the interested reader to~\cite{CalChe} for full details.

Keeping the above assumptions in mind, let us consider the running example of the ribbon graph $G\in C_1^{\sf comb}(M^{\sf fr}_{1,1,0})$ depicted as
\[\includegraphics[scale=.6]{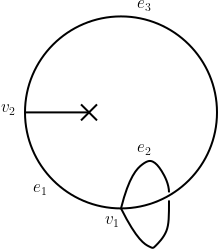}\]
In the first step we redraw the graph $G$ in the manner as illustrated in the black part of following picture.
\[\includegraphics[scale=.7]{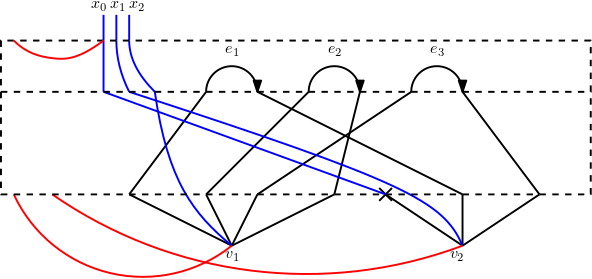}\]
In doing so, we need to choose 
\begin{itemize}
\item[(1)] an ordering of vertices, 
\item[(2)] an ordering of half-edges of $G$ (including leaves),
\item[(3)] an ordering of edges.
\end{itemize}
Also by definition an orientation of a ribbon graph is given by a choice of $(1)$ and $(2)$. We choose $(1)$ and $(2)$ so that it is compatible with the orientation of $G$. Since the Calabi-Yau pairing is of parity $d$ which we assumed to be even, the choice of $(3)$ is not relevant. In the blue part of the picture, there is also a Hochschild chain $x_0|x_1|x_2$ which we insert into the unique input cycle of the ribbon graph $G$. The red lines are attached to odd maps: the top red line is due to the shift map $$s: x_0|x_1|x_2 \mapsto sx_0|x_1|x_2;$$ while the bottom red lines are due to the odd map $$\widetilde{\m}_v:=\langle \m_{|v|-1}(-,\ldots,-),-\rangle$$  associated with black vertices of $G$.

In the evaluation $\rho^\CC(G)$, we simply put the Koszul sign associated with the permutation $\sigma_G$ as shown in the dashed box. Observe that this is indeed well-defined, i.e. we need to check that $\rho^\CC(G)$ flips its sign under $(1)$ and $(2)$, and is unchanged under $(3)$. At a black vertex $v$, its associated map is odd, which shows that switching $v_1$ with $v_2$ indeed produces a sign $-1$.  At an edge $e$, we put the inverse tensor of the pairing $\langle-,-\rangle^{-1}$ which is even, and hence the ordering of edges doesn't matter in this case. However, the orientation of each edge does flip the sign since the pairing $\langle-,-\rangle$ is anti-symmetric in the shifted degree. 

It is useful to write down the evaluation $\rho^\CC(G)$ (as explained above graphically) by the following equation:
\begin{equation}~\label{eq:action-map}
\rho^\CC(G):= (\prod_{v\in V_G^{\sf black}} \widetilde{\m}_v) \circ \sigma_G \circ  \big(\underbrace{s\otimes \cdots \otimes s}_{k \mbox{\; copies\;}} \otimes \underbrace{\langle-,-\rangle^{-1}\otimes \cdots \otimes \langle-,-\rangle^{-1}}_{|E_G| \mbox{\; copies\;}}\big)
\end{equation}
We choose the orderings involved in the above formula so that
\[ v_1\wedge\cdots\wedge v_n \wedge l_1 \wedge\cdots \wedge l_k  \wedge e_1^+\wedge e_1^- \wedge \cdots  \wedge e_m^+ \wedge e_m^- \]
agrees with the orientation of $G$, with $n=|V_G^{\sf black}|$ and $m=|E_G|$.

With the above construction of $\rho^\CC$, one can prove that it is a chain map, i.e.
\begin{equation}~\label{eq:chain-map}
 \rho^\CC(\partial G) = [\mathcal{L}_\m, \rho(G)].
 \end{equation}
See the Appendix~\ref{sec:sign} for a proof of this equation.


\begin{thm}
There exist chain maps
\[ \rho^{\CC\otimes \Omega^\bullet}_{g,k,l}: C^{\sf comb}_\bullet(M_{g,k,l}^{\sf fr}) \ra {\sf Hom} ( C_\bullet(\CC)[d]^{\otimes k}, C_\bullet(\CC)[d]^{\otimes l})\otimes_\mathbb{K} \Omega^\bullet\]
	with $g\geq 0$, $k\geq 1$, $l\geq 0$, and $2-2g-k-l<0$. The differential on the right hand side is $[\mathcal{L}_{\m(t)}+\mathcal{L}_{\fl(t)}dt, -]+d_{DR}$. Furthermore, these chain maps satisfy
	\begin{itemize}
	\item[(1.)] The restriction $\rho^{\CC\otimes \Omega^\bullet}_{g,k,l} \mid_{t=t_0} = \rho^{(\CC,\m(t_0))}_{g,k,l}$ for any fixed $t_0$.
	\item[(2.)] They are compatible with the composition maps on both sides.
	\end{itemize}
\end{thm}

\begin{proof}

The action map $\rho^{\CC\otimes \Omega^\bullet}_{g,k,l}$  is defined in the same way as in Theorem~\ref{thm:tcft}, except at one of the black vertices we put the operator $\fl dt$. Again, we shall only describe the construction in the case $l=0$ and $d\equiv 0\pmod{2}$. In view of Equation~\eqref{eq:action-map}, we set
\begin{equation}~\label{eq:action-map-tensor}
\rho^{\CC\otimes \Omega^\bullet}(G):= (\prod_{v\in V_G^{\sf black}} (\widetilde{\m}_v+\widetilde{\fl}_v dt)) \circ \sigma_G \circ  \big(\underbrace{s\otimes \cdots \otimes s}_{k \mbox{\; copies\;}} \otimes \underbrace{\langle-,-\rangle^{-1}\otimes \cdots \otimes \langle-,-\rangle^{-1}}_{|E_G| \mbox{\; copies\;}}\big)
\end{equation}
Since $dt^2=0$, we can decompose
\begin{align*}
&\rho^{\CC\otimes \Omega^\bullet}(G) = \rho^{\CC\otimes \Omega^\bullet}_{[0]}(G)+\rho^{\CC\otimes \Omega^\bullet}_{[1]}(G)\\
&\rho^{\CC\otimes \Omega^\bullet}_{[0]}(G):=(\prod_{v\in V_G^{\sf black}} \widetilde{\m}_v) \circ \sigma_G \circ  \big(\underbrace{s\otimes \cdots \otimes s}_{k \mbox{\; copies\;}} \otimes \underbrace{\langle-,-\rangle^{-1}\otimes \cdots \otimes \langle-,-\rangle^{-1}}_{|E_G| \mbox{\; copies\;}}\big)\\
&\rho^{\CC\otimes \Omega^\bullet}_{[1]}(G):=\\
&\sum_{j=1}^{n}  (\widetilde{\m}_{v_1}\otimes \cdots \otimes\widetilde{\fl}_{v_j}dt \otimes \cdots \otimes \widetilde{\m}_{v_{n}}) \circ \sigma_G \circ  \big(\underbrace{s\otimes \cdots \otimes s}_{k \mbox{\; copies\;}} \otimes \underbrace{\langle-,-\rangle^{-1}\otimes \cdots \otimes \langle-,-\rangle^{-1}}_{|E_G| \mbox{\; copies\;}}\big)
\end{align*}
To verify that $\rho^{\CC\otimes \Omega^\bullet}$ defined as above is indeed a chain map, we need to prove
\[ \rho^{\CC\otimes \Omega^\bullet}(\partial G) = [ \mathcal{L}_\m+ \mathcal{L}_\fl dt + d_{DR}, \rho^{\CC\otimes \Omega^\bullet}(G)]\]
In the case $l=0$, writing the above equation in its components yields
\begin{align}~\label{eq:25}
\begin{split}
 \rho^{\CC\otimes \Omega^\bullet}_{[0]}(\partial G)&+(-1)^{|G|}\rho^{\CC\otimes \Omega^\bullet}_{[0]}(G)\mathcal{L}_\m = 0\\
\underbrace{\rho^{\CC\otimes \Omega^\bullet}_{[1]}(\partial G)}_{(i)} & +\underbrace{(-1)^{|G|}\rho^{\CC\otimes \Omega^\bullet}_{[0]}(G) \mathcal{L}_\fl dt}_{(ii)}  + \underbrace{(-1)^{|G|}\rho^{\CC\otimes \Omega^\bullet}_{[1]}(G) \mathcal{L}_\m}_{(iii)}=[d_{DR},\rho^{\CC\otimes \Omega^\bullet}_{[0]}(G)]
\end{split}
\end{align}
The top equation is the same as in Equation~\eqref{eq:chain-map}. For the second equation,  since $[d_{DR},\m]+[\m,\fl dt]=0$ we have
\begin{align*}
& [d_{DR},\rho^{\CC\otimes \Omega^\bullet}_{[0]}(G)]= \\
& \sum_{j=1}^{n}  (-1)^j (\widetilde{\m}_{v_1}\otimes \cdots \otimes[\widetilde{\m},\widetilde{\fl}dt] \otimes \cdots \otimes \widetilde{\m}_{v_{n}}) \circ \sigma_G \circ  \big(\underbrace{s\otimes \cdots \otimes s}_{k \mbox{\; copies\;}} \otimes \underbrace{\langle-,-\rangle^{-1}\otimes \cdots \otimes \langle-,-\rangle^{-1}}_{|E_G| \mbox{\; copies\;}}\big)
\end{align*}
In this equation, there are three types of configurations in the above composition, depicted in the following picture.
\[\includegraphics[scale=.8]{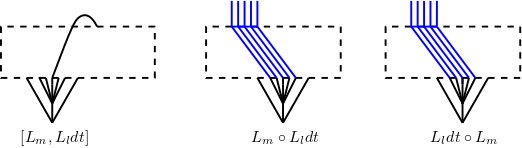}\]
The leftmost case is when the inputs of the upper vertex involves an edge in $G$, this term corresponds to $(i)$ in Equation $(23)$ above. The middle and rightmost cases correspond to $(ii)$ and $(iii)$ respectively. The signs involved here are explained in the Appendix~\ref{sec:sign}.
\end{proof}


\subsection{Homotopic DGLA's}

One of the main difficulties in proving that CEI are invariant under unital cyclic $A_\infty$-isomorphisms is that the construction of the DGLA $\widehat{\mathfrak{h}}_\CC$ in Equation~\eqref{eq:hhat} associated with $\CC$ is {\em not} functorial. To see this, it suffices to observe that the Mukai pairing is already not functorial on the chain-level. To avoid this issue, in this subsection we extend the construction of $\widehat{\mathfrak{h}}_\CC$ associated with an $A_\infty$ category $\CC$ to the family version. Again, we shall work with $ \m(t) + \fl(t) dt$ that is a minimal, cyclic and unital pseudo-isotopy on an $A_\infty$-category category $\CC$. 

The family version of the DGLA construction is defined as follows. As a graded vector space, it is given by 
\[ \widehat{\mathfrak{h}}_{\CC\otimes\Omega^\bullet}:= \bigoplus_{k\geq 1,l\geq 0} {\sf Hom}^c \Big( \Sym^k \big(L_+[1]\big), \Sym^l (L_-)\otimes_\mathbb{K} \Omega^\bullet\Big)[2] [[\hbar,\lambda]].\]
Its DGLA structure is analogously defined as in Subsection~\ref{subsec:tcft-dgla}:
\begin{itemize}
\item Its differential is of the form $\eth_\otimes+d_{DR}+\iota_\otimes+ \hbar \Delta_\otimes$, where
$\eth_\otimes$ is the commutator with $\mathcal{L}_{\m(t)}+\mathcal{L}_{\mathfrak{l}(t)}dt+uB$, $\Delta_\otimes$ is the twisted self-sewing operator, i.e. sewing with $\rho^{\CC\otimes\Omega^\bullet}(\mathbb{M})$, and the map $\iota_\otimes$ is sewing with $\rho^{\CC\otimes\Omega^\bullet}(M)$. 
\item Its Lie bracket is of the form
$\{-,-\}_\hbar := \sum_{r\geq 1} \frac{1}{r!} \cdot \{-,-\}_r \hbar^{r-1}$ with $\{-,-\}_r$ the $r$-th twisted sewing operator. This is the same Lie bracket (extended $\Omega^\bullet$ linearly) as in the DGLA $ \widehat{\mathfrak{h}}_{\CC}$. 
\end{itemize}

By definition, the isotopic family of TCFT's $\rho^{\CC\otimes\Omega^\bullet}$ then gives us a morphism of DGLA's:
\[ \rho^{\CC\otimes\Omega^\bullet}: \widehat{\mathfrak{g}} \; \ra \;  \widehat{\mathfrak{h}}_{\CC\otimes\Omega^\bullet},\]
which specializes to $\rho^{(\CC,\m(t_0))}$ in Equation~\eqref{eq:rho-dgla} for any fixed $t_0$. 
More precisely, for each $t_0$ there is a DGLA quasi-isomorphism ${\sf ev_{t_0}}: \widehat{\mathfrak{h}}_{\CC\otimes\Omega^\bullet} \to \widehat{\mathfrak{h}}_{(\CC, \m(t_0))}$, such that the following diagram commutes
\[\begin{tikzcd}
\widehat{\mathfrak{g}}  \arrow[r,"\rho^{\CC\otimes\Omega^\bullet}"] \arrow[dr,"\rho^{(\CC,\m(t_0))}"left] & \widehat{\mathfrak{h}}_{\CC\otimes\Omega^\bullet} \arrow[d, "{\sf ev_{t_0}}"] \\
& \widehat{\mathfrak{h}}_{(\CC, \m(t_0))}
\end{tikzcd}\]

Similarly, there is a family version of the DGLA $\mathfrak{h}$. We define
\[\mathfrak{h}_{\CC\otimes\Omega^\bullet}:=\big(\Sym(L_-)\otimes \Omega^\bullet\big)[[\hbar,\lambda]][1],\]
equipped with the differential $\mathcal{L}_{\m(t)}+\mathcal{L}_{\mathfrak{l}(t)}dt+uB+\hbar\Delta_\otimes$. The Lie bracket is determined by $\Delta_\otimes$ as in Subsection \ref{subsec:tcft-dgla}. Again, there are evaluation maps ${\sf ev_{t_0}}: \mathfrak{h}_{\CC\otimes\Omega^\bullet} \to \mathfrak{h}_{(\CC, \m(t_0))}$, for each $t_0\in \mathbb{K}$. These are maps of DGLAs by construction. Using our assumption that $\CC$ is minimal and arguing by the length filtration, one may deduce that ${\sf ev_{t_0}}$ are quasi-isomorphisms. Finally, as in Lemma \ref{lem:iota} there is a DGLA homomorphism
\[\bar{\iota}_\otimes: \mathfrak{h}_{\CC\otimes\Omega^\bullet}^+\to \widehat{\mathfrak{h}}_{\CC\otimes\Omega^\bullet}.\]
It follows from the construction, that $\bar{\iota}_\otimes$ specializes to $\bar{\iota}$ at $t=0$, that is ${\sf ev_{0}}\circ \bar{\iota}_\otimes = \bar{\iota}\circ {\sf ev_{0}}$. The evaluation maps and $\bar{\iota}$ are quasi-isomorphisms by Lemma \ref{lem:iota},  therefore $\bar{\iota}_\otimes$ is a quasi-isomorphism of DGLAs. 

\subsection{Cyclic invariance}
Let $\big(\CC,\langle-,-\rangle,\{\m_k\}\big)$ and $\big(\CC',\langle-,-\rangle,\{\m'_k\}\big)$ be two minimal, unital and cyclic $A_\infty$-categories with the same underlying set of objects, $\hom$-spaces, and the cyclic pairing. Let $f$ be an unital and cyclic $A_\infty$ functor $f=(f_1,f_2,\ldots): \CC \to \CC'$ with $f_1=\id$.
	
For a splitting map $s: H_*(L_{\CC})^{\sf Tate} \ra H_*(L_{\CC}^{\sf Tate})$ of $\CC$, there is a naturally associated splitting  of $\CC'$ defined by the composition
	\[ f_* s f_*^{-1} : H_*(L_{\CC'})^{\sf Tate} \ra H_*(L_{\CC})^{\sf Tate} \ra H_*(L_{\CC}^{\sf Tate})\ra H_*(L_{\CC'}^{\sf Tate}).\]
\begin{thm}~\label{thm:cyclic-inv}
Let $f$ be an unital and cyclic $A_\infty$-isomorphism
	\[ f=(f_1,f_2,\ldots): \big(\CC,\langle-,-\rangle,\{\m_k\}) \;\ra\; \big(\CC',\langle-,-\rangle,\{\m'_k\}\big)\]
	with $f_1=\id$. Then we have
	\[ f_* F_{g,n}^{\CC,s} = F_{g,n}^{\CC', f_* s f_*^{-1}}.\]
	\end{thm}
	
	\begin{proof}
	We first use Proposition~\ref{prop:isotopy} to construct a pseudo-isotopic family of $A_\infty$ categories given by
	\[ \m(t)+\fl(t) dt:= \exp (t \xi)^* \m + \xi dt \]
	with $f^{-1}= \exp (\xi)$. Using the associated isotopic family of TCFT's we obtain a morphism $\rho^{\CC\otimes\Omega^\bullet}: \widehat{\mathfrak{g}} \; \ra \;  \widehat{\mathfrak{h}}_{\CC\otimes\Omega^\bullet}$ of DGLA's. Use it to push-forward the string vertex $\widehat{\mathcal{V}}$ we obtain a Maurer-Cartan element $\widehat{\beta}^{\CC\otimes\Omega^\bullet} =  \rho^{\CC\otimes\Omega^\bullet}_* (\widehat{\mathcal{V}})$. Consider the following diagram
	\[ \begin{CD}
  @. @. \widehat{\mathfrak{g}} \\
 @.  @.    @VV \rho^{\CC\otimes \Omega^\bullet} V \\
\mathfrak{h}_{\CC\otimes\Omega^\bullet} @>>> \mathfrak{h}_{\CC\otimes\Omega^\bullet}^+@>\bar{\iota}>> \widehat{\mathfrak{h}}_{\CC\otimes\Omega^\bullet}
  \end{CD}\]
  
Since $\bar{\iota}$ is a quasi-isomorphism of DGLA's, denote by $\beta^{\CC\otimes\Omega^\bullet}$ the unique (up to gauge equivalence) Maurer-Cartan element of $\mathfrak{h}_{\CC\otimes\Omega^\bullet}^+$ such that $\bar{\iota}_*( \beta^{\CC\otimes\Omega^\bullet})$ is gauge equivalent to $\widehat{\beta}^{\CC\otimes\Omega^\bullet}$. As in Subsection~\ref{subsec:tcft-dgla}, we may consider $\beta^{\CC\otimes\Omega^\bullet}$ as a Maurer-Cartan element of $\mathfrak{h}_{\CC\otimes\Omega^\bullet}$ via the inclusion $\mathfrak{h}_{\CC\otimes\Omega^\bullet}^+\subset \mathfrak{h}_{\CC\otimes\Omega^\bullet}$.

Next, we proceed to construct the family version of the trivialization map used in Equation~\eqref{eq:trivialization} to compute CEI. Recall from {\em Loc. Cit.} this is an $L_\infty$ quasi-isomorphism
	\[ \mathcal{K}: \mathfrak{h}_{\CC\otimes\Omega^\bullet}  \; \ra \; \mathfrak{h}_{\CC\otimes\Omega^\bullet} ^{\sf triv}.\] 
The construction of $\mathcal{K}$ is similar to that of $K$, and relies on a homotopy trivialization of the twisted sewing map map $\rho^{\CC\otimes\Omega^\bullet}(\mathbb{M})=\rho^{\CC\otimes\Omega^\bullet}\big( \begin{tikzpicture}[baseline={([yshift=-0.4ex]current bounding box.center)},scale=0.3]
\draw [thick] (0,2) circle [radius=2];
\draw [thick] (2,2) to (-0.6,2);
\draw [thick] (-0.8,2.2) to (-0.4,1.8);
\draw [thick] (-0.8, 1.8) to (-0.4, 2.2);
\draw [thick] (2,2) to (3.4,2);
\draw [thick] (3.2,2.2) to (3.6,1.8);
\draw [thick] (3.2, 1.8) to (3.6, 2.2);
\end{tikzpicture}\big)$. That is, we need to construct an operator
\[\mathcal{H} : L_- \otimes L_- \ra \mathbb{K}[t,dt]\]
such that 
\begin{equation}~\label{eq:commutator}
 [ \mathcal{L}_{\m(t)} + uB + d_{DR}+ \mathcal{L}_\xi dt , \mathcal{H} ] (\alpha u^{-i},\beta u^{-j})= \begin{cases}
\rho^{\CC\otimes\Omega^\bullet} (\mathbb{M})(\alpha ,\beta) & \mbox{\; if \;} i=j=0;\\
0 & \mbox{\; otherwise. \;}
\end{cases}
\end{equation}
Indeed, as in~\cite[Section 7.1]{CT}, let us choose a chain level lift of the splitting map $s$ of the form
\[ R= \id +R_1 u +R_2 u^2 + \cdots \in {\sf End}(L_{\CC})[[u]],\]
and let $T$ denote its inverse.
From this, we form the operator
\[R(t):= \exp(-t\xi) _* \circ R  \circ\exp (t\xi)_* \in {\sf End}(L_{\CC}[t])[[u]]\]
which is a chain level lift of the splitting map $\exp(-t\xi)_* \circ s \circ\exp (t\xi)_*$. Denote by $T(t)=\exp(-t\xi) _* \circ T  \circ\exp (t\xi)_*$ the inverse operator of $R(t)$. As in~\cite[Proposition 7.5]{CT}, we set
\begin{align*}
 \mathcal{H} (\alpha u^{-i},\beta u^{-j}) &:= \rho^{\CC\otimes\Omega^\bullet}(M)\Big( (-1)^j \sum_{l=0}^j R(t)_l T(t)_{i+j+1-l} \alpha,\beta\Big)
\end{align*}
To verify the commutator identity we need the following identities:
\begin{itemize}
    \item $\mathcal{L}_{\xi}\circ \exp(t\xi)_*=\exp(t\xi)_*\circ \mathcal{L}_{\xi}$ ;
    \item $\exp(t\xi)_*\circ \mathcal{L}_{\m(t)} = \mathcal{L}_{\m} \circ \exp(t\xi)_*$ .
\end{itemize}
The first identity can be proved, using Lemma \ref{lem:lie_pull}(2), by checking both sides satisfy the same differential equation: $\frac{d}{dt}\gamma(t)= \mathcal{L}_\xi \gamma(t)$, with the same initial condition. Then the identity follows from uniqueness of solutions to ODE's. The second identity holds because $\exp(t\xi)$ is an \Ai homomorphism between the $\m(t)$ and $\m$ \Ai \ structures.

We now compute the left-hand side of 
Equation~\eqref{eq:commutator}, in the case $i=j=0$,
\begin{align*}
&[\mathcal{L}_{\m(t)}+d_{DR}+\mathcal{L}_\xi dt, \mathcal{H}] (\alpha,\beta)\\
=&-\rho^{\CC\otimes\Omega^\bullet}(M)\big(T_1(\mathcal{L}_{\m(t)}+d_{DR}+\mathcal{L}_\xi dt) \alpha, \beta\big)-(-1)^{|\alpha|}\rho^{\CC\otimes\Omega^\bullet}(M)\big(T_1 \alpha, (L_{\m(t)}+d_{DR}+L_\xi dt)\beta\big)\\
=& \rho^{\CC\otimes\Omega^\bullet}(M)\big((\mathcal{L}_{\m(t)}+d_{DR}+\mathcal{L}_\xi dt)T_1 \alpha, \beta\big)-\rho^{\CC\otimes\Omega^\bullet}(M)\big(T_1(\mathcal{L}_{\m(t)}+d_{DR}+\mathcal{L}_\xi dt) \alpha, \beta\big)\\
= & \rho^{\CC\otimes\Omega^\bullet}(M)\big( [\mathcal{L}_{\m(t)},T_1] \alpha, \beta\big) + \rho^{\CC\otimes\Omega^\bullet}(M)\big( (d_{DR} T_1) \alpha, \beta\big) + \rho^{\CC\otimes\Omega^\bullet}(M)\big( [\mathcal{L}_\xi dt, T_1] \alpha,\beta\big).
\end{align*}
Here we used the fact that $\rho^{\CC\otimes\Omega^\bullet}(M)$ is a chain map. Using the second identity above we observe that $[\mathcal{L}_{\m(t)},T(t)_1]= B$, since $[\mathcal{L}_{\m},T_1]= B$, which in turn follows from $R$ being a chain map. For the term $d_{DR} T_1$ we compute
\begin{align*} d_{DR} T(t)_1 & =d_{DR}( \exp(-t\xi)_* T_1 \exp(t\xi)_* )\\
&= \big(\mathcal{L}_{-\xi}\exp(-t\xi)_* T_1 \exp(t\xi)_* + \exp(-t\xi)_* T_1 \mathcal{L}_{\xi} \exp(t\xi)_*\big)dt\\  
&=-[ \mathcal{L}_\xi dt, T(t)_1], 
\end{align*}
using Lemma \ref{lem:lie_pull}(2) and the first identity above. Putting these together we obtain
\[[\mathcal{L}_{\m(t)}+d_{DR}+\mathcal{L}_\xi dt, \mathcal{H}] (\alpha,\beta)= \rho^{\CC\otimes\Omega^\bullet}(M)\big(B \alpha, \beta\big) = \rho^{\CC\otimes\Omega^\bullet}(\mathbb{M})\big(\alpha, \beta\big),\]
as claimed. The other cases with $i\neq 0$ or $j\neq 0$ can be computed similarly using the commutator relations proved in~\cite[Proposition 7.5]{CT}.

Using the homotopy operator $\mathcal{H}$, we obtain an $L_\infty$ map
\[ \mathcal{K}: \mathfrak{h}_{\CC\otimes\Omega^\bullet}  \; \ra \; \mathfrak{h}_{\CC\otimes\Omega^\bullet} ^{\sf triv}.\] 
defined using a stable graph sum formula~\cite[Section 4.8]{CT}. 

Since by construction $\mathcal{K}$ specializes to $K(0)$ and $K(1)$ at $t=0$ and $t=1$ respectively, the push-forward $\mathcal{K}_* \widetilde{\beta^{\CC\otimes \Omega^\bullet}}$ specializes to $\sum_{g,n} F_{g,n}^{\CC,s} \hbar^g\lambda^{2g-2+n}$ and $\sum_{g,n} F_{g,n}^{\CC',f_*sf_*^{-1}}\hbar^g\lambda^{2g-2+n}$ at $t=0$ and $t=1$. The theorem now follows from the following lemma.
	\end{proof}

	\begin{lem}
	Let $\gamma(t)$ be a Maurer-Cartan element of the trivialized DGLA $\mathfrak{h}_{\CC\otimes\Omega^\bullet} ^{\sf triv}$, i.e. it satisfies the equation
	\[ (\mathcal{L}_{\m(t)}+uB+d_{DR}+\mathcal{L}_\xi dt ) \gamma(t) =0.\]
	Then we have $[\exp(-\xi)_* \gamma(0)] = [\gamma(1)]$ in the $(b+uB)$-homology $H_*(\mathfrak{h}_{\CC'}^{\sf triv})$.
	\end{lem}
	
	\begin{proof}
	Writing $\gamma(t)$ as $\theta(t) + \eta(t)dt$, the Maurer-Cartan equation yields
	\begin{align}\label{eq:MCfamily}
	(\mathcal{L}_{\m(t)}+uB) \theta(t) &=0\nonumber\\
	\frac{d}{dt} \theta(t) + (\mathcal{L}_{\m(t)}+uB)\eta(t) + \mathcal{L}_\xi \theta(t) & = 0
	\end{align}
	We shall prove the difference $D(t):=\exp(-t\xi)_*\theta(0)- \theta(t)\equiv 0$ in the $(\mathcal{L}_{\m(t)}+uB)$-homology of all fixed $t\in [0,1]$. Using Lemma \ref{lem:lie_pull} we differentiate,
	\begin{align*}
\frac{d}{dt} D(t) &=-\mathcal{L}_\xi (\exp(-t\xi)_*\theta(0)) - \frac{d}{dt}\theta(t)\\
&= -\mathcal{L}_\xi (\exp(-t\xi)_*\theta(0)) +  (\mathcal{L}_{\m(t)}+uB)\eta(t) + \mathcal{L}_\xi \theta(t)\\
&= -\mathcal{L}_\xi D(t) + (\mathcal{L}_{\m(t)}+uB)\eta(t),
	\end{align*}
    and use Equation (\ref{eq:MCfamily}) in the second equality.

   Let $X(t)$ be the unique solution to the differential equation
   \[\frac{d}{dt} X(t)= - \mathcal{L}_\xi X(t) + \eta(t), \ X(0)=0.  \]
	Note that since $\xi$ has order two, this equation has indeed a unique solution. Now we compute:
    \begin{align*}
\frac{d}{dt}(\mathcal{L}_{\m(t)}+uB)X(t) &=\mathcal{L}_{\frac{d\m}{dt}}X(t) + \mathcal{L}_{\m(t)}\left(- \mathcal{L}_\xi X(t) + \eta(t)\right)+ uB\left(- \mathcal{L}_\xi X(t) + \eta(t)\right)\\
&= - \mathcal{L}_\xi\mathcal{L}_{\m(t)}X(t)- \mathcal{L}_\xi(uB X(t))+ (\mathcal{L}_{\m(t)}+uB)\eta(t)\\
&= -\mathcal{L}_\xi (\mathcal{L}_{\m(t)}+uB)X(t) + (\mathcal{L}_{\m(t)}+uB)\eta(t).
	\end{align*}
Here we have used the pseudo-isotopy condition $\frac{d\m}{dt}=[\m(t),\xi]$ together with Lemma \ref{lem:lieder}. Therefore $D(t)$ and $(\mathcal{L}_{\m(t)}+uB)X(t)$ satisfy the same differential equation and initial condition at $t=0$. Hence, by uniqueness of solutions of ODE's, $D(t)=(\mathcal{L}_{\m(t)}+uB)X(t)$, which proves the claim.   
    \end{proof}

\begin{cor}\label{coro:cyclic-inv}
Let $\CC$ and $\CC'$ be minimal, unital, cyclic \Ai-categories and let $f$ be an unital and cyclic $A_\infty$-isomorphism
	\[ f=(f_1,f_2,\ldots): \big(\CC,\langle-,-\rangle,\{\m_k\}\big) \;\ra\; \big(\CC',\langle-,-\rangle',\{\m'_k\}\big)\]
(not necessarily with $f_1=\id$). Then we have
	\[ f_* F_{g,n}^{\CC,s} = F_{g,n}^{\CC', f_* s f_*^{-1}}.\]
\end{cor}

\begin{proof}
We may factor the $A_\infty$-isomorphism $f$ as a composition
\[ \CC \stackrel{h}{\longrightarrow} \widetilde{\CC} \stackrel{g}{\longrightarrow} \CC', \]
where $g$ is a linear $A_\infty$-isomorphism and $h$ has $h_1=\id$. We define $\widetilde{\CC}$ to have the same objects, morphism spaces and pairing as $\CC$ and define the \Ai-operations as $\widetilde{\m}_k= f_1^{-1}\circ \m'_k \circ f_1^{\otimes k}$.  This is possible since by the assumptions, $f_1$ are linear isomorphisms. We then define $h_k:= f_1^{-1}\circ f_k$, for $k\geq 1$ and $g_1=f_1$, $g_k=0$ for $k\geq 2$. It is easy to check these define cyclic \Ai-functors.

The functor $h$ satisfies the conditions in Theorem \ref{thm:cyclic-inv} and hence preserves CEI. Since the $A_\infty$-isomorphism $g$ is linear, it induces isomorphisms on $\mathfrak{h}_\CC$ and $\widehat{\mathfrak{h}}_\CC$, hence the invariance of CEI is obvious in this case.
\end{proof}

\subsection{Morita invariance of CEI}\label{sec:definition}

In this subsection, we formulate and prove the Morita invariance of CEI. Let $\CC$ be an $A_\infty$-category (not necessarily cyclic) that is proper and smooth. We also assume that it satisfies the Hodge-to-de-Rham degeneration property. Note that in the $\Z$-graded case, this condition is automatic by Kaledin's theorem~\cite{Kal}.

As in Subsection \ref{sec:trivCY} we shall refer to elements of $HH_\bullet(\CC)$ ($HC_\bullet^-(\CC)$) satisfying the same non-degeneracy condition as weak (strong) Calabi-Yau structures, via the duality isomorphisms:
\begin{align*}
HH_\bullet (\CC)^\vee & \cong HH_\bullet(\CC)\\
HC_\bullet^\lambda(\CC)^\vee & \cong HC_\bullet^-(\CC)
\end{align*}

Fix a parity $d\in \Z/2\Z$. Associated with the $A_\infty$-category $\CC$  we define a set $ \mathcal{M}^d_\CC$ consisting of pairs $(\omega, s)$ such that 
\begin{itemize}
\item[--] $\omega\in HH_\bullet (\CC)$ is a weak Calabi-Yau structure of parity $d$, 
\item[--] $s$ is an unital splitting with respect some split-generating subcategory $\mathcal{A}\subset \tw^\pi\CC$ (see Definition~\ref{defi:splitting}).
\end{itemize}

We refer to a pair $(\omega, s)$ as an \emph{extended Calabi--Yau structure} on $\CC$.

One may also define variants of $\mathcal{M}_\CC^d$ by requiring $s$ be good, or $\omega$-compatible, or both. Since the Mukai pairing and the $u$-connection are Morita invariant~\cite{She}, all these variants are also preserved under Morita equivalences.
 
Fix a pair of natural numbers $(g,n)$ such that $2g-2+n>0$. We proceed to define a function
\[ F^\CC_{g,n}: \mathcal{M}^d_\CC \ra {\sf Hom}^c\big( HH_\bullet(\CC)[d][[u]]^{\otimes n}, \mathbb{K}\big).\]
Indeed, given a pair $(\omega, s)\in \mathcal{M}_\CC^d$, since $s$ is unital, we may choose a split-generating subcategory $\mathcal{A}\subset \tw^\pi\CC$ such that $s(\omega)$ admits a lift to an unital Calabi-Yau structure of $\mathcal{A}$. We then apply the homological perturbation lemma to $\mathcal{A}$ and produce a minimal \Ai-algebra $\mathcal{A}_0$ quasi-isomorphic to $\mathcal{A}$. Thus, we may apply Proposition~\ref{prop:cyclic-even} (if $\omega$ is even) or Proposition~\ref{prop:cyclic-odd} (if $\omega$ is odd) to $\mathcal{A}_0$ and obtain an unital and cyclic $A_\infty$-model:
\[ G_\mathcal{A}: \mathcal{A}' \to \mathcal{A}_0 \to\mathcal{A}.\]
Denote by $(G_\CC)_*$ the induced map on the Hochschild invariants associated with the following zig-zag compositions:
\[ G_\CC: \mathcal{A}' \to \mathcal{A} \hookrightarrow \tw^\pi \CC \hookleftarrow \CC.\]
Define $F^\CC_{g,n}(\omega,s)$ to be the linear functional $F^{\CC,\omega, s}_{g,n}$ explicitly given by
\[ F^{\CC,\omega,s}_{g,n}(\alpha_1,\ldots,\alpha_n) := \langle (G_\CC)_*^{-1} \alpha_1,\ldots,(G_\CC)_*^{-1} \alpha_n \rangle_{g,n}^{\mathcal{A}', (G_\CC)_*^{-1} s (G_\CC)_*}.\]

\begin{lem}
The definition of $F^{\CC,\omega,s}_{g,n}$ is independent of the choice of the category $\mathcal{A}$ as well as its unital cyclic $A_\infty$-model $\mathcal{A}'$. 
\end{lem}

\begin{proof}
When $\omega$ is even, Proposition~\ref{prop:cyclic-even} guarantees any two cyclic, unital models are related by an unital, cyclic \Ai-isomorphism.  Applying Corollary~\ref{coro:cyclic-inv} to this \Ai-isomorphism implies the independence of the choice of $\mathcal{A}'$. To argue the independence of the choice of $\mathcal{A}$. If $s(\omega)$ is unital with respect to $\cA_1$ and $\cA_2$, both split-generate $\tw^\pi \CC$, then we may consider a third subcategory $\cA\subset \tw^\pi\CC$ defined by the union of objects in $\cA_1$ and $\cA_2$. Clearly, $\cA$ contains both $\cA_1$ and $\cA_2$. Furthermore, $s(\omega)$ is also unital with respect to $\cA$. Thus, we may choose an unital cyclic model $\cA'$ of $\cA$, which induces cyclic models $\cA_1'$, $\cA_2'$ of $\cA_1$ and $\cA_2$ respectively. At this point, it suffices to apply Proposition~\ref{lem:inclusion} to conclude that the invariants of both $\cA_1'$ and $\cA_2'$ match the invariants of $\cA'$.

In the odd case, the cyclic unital model $\cA'$ may not be unique up to cyclic unital $A_\infty$-isomorphisms, as shown in Example~\ref{ex:non-existence}. However, we may use Theorem~\ref{thm:clifford} proved in the Appendix~\ref{app:b-f} to argue the desired independence. Indeed, assume that we have two unital cyclic models $G_1: \cA_1' \to \cA$ and $G_2: \cA_2' \to \cA$. Tensoring with the Clifford algebra $\Cl$ yields two unital cyclic models of $\cA\otimes \Cl$. But now we are back to the even case, where the unital cyclic model is unique up to unital cyclic $A_\infty$-isomorphisms. Together with Theorem~\ref{thm:clifford}, it follows that $F^\CC_{g,n}$ is independent of the choice of the unital cyclic $A_\infty$-model $\cA'$. The independence of the choice of $\cA$ is the same as the even case.
\end{proof}

Now we formulate and prove the Morita invariance of $F^{\CC,\omega,s}_{g,n}$. As proved in \cite[Theorem A.3]{She} two \Ai-categories are Morita equivalent if and only if ${\tw}^\pi\CC$ and ${\tw}^\pi\DD$ are quasi-equivalent, where ${\tw}^\pi\CC$ is the triangulated split-closure of $\CC$.  Inspired by \cite{GPS} we give the definition.

\begin{defn}\label{defn:moreq}
	Let $\CC$ and $\DD$ be \Ai-categories with extended Calabi--Yau structures $(\omega_\CC, s_\CC)$ and $(\omega_\DD, s_\DD)$. We say $(\CC, \omega_\CC, s_\CC)$ and $(\DD, \omega_\DD, s_\DD)$ are Morita equivalent if there is a quasi-equivalence $f: {\tw}^\pi\CC \to {\tw}^\pi\DD$ with $[f_*\omega_\CC]=[\omega_\DD]$ and $f_* \circ s_\CC = s_\DD\circ f_*$.
\end{defn}

\begin{thm}~\label{thm:main}
Let $(\CC, \omega_\CC, s_\CC)$ and $(\DD, \omega_\DD, s_\DD)$ be \Ai-categories with extended Calabi--Yau structures and let $f: {\tw}^\pi\CC \to {\tw}^\pi\DD$ be a Morita equivalence between these extended Calabi--Yau categories. Then the CEI of $(\CC, \omega_\CC, s_\CC)$ and $(\DD, \omega_\DD, s_\DD)$ agree. More precisely,  
 for any $(g,n)$ such that $2g-2+n>0$, we have
\[ F_{g,n}^{\CC, \omega_\CC, s_\CC}(\alpha_1,\ldots,\alpha_n) = F_{g,n}^{\DD, \omega_\DD, s_\DD}\big(  f_*(\alpha_1),\ldots, f_*(\alpha_n)\big).\]
\end{thm}

\begin{proof}
The left hand side is defined using the CEI of an unital cyclic model $\cA'$ of $\cA\subset \tw^\pi\CC$. Since $\CC$ is a small category, by applying Proposition~\ref{lem:inclusion} we may replace $\cA$ by its skeleton. Equivalently, we may assume that no two objects in $\cA$ are isomorphic. Now, since $f$ is an equivalence, it preserves the Hochschild invariants as well as the higher residue pairing. Thus, the image full subcategory $f(\cA)\subset \tw^\pi\DD$ is split-generating, and with respect to which the pair $(\omega_\DD, s_\DD)$ is unital. Thus, the unital cyclic model $\cA'$, via the composition
\[ f \circ G_\cA: \cA' \to \cA \to f(\cA),\]
is an unital cyclic model of $f(\cA)$ whose CEI are by definition the right hand side.
\end{proof}

\section{Examples and Applications}\label{sec:ex-app}

\subsection{CEI of Frobenius associative algebras}\label{subsec:frob}

Throughout this subsection, let $A$ be a finite-dimensional, minimal, $\Z/2\Z$-graded, cyclic $A_\infty$ algebra such that its higher products $\m_3, \m_4,\ldots$ all vanish. As before, we continue to assume $A$ is unital, smooth, and satisfies the Hodge-to-de-Rham degeneration property. The cyclic pairing is of parity $d\in \Z/2\Z$. We shall refer to such an $A$ as a Frobenius associative algebra. Recall that the notation $L=C_\bullet^{red}(A)[d]$ is the shifted reduced Hochschild chain complex of $A$.

\begin{example}\label{ex:field}
	Consider the ground field  $\mathbb{K}$ as an $A_\infty$-category with one object. Its Hochschild invariants are $HH_\bullet(\mathbb{K})\cong \mathbb{K}$ and $HC_\bullet^-(\mathbb{K})\cong \mathbb{K}[[u]]$. In this case, the parity $d$ necessarily equals $0\in \Z/2\Z$. Furthermore, it has an unique unital splitting $s^{\sf can}(1)=1$. Thus, the set $\mathcal{M}_\mathbb{K} \cong \mathbb{K}^*$ with the correspondence given by
	\[ \lambda \in \mathbb{K}^* \mapsto (\lambda, s^{\sf can}) \in \mathcal{M}_\mathbb{K}.\]  
	It was shown in~\cite{Tu} that we have
	\[ F^{\mathbb{K},\lambda,s^{\sf can}}_{g,n} (u^{k_1},\ldots,u^{k_n}) = \lambda^{2-2g-n}\cdot \int_{[\overline{M}_{g,n}]} \psi_1^{k_1}\cdots\psi_n^{k_n}\]
	where $\psi_j$'s are the $\psi$-classes on the Deligne-Mumford moduli space $\overline{M}_{g,n}$. 
\end{example}

For a general Frobenius associative algebra, we need a few more lemmas.

\begin{lem}~\label{lem:mukai-vanish}
	The chain level Mukai pairing $\langle-,-\rangle_{\sf Muk}: L\otimes L \ra \mathbb{K}$ satisfies 
    \[ \langle a_0|a_1|\cdots|a_k, b_0|b_1|\cdots|b_l\rangle_{\sf Muk} = 0, \mbox{ if $k\geq 1$ or $l\geq 1$.}\]
\end{lem}

\begin{proof}
	This follows from the definition of the Mukai pairing, since if $k\geq 1$ or $l\geq 1$, then one needs higher products in order to be non-zero.
\end{proof}

\begin{lem}~\label{lem:surj}
	The canonical inclusion map $A[d] \to L$ as length zero Hochschild chains induces a surjective map $A[d] \to H_\bullet(L)$ in homology.
\end{lem}

\begin{proof}
	Since $A$ only has $\m_2$, the homology $H_\bullet(L)$ is also graded by the length of Hochschild chains. Let $\alpha=[a_0|a_1|\cdots|a_k]\in H_\bullet(L)$ be a Hochschild homology class with $k\geq 1$. Then by the previous lemma we have
	\[ \langle \alpha, \beta\rangle_{\sf Muk} = 0, \;\; \forall \beta \in H_\bullet(L).\]
	But $A$ is smooth and proper, by Shklyarov's non-degeneracy result we must have $\alpha=0$. Thus any Hochschild homology class must be represented by a length zero chain.
\end{proof}

\begin{lem}~\label{lem:s-can}
	There exists a unique splitting denoted by $s^{\sf can}: H_\bullet(L) \ra H_\bullet(L_+)$ of the Hodge filtration of $A$ characterized by the property that $\nabla_{u\frac{\partial}{\partial u}} s^{\sf can}(\alpha) = 0, \;\; \forall \alpha\in H_\bullet(L)$.
\end{lem}

\begin{proof}
	Since $\m_2$ is the only non-zero \Ai \ operation, the Hochschild class $\m'$ (in the notation of \cite{AT}) vanishes, hence the $u$-connection is given by the operator
	\[ \nabla_{u\frac{\partial}{\partial u}} (a_0|a_1|\cdots|a_k \cdot u^n) = (n-\frac{k}{2}) \cdot a_0|a_1|\cdots|a_k \cdot u^n.\]
	Let $\alpha=[a] \in H_\bullet(L)$ be a Hochschild class represented by a length zero chain $a\in A[d]$. Since both the Hochschild differential $b$ and the Connes operator $B$ are homogeneous respect to the length grading, i.e. $b$ is degree $-1$ while $B$ is degree $1$, we may choose $s(\alpha)$ to be of the form
	\[ s^{\sf can}(\alpha)= [ a+ \alpha_1 u +\alpha_2 u^2 + \cdots], \;\; \alpha_k \in A\otimes \overline{A}^{\otimes 2k}[d], \; \forall k\geq 1.\]
	Then one easily checks that we have $\nabla_{u\frac{\partial}{\partial u}} s(\alpha) = 0$. The uniqueness follows from the non-degeneracy of the higher residue pairing and Lemma~\ref{lem:mukai-vanish}.
\end{proof}

Recall from~\cite[Section 2.3]{AT} the cyclic pairing on $A$ induces a duality isomorphism
\[ D: HH^\bullet(A) \cong HH_\bullet(A)[d]=H_\bullet(L)\]
Let us denote by $\phi := D(\one_A)$. By construction, $\phi$ is a weak Calabi-Yau structure. 

\begin{lem}
	Under the correspondence in Corollary~\ref{cor:CY_SHIP}, the cyclic structure in $A$ corresponds to the strong Calabi-Yau structure $s^{\sf can}\big( \phi \big)\in H_\bullet(L_+)=HC_\bullet^-(A)[d]$.
\end{lem}
\begin{proof}
	We denote $s^{\sf can}(\phi)=\sum_{k\geq 0} \phi_k u^k$ and compute the action of this Calabi--Yau structure on $HC_\bullet^+(A)$:
	\[\langle \sum_{k\geq 0} u^{-k} \alpha_k, s^{\sf can}( \phi ) \rangle= \sum_{k\geq 0} \langle\alpha_k, \phi_k \rangle_{{\sf Muk}}= \langle\alpha_0,  D(\one_A)\rangle_{{\sf Muk}}.\]
	In the second equality, we have used the fact that all the $\phi_k,\ k>0$ can be represented by chains of positive length (see proof of Lemma \ref{lem:s-can}), and therefore the corresponding pairings vanish by Lemma \ref{lem:mukai-vanish}. Again, by Lemma \ref{lem:mukai-vanish}, the remaining term also vanishes unless $\alpha_0$ has length zero, in which case $\langle\alpha_0,  D(1_A)\rangle_{{\sf Muk}}= \langle\one_A, \alpha_0\rangle$ by definition of $D$. 
	
	As computed in Example \ref{ex:cyclicCY}, the cyclic pairing in $A$ corresponds to the Calabi--Yau structure $\Phi$ defined in (\ref{eq:CY}).  The comparison map between the $u$-model and cyclic complexes for cyclic homology (see \cite{Lod}) sends $\sum_{k\geq 0} \phi_k u^{-k}$ to $[\phi_0]$. Therefore, we conclude that $\Phi$ agrees with the map described above. 
\end{proof}

Observe that $s^{\sf can}$ is $\phi$-compatible, with the scalar $r=0$ in Condition $(S4.)$ Definition~\ref{defi:splitting}. Since $A$ is a Frobenius associative algebra, the $u$-connection has a simple pole, as we saw above. Hence, we are in the second case of Proposition \ref{prop:comp_unital}, which implies $s^{\sf can}$ is unital. Thus, we may use the Frobenius algebra $A$ as a cyclic model to compute the CEI of the triple $(A, \phi, s^{\sf can})$.

Using the duality map $D$ above, we can transfer the cup product on $HH^\bullet(A)$ to $H_\bullet(L)$. Together with the Mukai pairing, we obtain an unital commutative Frobenius algebra $\big( H_\bullet(L),\phi, \cup, \langle-,-\rangle_{\sf Muk} \big)$. We refer to~\cite[Theorem 2.4]{AT} for a proof. It is well known that a commutative Frobenius algebra is equivalent to the structure of a $2$-dimensional topological field theory. In particular, this yields linear functionals
\begin{equation}~\label{eq:tqft}
	\omega_{g,n}^A: H_\bullet(L)^{\otimes n} \ra \mathbb{K}
\end{equation}

\begin{thm}
	The CEI of the triple $(A,\phi,s^{\sf can})$ are given by
	\[ F_{g,n}^{A,\phi,s^{\sf can}}(\alpha_1\cdot u^{k_1},\ldots,\alpha_n\cdot u^{k_n})= \omega_{g,n}^A(\alpha_1,\ldots,\alpha_n) \cdot \int_{[\overline{M}_{g,n}]} \psi_1^{k_1}\cdots\psi_n^{k_n}\]
	where $\alpha_1,\ldots,\alpha_n \in H_\bullet(L)$, and $k_1,\ldots,k_n\geq 0$.
\end{thm}

\begin{proof}
	As in the proof of Lemma~\ref{lem:s-can}, since the operators $b$ and $B$ are both homogeneous with respect to the length of Hochschild chains, we may choose a chain level lift $R: L \to L_+$ of the splitting operator $s^{\sf can}$ such that it is of the form
	\[ R(\alpha) = \alpha + u R_1\alpha +u^2 R_2\alpha + \cdots,\]
	with $R_n \alpha \in A\otimes \overline{A}^{2n+k}, \forall n\geq 0$ if $\alpha \in A\otimes \overline{A}^k$. Using Lemma~\ref{lem:mukai-vanish}, we see that $R$ trivially satisfies the symplectic condition, i.e.
	\[ \langle \alpha, \beta \rangle_{\sf Muk} = \langle R(\alpha), R(\beta) \rangle_{\sf hres}, \;\; \forall \alpha, \beta \in L.\]
	To this end, we may use of the following explicit formula from~\cite{CT} expressing $F_{g,n}$'s as a type of graph sum:
	\[ F_{g,n}^{A,\phi,s^{\sf can}}(\alpha_1\cdot u^{k_1},\ldots,\alpha_n\cdot u^{k_n}) =\sum_{j=1}^n \sum_{\mathbb{G}} {\sf wt}(\mathbb{G}) \cdot  \langle\rho(\mathbb{G})(\alpha_j u^{k_j}), \alpha_1 u^{k_1}\cdots \widehat{\alpha_ju^{k_j}}\cdots\alpha_n u^{k_n} \rangle_{\sf Muk}\]
	By Lemma~\ref{lem:surj}, we may assume that $\alpha_i= a_i \in A[d]$. Again by the vanishing in Lemma~\ref{lem:mukai-vanish}, in the above summation only the star graph $\star_{g,1,n-1}$ with one input and $n-1$ outputs can contribute. This reduces the above summation to
	\[ F_{g,n}^{A,\phi,s^{\sf can}}(\alpha_1\cdot u^{k_1},\ldots,\alpha_n\cdot u^{k_n}) =\sum_{j=1}^n  \frac{1}{(n-1)!}  \langle\rho(\mathcal{V}_{g,1,n-1})(\alpha_j u^{k_j}), \alpha_1 u^{k_1}\cdots \widehat{\alpha_ju^{k_j}}\cdots\alpha_n u^{k_n} \rangle_{\sf Muk}\]
	But since $\m_2$ is the only non-trivial operation in $A$, only degree zero ribbon graphs in $\mathcal{V}_{g,1,n-1}$ can contribute. Thus the invariant $ F_{g,n}^{A,\phi,s^{\sf can}}(\alpha_1\cdot u^{k_1},\ldots,\alpha_n\cdot u^{k_n})$ is the product of two parts: the first part is by action of degree zero ribbon graphs which give exactly the topological part $\omega^A_{g,n}(\alpha_1,\ldots,\alpha_n)$; the second part is the $\psi$-class contribution, which is computed in~\cite{Tu} and is given by  $\int_{[\overline{M}_{g,n}]} \psi_1^{k_1}\cdots\psi_n^{k_n}$.
\end{proof}

\begin{cor}
	Assume that $H_\bullet(L)\cong \mathbb{K}$ as a Frobenius algebra. Then we have
	\[ F_{g,n}^{A,\phi,s^{\sf can}}(\phi\cdot u^{k_1},\ldots,\phi\cdot u^{k_n})=  \int_{[\overline{M}_{g,n}]} \psi_1^{k_1}\cdots\psi_n^{k_n}.\]
\end{cor}

\begin{example}
	Consider the Clifford algebra ${\sf Cl}$ introduced in Example \ref{ex:non-existence}. It was computed in~\cite{CLT} that $H_\bullet(L)\cong \mathbb{K}$ as a Frobenius algebra. The weak Calabi-Yau element is $[\epsilon] \in H_\bullet(L)$. Thus, we have
	\[ F_{g,n}^{{\sf Cl},[\epsilon],s^{\sf can}}([\epsilon]\cdot u^{k_1},\ldots,[\epsilon]\cdot u^{k_n})=  \int_{[\overline{M}_{g,n}]} \psi_1^{k_1}\cdots\psi_n^{k_n}.\]
\end{example}

\begin{example}
 There is a generalization of the previous example. The Clifford algebra ${\sf Cl}_d= \mathbb{K}[\epsilon_1,\ldots,\epsilon_d]$ generated by $d$ odd elements, with the relations $\m_2(\epsilon_i,\epsilon_j)+\m_2(\epsilon_j,\epsilon_i)= \delta_{ij}$. Again, the Hochschild homology $H_\bullet(L)$ has rank one, generated by $[\epsilon_1\cdots\epsilon_d]$. Using Lemma \ref{lem:5-3} for the Mukai graph (or \cite[Lemma 5.3]{AT}) one computes 
 \[\langle[\epsilon_1\cdots\epsilon_d], [\epsilon_1\cdots\epsilon_d]\rangle_{\sf Muk} = (-1)^{\frac{d(d-1)}{2}}.\]
 We choose the weak Calabi-Yau structure $\phi:=(\sqrt{-1})^{d(d-1)/2}[\epsilon_1\cdots\epsilon_d]$. Under the correspondence in Example \ref{ex:cyclicCY}, this Calabi-Yau structure is equivalent to equipping ${\sf Cl}_d$ with a cyclic structure determined by $\langle 1, \epsilon_1\cdots\epsilon_d\rangle = (\sqrt{-1})^{3d(d-1)/2}$. Moreover, with this choice of weak Calabi-Yau structure we have $H_\bullet(L)\cong \mathbb{K}$ as a Frobenius algebra, therefore
 \[ F_{g,n}^{{{\sf Cl}_d},\phi,s^{\sf can}}(\phi\cdot u^{k_1},\ldots,\phi\cdot u^{k_n})=  \int_{[\overline{M}_{g,n}]} \psi_1^{k_1}\cdots\psi_n^{k_n}.\]
\end{example}

\begin{rmk}
	There are many splittings of the Hodge filtration for ${\sf Cl}$, other than $s^{\sf can}$. These other splitting are unital (since we are in the odd case) but not homogeneous. We expect the CEI for these other splittings will give different (including non homogeneous) Cohomological Field Theories. In particular, we expect that for an appropriately chosen splitting for ${\sf Cl}$, one recovers the Cohomological Field Theory given by the total Chern class of the Hodge bundle on $\overline{M}_{g,n}$ \cite[Section 1]{Pan}. We will further study this problem in a future work. 
\end{rmk}

The next corollary serves to illustrate how one can use our Morita invariance result to compute the CEI of \Ai-categories relevant to Mirror Symmetry.

\begin{cor}
	Let $W\in\mathbb{K}[[x_1,\ldots, x_n]]$ be a potential with an isolated, Morse (non-degenerate) singularity at the origin. The category of matrix factorizations $\MF(W)$ is smooth, proper, has the Hodge-to-de-Rham degeneration property and $HH_\bullet(\MF(W))\cong \mathbb{K}\cdot\phi$. Moreover there is an unique $\phi$-compatible (and unital) splitting $s^{\sf can}$ and the corresponding CEI are equal to
	\[ F_{g,n}^{\MF(W),\phi,s^{\sf can}}(\phi\cdot u^{k_1},\ldots,\phi\cdot u^{k_n})=  \int_{[\overline{M}_{g,n}]} \psi_1^{k_1}\cdots\psi_n^{k_n}.\]
\end{cor}
\begin{proof}
	The facts that $\MF(W)$ is smooth, proper and has the Hodge-to-de-Rham degeneration property are proved in \cite{Dyc}, and in fact hold without the Morse assumption. It is also proved in \cite{Dyc} that $\MF(W)$ has a compact generator $\mathbb{K}^{stab}$, and the cohomology of the endomorphism dg-algebra $end(\mathbb{K}^{stab})$ is isomorphic to a Clifford algebra corresponding to the Hessian of $W$. When $W$ has a Morse singularity the Clifford algebra is intrinsically formal \cite[Section 6.1]{She1} and therefore $end(\mathbb{K}^{stab})$ is quasi-isomorphic to the Clifford algebra ${\sf Cl}_n$. Hence $\MF(W)$ is Morita equivalent to ${\sf Cl}_n$. The result now follows from Theorem \ref{thm:main} and the previous example.
\end{proof}

\subsection{A-model comparison with Gromov-Witten theory}
Let $X$ be a compact symplectic manifold of real dimension $2d$ and let $\CC=\Fuk(X)$ be the Fukaya category of $X$. This is an \Ai-category linear over $\mathbb{K}=\Lambda$ the Novikov field  (with complex coefficients).  It is natural to expect that for the ``correct" choice of Calabi--Yau structure and splitting, the  induced CEI should coincide with the geometrically defined Gromov--Witten invariants of $X$.

We expect the correct splitting to be determined by the geometrically defined \emph{open-closed map}
\[\OC: HH_\bullet\big( \Fuk(X)\big) \to H^\bullet(X,\Lambda)[d],\]
and its cyclic enhancement
\[\OC^{\sf cyc}: HC_\bullet^-\big( \Fuk(X)\big) \to H^\bullet(X,\Lambda)[d][[u]].\]
The open-closed map $\OC$ has been constructed in several settings (starting with \cite{FOOOb}) and is expected to be an isomorphism for a wide class of symplectic manifolds. The cyclic version $\OC^{\sf cyc}$ was more recently constructed by Ganatra \cite{Gan}. Although the construction in \cite{Gan} is carried out in a somewhat limited technical setup, the arguments should generalize to any other setting where one constructs the Fukaya category.

From now on we make the assumption on $X$ that the open-closed map $\OC$ is an isomorphism. Then, as explained in \cite{Gan} for example, one can define a weak Calabi--Yau structure on $\Fuk(X)$ by taking the composition
\[\omega^{\OC}: HH_\bullet\big( \Fuk(X)\big)\stackrel{\OC}{\longrightarrow}H^\bullet(X,\Lambda)[d]\stackrel{\int}{\longrightarrow}\Lambda[-d],\]
where $\int$ is the integration map (or Poincare pairing) on $X$.

Assuming the $\OC$ map is an isomorphism also implies that $\OC^{\sf cyc}$ is an isomorphism. Therefore it determines a splitting $s^\OC$, by requiring the following diagram be commutative
\[\begin{CD}
	HH_\bullet\big( \Fuk(X)\big) @> s^{{\mathcal{OC}}}>> HC_\bullet^-\big( \Fuk(X)\big)[d]\\
	@VV {\mathcal{OC}} V       @VV {\mathcal{OC}}^{\sf cyc} V \\
	H^\bullet(X,\Lambda)[d]  @> i >> H^\bullet(X,\Lambda)[d][[u]],
\end{CD}\]
where $i$ is the inclusion by constant map. It should follow from general properties of open-closed maps that $s^\OC$ is a good, $\omega^\OC$-compatible and unital splitting. Therefore, we can use these to construct the CEI and obtain maps 
\[ F^{\Fuk(X)}_{g,n}(\omega^\OC, s^\OC; - ):  HH_\bullet(\Fuk(X))[[u]]^{\otimes n} \to  \Lambda\]
for each pair of integers $(g,n)$ such that $2-2g-n<0$. We make the following conjecture.
\begin{conj}\label{conj:a-model}
Assume the open-closed map $\OC: HH_\bullet\big( \Fuk(X)\big) \to H^\bullet(X,\Lambda)[d]$ is an isomorphism and let $( \omega^\OC, s^{\mathcal{OC}})$ be the extended Calabi--Yau structure defined above.
 Then, for any $\alpha_1,\ldots, \alpha_n \in HH_\bullet\big( \Fuk(X)\big)$ we have
\[ F^{\Fuk(X)}_{g,n}\big( \omega^\OC, s^{\mathcal{OC}};\alpha_1u^{k_1},\ldots,\alpha_nu^{k_n}\big) = \langle \mathcal{OC}(\alpha_1)\psi^{k_1},\ldots,\mathcal{OC}(\alpha_n)\psi^{k_n}\rangle_g^{X},\]
where the right-hand side denotes the descendant Gromov--Witten invariants of $X$.
\end{conj}

In \cite{AT} the authors use the closed-open map (dual to the open-closed map above) to construct a splitting $s^\mu$ for $\Fuk(X)$, under the extra assumption that $HH^\bullet(\Fuk(X))$ is a semi-simple ring. One can show that $s^\mu$ agrees with the splitting $s^\OC$ in the conjecture. Furthermore, it is proved in \cite{AT} that for the splitting $s^\mu$ the genus zero categorical Gromov--Witten invariants, computed using a categorical analogue of Saito's primitive forms, agree with the Gromov--Witten invariants of $X$. Even though it is not yet known these invariants agree with the genus zero CEI, this provides some evidence for the above conjecture.

\subsection{B-model: derived invariants of smooth projective Calabi-Yau's}

Consider the case when $\CC= D^b_{dg}\big({\sf Coh}(X)\big) $ is the dg enhancement of the derived category of coherent sheaves on a smooth projective Calabi-Yau variety $X$ over $\mathbb{C}$. Blanc~\cite{Bla} developed a theory of complex topological $K$-theory for $\mathbb{C}$-linear differential graded categories, based on To\"en's proposal~\cite{Toe}. An immediate consequence of this construction is that for the category $D^b_{dg}\big({\sf Coh}(X)\big)$ there exists an intrinsic splitting of the non-commutative Hodge filtration which we denote by $s^{\sf BT}$ (for ``Blanc-To\"en"). Under the comparison result obtained by Blanc~\cite[Theorem 1.1 (b), (d)]{Bla}, this splitting corresponds to the complex conjugate splitting of the classical Hodge filtration of $H^\bullet(X,\mathbb{C})$, through a comparison map (see~\cite[Section 4.6]{Bla}) $HP_\bullet(\CC)\cong H^\bullet(X,\mathbb{C})$.

Assume the complex dimension of $X$ is $d$. Using the splitting $s^{\sf BT}$, the CEI give a map
\[ F_{g,n}^{\CC, s^{\sf BT}}: HH_d(\CC)^* \times HH_\bullet(\CC)[d][[u]]^{\otimes n} \ra \;\mathbb{C}.\]
Note that here the space of weak (=strong) Calabi-Yau structures $HH_d(\CC)^*$  can be identified with the space of nowhere vanishing holomorphic volume forms on $X$ which is a $\mathbb{C}^*$-torsor since $X$ is a smooth and projective Calabi-Yau. Since the splitting map $s^{\sf BT}$ is intrinsic to the dg-category $\CC$, our main theorem above implies that $F_{g,n}^{\CC, s^{\sf BT}}$ only depends on the dg category $\CC=D^b_{dg}\big({\sf Coh}(X)\big) $. In fact, by Lunts-Orlov's uniqueness of dg-enhancements~\cite{LunOrl}, we can conclude that for each pair $(g,n)$, the map $F_{g,n}^{\CC, s^{\sf BT}}$ is an invariant of the derived category $D^b\big({\sf Coh}(X)\big)$!

In the case $d=0$ when $X={\sf Spec} \,\mathbb{C}$,  and we set the weak Calabi-Yau structure $\phi=1 \in HH_0(\mathbb{C})^*=\mathbb{C}^*$, the CEI agree with the Gromov-Witten invariants of a point, i.e. 
\[ F^{\mathbb{C}}_{g,n} (u^{k_1},\ldots,u^{k_n}) =  \int_{[\overline{M}_{g,n}]} \psi_1^{k_1}\cdots\psi_n^{k_n},\]
as explained in Example \ref{ex:field}.

In the case $d=1$ when $X_\tau = \mathbb{C}/ \langle 1,\tau\rangle$ is an elliptic curve. The invariant $F_{1,1}^{\CC, s^{\sf BT}}$ is computed in~\cite{CT2}.  It turns out the only non-trivial invariant is computed as
\[ F_{1,1}^{\CC,s^{\sf BT}} (2\pi i [dz]; \frac{1}{\tau-\overline{\tau}}[d\overline{z}])= -\frac{1}{24} E_2^*(\tau).\]
In this formula, the weak Calabi-Yau structure is $\phi= 2\pi i [ dz] \in HH_1(\CC)\cong H^{1,0}(X_\tau)$, and the insertion is $\frac{1}{\tau-\overline{\tau}}[d\overline{z}] \in HH_{-1}(\CC)\cong H^{0,1}(X_\tau)$. The function $E_2^*(\tau)$ is the non-holomorphic, but modular Eisenstein series  $E_2^*(\tau) = 1 - 24\sum_{k=1}^\infty \frac{k q^k}{1-q^k} - \frac{6 i}{(\tau-\overline{\tau}) \pi}$ with $q= \exp(2\pi i \tau)$. 

In the case $d=2$, $X$ is a $K3$ surface. In view of mirror symmetry that the Gromov-Witten invariants are essentially trivial on a $K3$, we make the following
\begin{conj}
Let $\CC=D^b({\sf Coh}(X))$ with $X$ a smooth projective $K3$ surface. Then the CEI are the ``topological invariants" determined by the unital Frobenius algebra $\big(HH_\bullet(\CC)[2],\phi,\cup,\langle-,-\rangle_{\sf Muk}\big)$, see Equation~\eqref{eq:tqft}. More precisely, we have
\begin{itemize}
\item if $g=0$, $k_1+\cdots+k_n=n-3$, $|\alpha_1|+\cdots +|\alpha_n| = -4$, we have
\[ F_{0,n}^{\CC,s^{\sf BT}}(\phi;\alpha_1\cdot u^{k_1},\ldots,\alpha_n\cdot u^{k_n}) =\omega_{0,n}^{\CC,\phi}(\alpha_1,\ldots,\alpha_n)\cdot \frac{(n-3)!}{k_1!\cdots k_n!}.\]
\item if $g=1$, $k_1+\cdots+k_n=n$, $|\alpha_1|+\cdots +|\alpha_n| = 0$, we have
\[ F_{1,n}^{\CC,s^{\sf BT}}(\phi;\alpha_1\cdot u^{k_1},\ldots,\alpha_n\cdot u^{k_n}) = 
\omega_{1,n}^{\CC,\phi}(\alpha_1,\ldots,\alpha_n)\cdot \int_{[\overline{M}_{1,n}]} \psi_1^{k_1}\cdots\psi_n^{k_n}\]
\item if $g\geq 2$, we have $F_{g,n}^{\CC,s^{\sf BT}}=0$.
\end{itemize}
Here the contribution $\omega_{0,n}^{\CC,\phi}(\alpha_1,\ldots,\alpha_n)$ and $\omega_{1,n}^{\CC,\phi}(\alpha_1,\ldots,\alpha_n)$ are topological invariants computed using the Frobenius algebra $HH_\bullet(\CC)[2]$.
\end{conj}

The case $d=3$ is arguably the most interesting case due to its physics interpretation as the partition functions in topological string theory, which was already pointed out by Costello~\cite{Cos2}. In general, the CEI $F_{g,n}^{\CC,\phi,s^{\sf BT}}$ are maps, with Hochschild classes as inputs. A remarkable property in dimension $d=3$ is that we may obtain actual numbers. Indeed, we may define a complex number
\[ F_g^{\CC,\phi,s^{\sf BT}} :=
\frac{1}{2g-2} F_{g,1}^{\CC,\phi,s^{\sf BT}} (\phi\cdot u), \;\; \mbox{ if $g\geq 2$.}\]
The factor $\frac{1}{2g-2}$ is from the Dilaton equation in Gromov-Witten theory. It was shown in~\cite{CT} that if we scale $\phi$ to $\lambda\cdot\phi$, the invariants $F_{g,n}^{\CC,\phi,s^{\sf BT}}$ scale by $\lambda^{2-2g-n}$. Thus using the above definition we see that $F_{g}^{\CC,\phi,s^{\sf BT}}$ scales as $\lambda^{2-2g-1}\cdot \lambda = \lambda^{2-2g}$. To this end, let us consider the commutative ring 
\[ R:=\mathbb{C}[T_2,T_4,T_6\cdots]\]
freely generated by variables $T_2,T_4,T_6,\ldots$ of weights $\wt(T_i)=i$. Combining with a famous theorem~\cite{Bri} of Bridgeland that in complex dimension $3$ birational equivalence implies derived equivalence, our main Theorem~\ref{thm:main} implies the following.

\begin{cor}
Let $f(T)=\frac{P(T)}{Q(T)}$ be a weight zero element in the field of fractions of $R$. Then the evaluation
\[ f(T)|_{T_i= F_{\frac{2+i}{2}}^{\CC,\phi,s^{BT}}}\]
is a birational invariant of $X$ whenever it's well-defined. For example, we may take $f$ to be fractions such as $\frac{T_2^2}{T_4}$, $\frac{T_6^2}{T_2^6+T_4^3}$, and so on. 
\end{cor}

It is a remarkable conjecture~\cite{Cos2} that the derived invariants $F_g^{\CC,\phi,s^{BT}}$ constructed above should agree with topological string partition function originally studied by physicists~\cite{BCOV,COGP} in string theory. For example, in genus one, it is an open question to compare the CEI $F_{1,1}^{\CC,\phi,s^{\sf BT}}$ with the BCOV torsion invariants~\cite{BCOV,FLY}.  The discussion on its birational invariance  also echoes well with related works of Hu-Li-Ruan~\cite{HLR} in symplectic geometry and Maillot-R\"ossler~\cite{MaiRos} in genus one. Finally, we also refer to the upcoming work~\cite{ST} for some recent progresses on proving these B-model CEI satisfy the holomorphic anomaly equations from~\cite{BCOV}.

\appendix

\section{Sign diagrams in TCFT's}~\label{sec:sign}

\subsection{Signs in Equation~\eqref{eq:chain-map}} We shall deal with the case $l=0$, i.e. the number of outputs is zero. In this case, the right hand side of Equation~\eqref{eq:chain-map} is $(-1)^{|G|} \cdot \rho^\CC(G) \mathcal{L}_\m$. The composition $ \rho^\CC(G) \mathcal{L}_\m$ is illustrated in the following figure
\[\includegraphics[scale=.5]{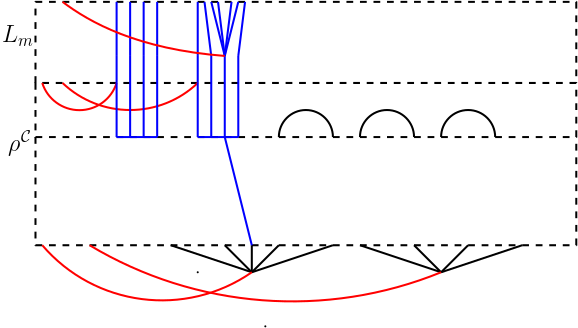}\]
Pulling the top operation $\mathcal{L}_\m$ down yields the following diagram
\[\includegraphics[scale=.4]{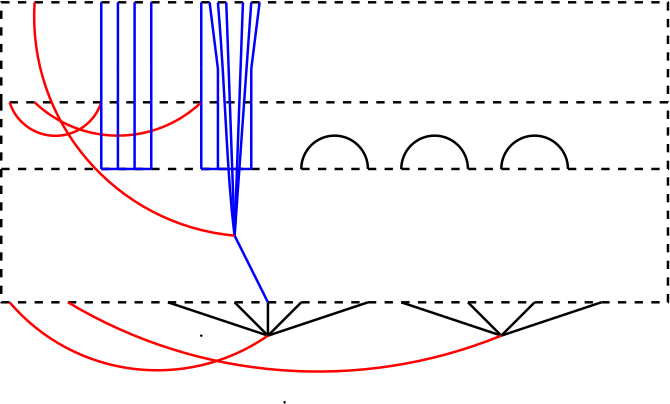}\]
The extra intersections of the red lines in these two diagrams is exactly equal to $k$, the number of inputs. Algebraically, this is because 
\[ \underbrace{(s\otimes \cdots \otimes s)}_{k \mbox{\; copies\;}}\circ \mathcal{L}_\m = (-1)^k \mathcal{L}_\m'\circ  \underbrace{(s\otimes \cdots \otimes s)}_{k \mbox{\; copies\;}},\]
where $\mathcal{L}_\m$ is the Hochschild differential of $CC_\bullet(\CC)^{\otimes k} $, while $\mathcal{L}_m'$ is the Hochschild differential of $CC_\bullet(\CC)[1]^{\otimes k}$.
Similarly, the left hand side of Equation~\eqref{eq:chain-map} $\rho^\CC(\partial G)$ is illustrated by the figure:
\[\includegraphics[scale=.6]{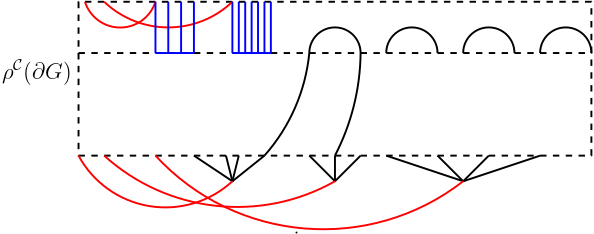}\]
We may move the first black vertex (from left to right) up to obtain the following figure:
\[\includegraphics[scale=.6]{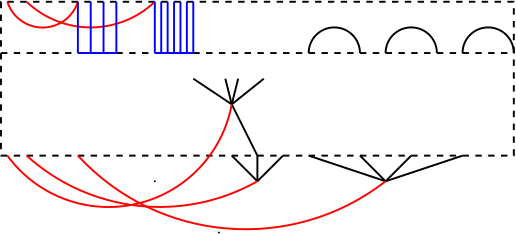}\]
The number of crossings of the red lines in the bottom is exactly equal to $|V_G^{\sf black}|$, the number of black vertices. To this point, we observe that the degree of a ribbon graph $G$ (in the case $l=0$) is equal to
\begin{equation}~\label{eq:graph-sign}
 |G|\equiv \sum_{v\in V_G^{\sf black}} \big({\sf val}(v) -3\big) \equiv \sum_{v\in V_G^{\sf black}} {\sf val}(v)+ |V_G^{\sf black}|\equiv k + |V_G^{\sf black}| \pmod{2}.
 \end{equation}
Putting the above together, and using the $A_\infty$ relation yields the desired identity
\[ \rho^{\CC}(\partial G) +(-1)^{|G|} \rho^{\CC}(G) \mathcal{L}_\m =0\]
\subsection{Signs in Equation~\eqref{eq:25}} We first compare the signs of the term $(i)$ with the first type term in $[d_{DR},\rho^{\CC\otimes \Omega^\bullet}_{[0]}(G)]$ of Equation~\eqref{eq:25}. This is illustrated in the following figure:
\[\includegraphics[scale=.6]{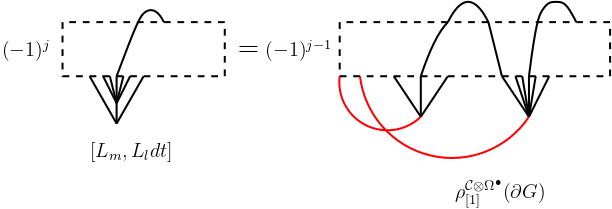}\]
Here the sign $(-1)^{j-1}$ appears from the definition of $\partial G$, while the extra $(-1)$ sign comes from the intersection of the two red lines.

For the $(ii)$ term, the corresponding term on the right hand side of Equation~\eqref{eq:25} is given by a composition of the form:
\[(-1)^j (\widetilde{\m}_{v_1}\otimes \cdots \otimes \widetilde{\m}\circ \widetilde{\fl}dt \otimes \cdots \otimes \widetilde{\m}_{v_{n}}) \circ \sigma_G \circ  \big(\underbrace{s\otimes \cdots \otimes s}_{k \mbox{\; copies\;}} \otimes \underbrace{\langle-,-\rangle^{-1}\otimes \cdots \otimes \langle-,-\rangle^{-1}}_{|E_G| \mbox{\; copies\;}}\big)\]
To match with the $(ii)$ term on the left hand side, we need to move the operator $\widetilde{\fl}dt$ to the rightmost part. This move yields a sign $(-1)^{n-j}\cdot (-1)^k$. Together with the sign $(-1)^j$, we obtain $(-1)^{n+k}$ which is equal to $(-1)^{|G|}$ by Equation~\eqref{eq:graph-sign}. The other term $(iii)$ is similar. 

\section{Boson-Fermion correspondence}\label{app:b-f}

In this section, we prove a type of Boson-Fermion correspondence for CEI. 
Let ${\sf Cl}$ be the Clifford algebra introduced in Example \ref{ex:non-existence}.
It is a $\Z/2\Z$-graded, cyclic, and unital $A_\infty$-algebra over $\mathbb{K}$. Moreover $\Cl$ is smooth and satisfies the Hodge-to-de-Rham degeneration property. Indeed, recall from~\cite[Section 2]{CLT} that its Hochschild homology is $1$-dimensional generated by $[\epsilon]$.  Furthermore, it has an unique homogeneous splitting of its non-commutative Hodge filtration explicitly given by
\[ \epsilon\mapsto \sum_{j\geq 0} (-1)^j (2j-1)!! \cdot  \epsilon\underbrace{[\epsilon|\cdots|\epsilon]}_{\mbox{$2j$-copies}} u^j.\]

Let $\CC$ be a $\Z/2\Z$-graded, cyclic, and unital $A_\infty$-category over $\mathbb{K}$. Assume also that $\CC$ is smooth, and satisfies the Hodge-to-de-Rham degeneration property. We may form another such $A_\infty$-category $\CC\otimes \Cl$. It has the same objects as $\CC$, and the morphism spaces are given by
\[ \CC\otimes \Cl ( X, X') := \CC (X, X') \otimes \Cl.\]
The $A_\infty$ operations are given by the formula
\begin{align}\label{eq:clubsuit}
\begin{split}
 \m_n(x_1\otimes \epsilon^{k_1},\ldots,x_n\otimes \epsilon^{k_n}) &= (-1)^\clubsuit \m_n(x_1,\ldots,x_n)\otimes \epsilon^{k_1+\cdots k_n},\\
 \clubsuit &= \sum_{j=2}^n (k_1+\cdots+k_{j-1}) |x_j|'.
 \end{split}
 \end{align}
 Here, in $\Cl$, we use the associative product $c_1\cdot c_2:=(-1)^{|c_1|}\m_2(c_1,c_2)$. The sign $(-1)^\clubsuit$ is given by the Koszul sign of moving the $\epsilon$'s to the right side of the $x$'s. This is a special case of the tensor product defined in \cite{A}.

We endow $\CC\otimes \Cl$ with the tensor product inner product:
\begin{equation}\label{eq:tensorpairing}
    \langle x\otimes \epsilon^k, y\otimes \epsilon^l\rangle:= (-1)^{k|y|'}\langle x, y\rangle \langle \epsilon^k,\epsilon^l\rangle^{ns},
\end{equation} 
where we use the non-shifted pairing in $\Cl$, precisely $\langle c_1, c_2\rangle^{ns}=(-1)^{|c_1|}\langle c_1, c_2\rangle$. The non-shifted pairing is graded symmetric and satisfies $\langle c_1\cdot c_2, c_3\rangle^{ns}= \langle c_1, c_2\cdot c_3\rangle^{ns}$.
The pairing in (\ref{eq:tensorpairing}) makes $\CC\otimes \Cl$ a cyclic \Ai-category. An important observation is that the parity of this pairing is the opposite from that of $\CC$. Moreover, this construction is functorial \cite{A}, meaning if $f: \CC \to \CC'$ is a (cyclic/unital) \Ai-functor there is an induced (cyclic/unital) \Ai-functor $F\otimes \id : \CC\otimes \Cl \to \CC'\otimes \Cl$.

The goal of this appendix is to prove that the CEI of $\CC$ and $\CC\otimes \Cl$ are the same. This is formulated more precisely in Theorem~\ref{thm:clifford} below.

\subsection{Hochschild invariants}~\label{subsec:hoch} 
We start by comparing the Hochschild invariants of $\CC$ and $\CC\otimes\Cl$ - this is a special case of \cite{Azu}. There is a shuffle product map (see~\cite[Section 4.2]{Lod} or \cite{Azu}) $ \sh : L^\CC \otimes L^{\sf Cl} \ra L^{\CC\otimes {\sf Cl}}$. It is defined as 
\[ \sh(a_0|a_1|\cdots|a_n,b_0|b_1|\cdots|b_m)= (-1)^\clubsuit a_0\otimes b_0 | \sh(a_1\otimes 1,\ldots,a_n\otimes 1,\one\otimes b_1,\ldots,\one\otimes b_m), \]
where $\clubsuit:=\sum_{p>q}|a_p|'|b_q|$ denotes the Koszul sign, and $\sh$ denotes the summation over $(n,m)$-type shuffles. Note that given a shuffle, the $\one$ in $\one\otimes b_j$ is labeled by the object which is the source of the next appearing morphism $a_j$.
Since $\epsilon\in L^{\sf Cl}$ is an even closed element, it determines a chain map denoted by
\[ i_0: L^\CC \ra L^{\CC\otimes {\sf Cl}}, \;\;\; i_0 := \sh(-,\epsilon).\]
\begin{lem}
The map $i_0: L^\CC \ra L^{\CC\otimes \Cl}$ is a quasi-isomorphism.
\end{lem}

\begin{proof}
Consider the length filtration of Hochschild chains on both sides. Obviously, the map $i_0$ respects the filtration. Hence, it induces maps on the associated spectral sequences. In the first page, we obtain the homology of $\m_1$. To this point, it suffices to observe that after taking $\m_1$, the product $\m_2$ becomes associative, which implies that the induced map of $i_0$ on the first page is already an isomorphism by the K\"unneth formula of Hochschild homology in the case of associative algebras~\cite[Theorem 4.2.5]{Lod}. This finishes the proof.
\end{proof}
It is easy to see that the map $i_0$ is not compatible with the circle operators. However, there is a cyclic extension of $\sh$ using cyclic shuffle product map  (see~\cite[Section 4.3]{Lod} or \cite{Azu}):
\[ \sh + u {\sf Sh} : L^\CC((u)) \otimes L^{\sf Cl}((u)) \ra L^{\CC\otimes {\sf Cl}}((u)).\]
Here ${\sf Sh}$ denotes the so-called cyclic shuffle map, which is of the form
\[ {\sf Sh}(a_0|a_1|\cdots|a_n,b_0|b_1|\cdots|b_m)= (-1)^\star \one\otimes 1 | {\sf Sh}(a_0\otimes 1,\ldots,a_n\otimes 1,\one\otimes b_0,\ldots,\one\otimes b_m), \]
where ${\sf Sh}$ is a sum of some particular permutations. We refer to~\cite[Section 4.3]{Lod} or~\cite{Azu,Shk2} for more details of this map. However, its precise formula is not needed in what follows. Using the above cyclic extension, the element $$\widetilde{\epsilon}=\sum_{j\geq 0} (-1)^j (2j-1)!! \cdot  \epsilon\underbrace{[\epsilon|\cdots|\epsilon]}_{\mbox{$2j$-copies}} u^j$$ induces a cyclic extension of $i_0$ which we denote by
\[ i: L^\CC((u)) \ra L^{\CC\otimes {\sf Cl}}((u)), \;\;\; i := \sh(-,\widetilde{\epsilon})+ u {\sf Sh}(-,\widetilde{\epsilon}).\]
Since $\widetilde{\epsilon}$ has no negative powers in $u$, the map $i$ preserves the positive sub-complexes, and hence it also induces a map on the quotient complexes. We shall still denote these induced maps by 
\begin{align*}
i:&  L^\CC_+ \to L^{\CC\otimes {\sf Cl}}_+,\\
i:& L^\CC_- \to L^{\CC\otimes {\sf Cl}}_-.
\end{align*}
\begin{lem}
The map $i: L^\CC((u)) \to L^{\CC\otimes {\sf Cl}}((u))$ is a quasi-isomorphism. The same holds for its induced maps on the positive and negative subspaces.
\end{lem}

\begin{proof}
We may consider the $u$-filtration on both sides. The map $i$ preserves this filtration, and induces an isomorphism on the first page of the associated spectral sequences by the previous lemma.
\end{proof}

We shall also need a backward chain map $p: L^{\CC\otimes \Cl} \to L^\CC$ defined explicitly by
\[ p\big(x_0\otimes \epsilon^{k_0} | x_1\otimes \epsilon^{k_1} | \cdots | x_n\otimes \epsilon^{k_n} \big):= (-1)^\clubsuit\pi ( \epsilon^{k_0+\cdots+k_n-1} )\cdot  x_0 | x_1 |\cdots | x_n,\]
where the sign $(-1)^\clubsuit$ is as in (\ref{eq:clubsuit}) and $\pi: \Cl\to \mathbb{K}$ is the projection to scalars (unit component) in $\Cl$.
\begin{lem}
The map $p$ defined above is a chain map, $pb=bp$. Furthermore, it is also compatible with the circle actions, that is $pB=Bp$. Thus, it induces a chain map still denoted by
\[ p: L^{\CC\otimes \Cl}((u)) \to L^\CC((u)).\]
Then, we also have $p\circ i = \id$. In particular, since $i$ is a quasi-isomorphism, so is $p$.
\end{lem}

\begin{proof}
Checking the two equations $pb=bp$ and $pB=Bp$ is a straight-forward calculation. For the identity $p\circ i=\id$, we first observe that $p \circ i_0 = \id$. Thus it suffices to prove that $p \circ (i-i_0)=0$.  Indeed, using the explicit formula of the shuffle product and cyclic shuffle product, the terms appearing in $(i-i_0)\big( x_0|x_1|\cdots|x_n )$ are all of the form:
\[ x_0\otimes \epsilon | \cdots | \one_{X_j}\otimes \epsilon |\cdots  \mbox{\;\;\; or \;\;} \one_{X_0} \otimes \one |\cdots |\one_{X_j}\otimes \epsilon | \cdots .\]
In either case, applying the map $p$ yields zero since we are using the reduced Hochschild chain complex of $\CC$.
\end{proof}

\subsection{CEI of tensor products with the Clifford algebra}

Recall from Subsection~\ref{subsec:tcft-dgla} the TCFT action of ribbon graphs associated with the cyclic $A_\infty$-category $\CC$ is given by a multi-linear map
\[ \rho^\CC_{g,k,l} (G): (L^\CC)^{\otimes k} \to (L^\CC)^{\otimes l},\]
for each ribbon graph $G$ of genus $g$ and with $k$ inputs and $l$ outputs.

\begin{lem}\label{lem:5-3}
As maps $(L^{\CC\otimes \Cl})^{\otimes k} \to (L^\CC)^{\otimes l}$, we have
\[ p^{\otimes l} \circ \rho^{\CC\otimes \Cl}_{g,k,l} (G) = \begin{cases}
    \rho^{\CC}_{g,k,l} (G) \circ p^{\otimes k}, \mbox{\;\;\;\;\;\; if $\CC$ is even,}\\
    (-1)^{1-g+k} \rho^{\CC}_{g,k,l} (G) \circ p^{\otimes k},  \mbox{\;\; if $\CC$ is odd.}
\end{cases} \]
\end{lem}

\begin{proof}
The evaluation of $\rho^{\CC\otimes \Cl}_{g,k,l} (G)$  is obtained in two steps: one from evaluating Hochschild chains in $\CC$ which gives $\rho^{\CC}_{g,k,l} (G)$; the other one from evaluating in $\Cl$ on a trivalent graph obtained by replacing all black vertices of $G$ by a binary tree. Then we observe that in the second step the result of evaluation is first of all independent of the choice of the binary trees by associativity of the product in $\Cl$, and furthermore it yields the desired equation. Most of the complications are from getting the signs to work out. For this reason, we shall use concrete examples and sign diagrams to illustrate how the signs appear in the proof. Also, in the following, we shall assume $\CC$ is an odd Calabi-Yau category since this is the case that we are mainly interested in.

{\bf Case 1.} Let us consider $G$ is the Mukai graph as in~\eqref{eq:Mukai-graph} with two insertions of the form $x\otimes \epsilon^i$ and $y\otimes \epsilon^j$. Then we may compute $\rho^{\CC\otimes\Cl}(G)$ according to the following sign diagram:
\[\includegraphics[scale=.5]{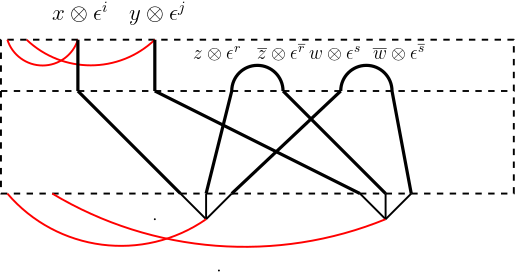}\]
Using the definition of multiplication and inner product of $\CC\otimes \Cl$, the evaluation of the bottom line in the above diagram is illustrated as follows.
\[\includegraphics[scale=.5]{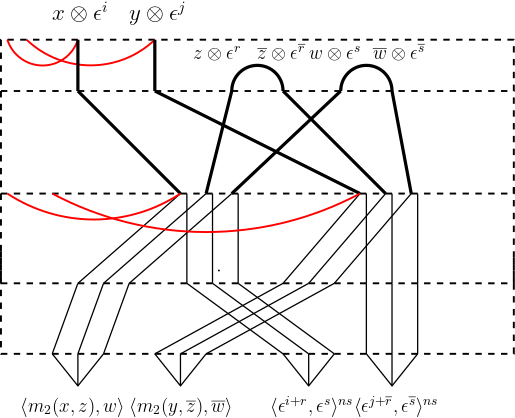}\]
Thus, the total sign of the evaluation  $\rho^{\CC\otimes \Cl} (G)(x\otimes \epsilon^i, y\otimes \epsilon^j)$ is given by
\[ (-1)^{|x|+i}\cdot (-1) \cdot (\mbox{Koszul sign from the total permutation in the above diagram}).\]
To compare with the evaluation $\rho^\CC(G)(x,y)$, we may first move the $\epsilon$'s to the right and then evaluate. The corresponding sign diagram is given by
\[\includegraphics[scale=.6]{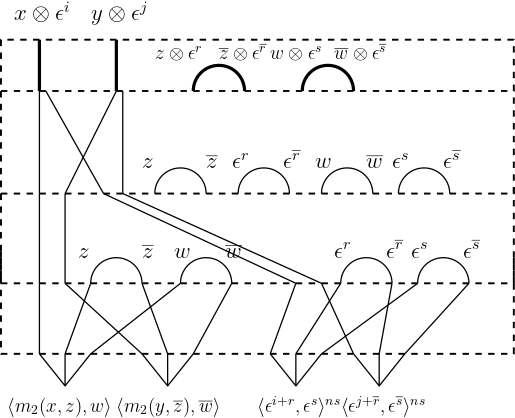} \]
Comparing the two evaluations yields
\[ \rho^{\CC\otimes \Cl}(G)(x\otimes \epsilon^i,y\otimes \epsilon^j) = (-1)^{|x|+i}(-1)^{|y|'i} \rho^\CC(G)(x,y)\cdot \rho^{\Cl,ns}(G)(\epsilon^i,\epsilon^j).\]
Finally, the non-shifted evaluation and the usual shifted evaluation are related by conjugation by the tensor products of the shift operator $\Cl \rightarrow \Cl[1]$. Analyzing this yields 
\[\rho^{\Cl,ns}(G)(\epsilon^i,\epsilon^j)= - A_G \cdot  \rho^{\Cl}(G)(\epsilon^i,\epsilon^j),\]
where $A_G$ is the sign difference of the orientation of $G$ from the canonical orientation of $G$ being a trivalent graph. But for the Clifford algebra $\Cl$ we have
\[ A_G \cdot \rho^{\Cl}(G)(\epsilon^i,\epsilon^j) = \begin{cases}
    1 & \mbox{if\;\;} i=j=1,\\
    0 & \mbox{otherwise.}
\end{cases}\]
Note that in the non-vanishing case when $i=j=1$, we also have $(-1)^{|x|+i}(-1)^{|y|'i}=(-1)^{|x|'+|y|'}$. But in order for $\rho^\CC(G)(x,y)\neq 0$, the degree sum must be $|x|'+|y|'= 0 \pmod{2}$. Hence we deduced that
\[  \rho^{\CC\otimes \Cl}(G)(x\otimes \epsilon^i,y\otimes \epsilon^j) = -\rho^\CC(G)(x,y).\]
The minus sign is indeed equal to $(-1)^{1-g+k}$ as for the Mukai graph $g=0$ and $k=2$.

{\bf Case 2.} Repeating the above sign analysis with a general trivalent graph $G$ of type $(g,k,0)$ and insertions $x_1\otimes \epsilon^{i_1},\ldots,x_k\otimes \epsilon^{i_k}$ yields
\begin{align*}
    \rho^{\CC\otimes \Cl} & (x_1\otimes \epsilon^{i_1},\ldots,x_k\otimes \epsilon^{i_k})= \begin{cases}
        (-1)^\star  \rho^{\CC}(G)(x_1,\ldots,x_k), & \mbox{if\;\;} i_1=\cdots = i_k=1,\\
    0 & \mbox{otherwise.}
    \end{cases},\\
    \star & = (k-1)\sum_{i=1}^k |a_i|'+ \frac{V(V-1)}{2}+\frac{E(E-1)}{2}+\frac{k(k-1)}{2}+kE.
\end{align*}
Here $V$ and $E$ denote the number of vertices and edges respectively. Since $G$ is a trivalent graph, we have $3V=2E+k$. Together with the Euler number equality $V-E=2-2g-k$ we may deduce that 
\[E=2-2g \pmod{4}, \mbox{ \;\;and\;\;}  V+k=0 \pmod{4}.\] 
Furthermore, the evaluation $\rho^{\CC}(G)(x_1,\ldots,x_k)$ is only non-zero with $\sum_{i=1}^k |a_i|'= 0 \pmod{2}$, which reduces star to $\frac{V(V-1)}{2}+\frac{E(E-1)}{2}+\frac{k(k-1)}{2}$. Thus the sign $(-1)^\star$ is indeed equal to
\[ (-1)^{\frac{E(E-1)}{2}}\cdot (-1)^{\frac{(V+k)(V+k-1)}{2}}\cdot (-1)^{Vk}= (-1)^{1-g+k}.\]

{\bf Case 3.} When the insertions are general elements in $L^{\CC\otimes\Cl}$ and $G$ is not necessarily trivalent, one can still repeat the above sign analysis by fixing a combinatorial configuration of how the insertions are distributed along vertices of $G$. The only difference with the previous case is the part that involves the evaluation of $\rho^\Cl$ where we need to replace each vertex of $G$ by a trivalent tree. Denote the resulting trivalent graph by $\hat{G}$. Using the associativity, we may move around the leaves in a cycle of $\hat{G}$ all to the cycle marking as shown in the following picture. 
\[\begin{tikzpicture}[baseline={([yshift=-0.5ex]current bounding
      box.center)},scale=0.5] 
\draw [thick] (-10,2) to (-10,-2);
\draw [thick] (-9,1) to (-10,2);
\draw [thick] (-8.2, -.2) to (-7.8, .2);
\draw [thick] (-8.2, .2) to (-7.8, -.2);
\draw [thick] (-8,0) to (-10,0);
\draw [thick] (-10,2) to (-7,3);
\draw [thick] (-7,3) to (-4,0);
\draw [thick] (-5.5,0) to (-4,0);
\draw [thick] (-7,-1.5) to (-7,-3);
\draw [thick] (-4,0) to (-7,-3);
\draw [thick] (-7,-3) to (-10,-2);
\draw [thick] [|->] (-3,0) to (3,0);
\draw [thick] (5,2) to (5,-2);
\draw [thick] (7.8, -.2) to (8.2, .2);
\draw [thick] (7.8, .2) to (8.2, -.2);
\draw [thick] (6.3,0) to (5,0);
\draw [thick] (5,2) to (8,3);
\draw [thick] (8,3) to (11,0);
\draw [thick] (6,0) to (8,0);
\draw [thick] (5.5,0) to (6.5,0.5);
\draw [thick] (6,0) to (7,-0.5);
\draw [thick] (6.5,0) to (7.5,0.5);
\draw [thick] (11,0) to (8,-3);
\draw [thick] (8,-3) to (5,-2);
\end{tikzpicture}\]
Taking the signs into account, we obtain
\begin{align*}
    \rho^{\CC\otimes \Cl}(G) & (\alpha_1,\ldots,\alpha_k)= 
        (-1)^\star  \rho^{\CC}(G)(p(\alpha_1),\ldots,p(\alpha_k)), \\
    \star & = (k-1)\sum_{i=1}^k |\alpha_i|'+ \frac{V(V-1)}{2}+\frac{E(E-1)}{2}+\frac{k(k-1)}{2}+kE + (k+1)|G|
\end{align*}
where $G$ is the degree of $G$. Since the evaluation $\rho^{\CC}(G)(p(\alpha_1),\ldots,p(\alpha_k))$ is only non-zero with $\sum_{i=1}^k |\alpha_i|'=|G| \pmod{2}$, we see that $(k-1)\sum_{i=1}^k |\alpha_i|'+(k+1)|G|=0\pmod{2}$.
Together with the Euler number equality $V-E=2-2g-k$ we may deduce that
\[ (-1)^{\frac{V(V-1)}{2}+\frac{E(E-1)}{2}+\frac{k(k-1)}{2}}= (-1)^{1-g}(-1)^{Vk}.\]
Putting together yields $(-1)^\star= (-1)^{1-g}(-1)^{k(V+E)}=(-1)^{1-g+k}$, as desired.

{\bf Case 4.} In the last step, we deal with the case when $G$ possibly has also white vertices. By the definition of the map $p$, in the Clifford part, we multiple the $\epsilon$'s then followed by Mukai pairing with $\epsilon$. Graphically, this means we replace each white vertex of $G$ to a trivalent graph locally as follows.
\[\begin{tikzpicture}[baseline={([yshift=-0.5ex]current bounding
      box.center)},scale=0.4] 
      \draw [thick] (-3,0) circle [radius=0.2];
\draw [thick] (-6,-0) to (-3.2,0);
\draw [thick] (-5,2) to (-3.2,0);
\draw [thick] (-5,-2) to (-3.2,0);
\draw [thick] (-2.8,0) to (-.5,0);
\draw [line width=2.4pt] (-1.5,0) to (-2.8,0);
\draw [thick] [|->] (1.5,0) to (4,0);
\draw [thick] (10,2) circle [radius=1];
\draw [thick] (6,1) to (8,-1);
\draw [thick] (7,2) to (8.8,0);
\draw [thick] (7,-2) to (8.8,0);
\draw [thick] (10,0) to (10,1);
\draw [thick] (10,3) to (10,2.2);
\draw [thick] (8.8,0) to (12,0);
\end{tikzpicture}\]
where inside the circle we insert $\epsilon$. After this replacements, the evaluation reduces to the previous case. The proof is finished.
\end{proof}

Let $s: H_\bullet(L^\CC) \to H_\bullet(L^\CC_+)$ be a splitting of the non-commutative Hodge filtration.  To compare the CEI of $\CC$ and $\CC\otimes \Cl$, we first construct a splitting $s^{\otimes} : H_\bullet(L^{\CC\otimes\Cl}) \to H_\bullet(L^{\CC\otimes \Cl}_+)$ from $s$. Indeed it is defined by the following diagram.
\[\begin{CD}
    H_\bullet(L^\CC) @>s>> H_\bullet(L^\CC_+) \\
    @A p AA @VV i V\\
    H_\bullet(L^{\CC\otimes \Cl}) @>s^{\otimes}>> H_\bullet(L^{\CC\otimes \Cl}_+)
\end{CD}\]

\begin{lem}
    The map $s^{\otimes} : H_\bullet(L^{\CC\otimes\Cl}) \to H_\bullet(L^{\CC\otimes \Cl}_+)$ defined above is indeed a splitting of the nc-Hodge filtration of $\CC\otimes \Cl$.
\end{lem}

\begin{proof}
    It is clear that $s^\otimes$ splits the canonical map $H_\bullet(L_+^{\CC\otimes \Cl}) \to H_\bullet(L^{\CC\otimes \Cl})$ since $s$ is a splitting. It remains to verify the Lagrangian condition in Definition~\ref{defi:splitting}.
    By the previous lemma (applied to the Mukai graph), we obtain that 
\[ \langle p\alpha, p\beta\rangle_{\sf Muk} = (-1)^d\langle \alpha,\beta\rangle_{\sf Muk},\]
with $d$ equal to the Calabi-Yau dimension of $\CC$. Then we may verify the Lagrangian condition:
\begin{align*}
    \langle s^{\otimes}(\alpha),s^{\otimes}(\beta)\rangle_{\sf hres} &= \langle isp(\alpha),isp(\beta)\rangle_{\sf hres}\\
    &= (-1)^d\langle pisp(\alpha),pisp(\beta)\rangle_{\sf hres}\\
    &= (-1)^d\langle sp(\alpha),sp(\beta)\rangle_{\sf hres}\\
    &= (-1)^d\langle p(\alpha),p(\beta)\rangle_{\sf Muk}\\
    &= \langle \alpha,\beta\rangle_{\sf Muk}.
\end{align*}
\end{proof}

Let $S: L^\CC \to L^\CC_+$ be a chain-level splitting which lifts the map $s$. In order to apply the formula in Equation~\eqref{eq:cei}, we would like to extend $S$ to the tensor product chain complexes. Since $p\circ i=\id$, we obtain a direct sum decomposition
\[ L^{\CC\otimes \Cl}_+ \cong L^\CC_+ \oplus {\sf ker}(p)_+,\]
where ${\sf ker}(p)\subset L^{\CC\otimes \Cl}$ denotes the kernel of the map $p$. Note that since $pB=Bp$, the subspace ${\sf ker}(p)$ is $B$-invariant. The equivariant chain complex ${\sf ker}(p)_+$ is formed using the subspace circle action. Since ${\sf ker}(p)$ is acyclic, we may choose a chain-level splitting of it, say
\[ S': {\sf ker}(p) \to {\sf ker}(p)_+.\]
Define a chain-level splitting of $\CC\otimes \Cl$ by setting
\[ S^{\otimes} := \begin{bmatrix}
S & 0\\
0 & S'
\end{bmatrix}\]
It is easy to see that its induced splitting in homology is indeed $s^{\otimes}$. Note that by construction, we have
\[ p\circ S^{\otimes} = S \circ p\]
at the chain-level.
Using the chain level splitting $S^{\otimes}$ and following the constructions in Equation~\eqref{eq:cei}, we define operators $\Theta^\otimes$, $H^\otimes$, $F^\otimes$, and $\delta^{\otimes}$ for the category $\CC\otimes \Cl$, compatible with $p$.
\begin{lem}\label{lem:commutators}
Let $\CC$ be of Calabi-Yau dimension $d$. On the category $\CC\otimes \Cl$ the maps $\Theta^\otimes$, $H^\otimes$, $F^\otimes$, and $\delta^{\otimes}$ satisfy the following identities:
\begin{itemize}
\item[(a.)] $H^\sym \circ (p\otimes p) = (-1)^d (H^{\otimes})^\sym$.
\item[(b.)] $p \circ F^{\otimes} = F\circ p$.
\item[(c.)] $\delta \circ (p\otimes p) = (-1)^d \delta^{\otimes}$.
\item[(d.)] $ p \circ \Theta^{\otimes} = \Theta \circ p$.
\end{itemize}
\end{lem}

\begin{proof}
We begin with part $(a.)$. It suffices to prove $H\circ (p\otimes p) = H^{\otimes}$. For this, we compute
\begin{align*}
H^{\otimes} (\alpha,\beta) =&  - \langle S^\otimes \tau_{\geq 1} R^\otimes (\alpha), \beta \rangle_{\sf res}\\
=& - (-1)^d\langle p S^\otimes \tau_{\geq 1} R^\otimes (\alpha), p\beta \rangle_{\sf res}\\
=& - (-1)^d \langle S \tau_{\geq 1} R (p\alpha), p\beta \rangle_{\sf res}\\
=& (-1)^d H(p\alpha,p\beta)
\end{align*}
The sign $(-1)^d$ is from Lemma~\eqref{lem:5-3}. Part $(b.)$ is similar to $(a.)$, and will be omitted. 

For part $(c.)$, from $(a.)$ we see that $H^{\otimes}(\alpha,\beta)=0$ if $\alpha\in {\sf ker}(p)$ or $\beta \in {\sf ker}(p)$. Hence we may simply take 
\[ \delta^\otimes = (-1)^d\begin{bmatrix}
\delta & 0\\
0 & 0 
\end{bmatrix},\]
which implies $(c.)$ immediately. Part $(d.)$ is a straightforward calculation. Indeed, by definition $\Theta: L^\CC_- \ra L^\CC_+[1]$ is given by $\Theta( \alpha_0 + \alpha_1 u^{-1} +\cdots ) = B\alpha_0$. Hence we have
\[ p \big(\Theta^\otimes ( \alpha_0 + \alpha_1 u^{-1} +\cdots)\big) = pB\alpha_0=Bp\alpha_0 = \Theta \big(p(\alpha)\big).\]
\end{proof}

With these preparations, we are ready to prove the following

\begin{thm}\label{thm:clifford}
Let $\CC$ be of Calabi-Yau dimension $d$. Let $s$ be a splitting of $\CC$. Let $s^\otimes$ the induced splitting of $\CC\otimes {\sf Cl}$. Then we have
\[ \langle \sh(\alpha_1,\epsilon)  u^{k_1},\ldots,\sh(\alpha_n,\epsilon) u^{k_n}\rangle_{g,n}^{\CC\otimes {\sf Cl},s^\otimes}= \begin{cases}
    (-1)^{1-g-n}\langle \alpha_1 u^{k_1},\ldots,\alpha_n u^{k_n}\rangle_{g,n}^{\CC,s} \mbox{\;\; if $d$ is odd,}\\
    \langle \alpha_1 u^{k_1},\ldots,\alpha_n u^{k_n}\rangle_{g,n}^{\CC,s} \mbox{\;\; if $d$ is even.}
\end{cases}.\]
Here $\sh$ is the K\"unneth map given by shuffle product.
\end{thm}

\begin{proof}
The sign is more involved in the case when $\CC$ is odd. Hence we shall deal with this case in the proof. Since the map $\bar{\iota}$ from Equation~\eqref{eq:iota-bar} is an embedding in homology, the homology class $[\bar{\iota}(F_{g,n}^{\CC,s})]$ determines the CEI of $(\CC,s)$ completely. Thus, to prove this theorem, it suffices to show that we have
\[ p^{\otimes n-1} \circ \bar{\iota}(F_{g,n}^{\CC\otimes \Cl, s^{\otimes}}) \circ i = (-1)^g \cdot \bar{\iota} (F_{g,n}^{\CC,s}),\]
as both $p$ and $i$ are quasi-isomorphisms. Note that we got rid of a sign $(-1)^{n-1}$ due to applying $(n-1)$ Mukai-pairings when turning outputs on both sides to inputs.

To prove the identity above, we make use of Equation~\eqref{eq:cei} to write both sides as a sum over partially directed graphs. Since the rational number $\frac{\wt(\GG)}{\Aut(\GG)}$ is only depends on $\GG$, it suffices to prove that the contributions from each graph $\GG\in \Gamma((g,1,n-1))_m$ are equal, up to a sign $(-1)^g$. For this, the idea is to make use of Lemma~\ref{lem:5-3} to move the projection maps $p^{\otimes n-1}$ from the output legs all the way ``up" to the input leg in a given partially directed graph $\GG\in \Gamma((g,1,n-1))_m$. In the following comparison we shall refer to $p^{\otimes n-1} \circ \bar{\iota}(F_{g,n}^{\CC\otimes \Cl, s^{\otimes}}) \circ i $ as LHS, and $\bar{\iota}(F_{g,n}^{\CC,s})$ as RHS.

There is a partial ordering $>$ defined on the set of vertices $V_G$ of $\GG$: we set $w>v$ if there is a directed path from $w$ to $v$. Let $v$ be a minimal element in this partial order. Then an outgoing half-edge at the vertex $v$ can be exactly one of the following three cases:
\begin{itemize}
\item[(1)] an outgoing leg,
\item[(2)] part of an un-directed edge not in the spanning tree $T$,
\item[(3)] part of an un-directed edge  in the spanning tree $T$.
\end{itemize}
In the case $(1)$, at an output leg of $\GG$, in $\bar{\iota}(F_{g,n}^{\CC\otimes \Cl, s^{\otimes}})$ it is assigned the operator $R^{\otimes}$. By construction, we have
\[ p \circ R^{\otimes}  = R\circ p.\]
This shows that indeed, after moving $p$, the output leg contribution becomes $R$ which matches with the contribution from the RHS. In the case $(2)$, the un-directed edge is decorated by $(H^{\otimes})^\sym$. We use the identity $$H^\sym \circ p^{\otimes 2} = -(H^{\otimes})^\sym$$. While in the case $(3)$, the un-directed edge is by $\delta^{\otimes} $ for which we use the identity $$ \delta \circ p^{\otimes 2} = - \delta^{\otimes}.$$ After these replacements, every outgoing half-edge of $v$ has a projection operator $p$ adjacent to it. Thus, we can apply Lemma~\ref{lem:5-3} to move the $p$'s to the input half-edges of $v$, with a sign $(-1)^{1-g(v)+k(v)}$.

To continue moving the $p$'s ``up", let $w$ be a second-minimal element of $V_G$ in the partial order $>$. At such a vertex, two more cases can appear for an outgoing half-edge:
\begin{itemize}
\item[(4)] part of a directed edge not in the spanning tree $T$,
\item[(5)] part of a directed edge  in the spanning tree $T$.
\end{itemize}
In the case $(4)$, we move the operator $p$ using the identity
\[ p \circ \Theta^{\otimes} = \Theta \circ p.\]
While in the case $(5)$, we use the identity
\[ p \circ F^{\otimes} = F\circ p\]
After these replacements, every outgoing half-edge of $w$ has a projection operator $p$ adjacent to it. Using Lemma~\ref{lem:5-3} we may move the $p$'s to the input half-edges of $w$ as well. 

Continuing to perform the above replacements at vertices of $G$ according to the partial ordering $<$, in the end we arrive at the desired equality that
\[ p^{\otimes n-1} \circ \bar{\iota}(F_{g,n}^{\CC\otimes \Cl, s^{\otimes}}) \circ i = \pm\cdot \bar{\iota} (F_{g,n}^{\CC,s}).\]
In the above process the sign $\pm$ is equal to $\prod_v (-1)^{1-g(v)-k(v)} \cdot \prod_{e \in E^{un-dir}} (-1)$ with $E^{un-dir}$ the set of un-directed edges in $\GG$. But since $\GG$ is of type $(g,1,n-1)$, we have
\[ g= \sum_{v} (g(v)+k(v)-1) + |E^{un-dir}|,\]
which implies that $\prod_v (-1)^{1-g(v)-k(v)} \cdot \prod_{e \in E^{un-dir}} (-1) = (-1)^g$. The proof is complete.
\end{proof}

\vspace{1cm}

\textbf{Address:}

\vspace{.3cm}	
	
\noindent Lino Amorim: Department of Mathematics, Kansas State University, 138 Cardwell Hall, 1228 N. 17th Street, Manhattan, KS 66506, USA.

Email: {\tt lamorim@ksu.edu}\\

\noindent Junwu Tu: Institute of Mathematical Sciences, ShanghaiTech University. 393 Middle Huaxia Road, Pudong New District, Shanghai, China, 201210. 

Email: {\tt tujw@shanghaitech.edu.cn}
	
\end{document}